\documentclass[11pt, oneside]{article}   	
\usepackage{geometry}                		
\usepackage{graphicx}				
								
\usepackage[all]{xy}
\CompileMatrices

\usepackage{amssymb}
\usepackage{amsmath}
\usepackage{stmaryrd}
\usepackage{amsthm}
\usepackage{tikz-cd}
\usepackage{extarrows}
\usepackage[mathscr]{euscript}
\usepackage{mathrsfs} 
\usepackage{verbatim}
\usepackage[OT2,T1]{fontenc}
\DeclareSymbolFont{cyrletters}{OT2}{wncyr}{m}{n}
\DeclareMathSymbol{\Sha}{\mathalpha}{cyrletters}{"58}
\DeclareMathSymbol{\Che}{\mathalpha}{cyrletters}{"51}

\usepackage{calligra}
\usepackage{mathrsfs}
\newcommand{\calHom}{\mathscr{H}\mathit{om}}

\newcommand{\Ga}{{\mathbf{G}}_{\rm{a}}}
\newcommand{\Gm}{{\mathbf{G}}_{\rm{m}}}

\DeclareMathOperator{\Gal}{Gal}

\DeclareMathOperator{\Br}{Br}

\DeclareMathOperator{\Pic}{Pic}

\DeclareMathOperator{\Ext}{Ext}
\DeclareMathOperator{\Hom}{Hom}

\DeclareMathOperator{\R}{R}

\DeclareMathOperator{\Spec}{Spec}
\DeclareMathOperator{\coker}{coker}

\DeclareMathOperator{\et}{\acute{e}t}

\newcommand*{\Z}{\ensuremath{\mathbf{Z}}}                        
\newcommand*{\Q}{\ensuremath{\mathbf{Q}}}                     
\newcommand*{\F}{\ensuremath{\mathbf{F}}}                        
\newcommand*{\C}{\ensuremath{\mathbf{C}}}                        
\newcommand*{\A}{\ensuremath{\mathbf{A}}}                        
\newcommand*{\calO}{\mathcal{O}}                                  

\usepackage{rotating}
\newcommand*{\isoarrow}[1]{\arrow[#1,"\rotatebox{90}{\(\sim\)}"
]}
\usepackage{bm}

\numberwithin{equation}{section}

\newtheorem{theorem}{Theorem}[section]

\newtheorem{lemma}[theorem]{Lemma}
\newtheorem{proposition}[theorem]{Proposition}

\theoremstyle{definition}
  \newtheorem{definition}[theorem]{Definition}

\theoremstyle{remark}
  \newtheorem{remark}[theorem]{Remark}

\usepackage[OT2,T1]{fontenc}

\tikzset{commutative diagrams/.cd,
mysymbol/.style={start anchor=center,end anchor=center,draw=none}
}

\usepackage{stmaryrd}
\usepackage{hyperref}

\title{\textbf{TAMAGAWA NUMBERS AND OTHER INVARIANTS OF PSEUDO-REDUCTIVE GROUPS OVER GLOBAL FUNCTION FIELDS}}
\author{Zev Rosengarten \thanks{While completing this work, the author was supported by an ARCS Scholar Award and by a Ric Weiland Graduate Fellowship. \newline
}}

\date{}
\begin{document}
\maketitle

\begin{abstract}
We study Tamagawa numbers and other invariants (especially Tate-Shafarevich sets) attached to commutative and pseudo-reductive groups over global function fields. In particular, we prove a simple formula for Tamagawa numbers of commutative groups and pseudo-reductive groups. We also show that the Tamagawa numbers and Tate-Shafarevich sets of such groups are invariant under inner twist, as well as proving a result on the cohomology of such groups which extends part of classical Tate duality from commutative groups to all pseudo-reductive groups. Finally, we apply this last result to show that for suitable quotient spaces by commutative or pseudo-reductive groups, the Brauer--Manin obstruction is the only obstruction to strong (and weak) approximation.
\end{abstract}

\tableofcontents{}

\section{Introduction}

\subsection{Tamagawa numbers}

Let $G$ be a linear algebraic group (i.e., a smooth affine group scheme) over a global field $k$, and assume it is connected. For the adele ring $\A := \A_k$ of $k$, the group $G(\A)$ is a locally compact group, hence admits a (left or right) Haar measure, unique up to scaling. Tamagawa observed that in fact one can pin down a certain canonical Haar measure on this group, the {\em Tamagawa measure}. 

Informally, the Tamagawa measure is obtained as follows. Consider the set of global left (resp. right) invariant differential $d$-forms on $G$, where $d$ is the dimension of $G$. By translation-invariance, such forms are in one-to-one correspondence with elements of $\bigwedge^d(T_eG)^*$, the $d$th exterior power of the cotangent space of $G$ at the identity. Thus, they form a one-dimensional $k$-vector space. Given any one such nonzero form $\omega$, once we fix a Haar measure on each $k_v$ (we choose the one giving $\calO_v$ measure $1$ in the non-archimedean case, $dx$ on $\mathbf{R}$ giving $\mathbf{R}/\Z$ measure $1$, and $|dz \wedge d\overline{z}|$ on $\C$ giving $\C/\Z[i]$ measure $2$ as in \cite[Ch.\,I, \S 4.3]{oesterle}), we obtain left (resp. right) Haar measures $\mu_v$ on $G(k_v)$ for each place $v$ of $k$. Then we may define the ``restricted product'' measure 
\begin{equation}
\label{productmeasure}
\prod_v \mu_v
\end{equation}
on the group $G(\A)$ if this converges in a suitable sense. The key observation here is that this measure is independent of the choice of $\omega$! Indeed, replacing $\omega$ by $\lambda \omega$ for some $\lambda \in k^{\times}$ multiplies this measure by $\prod_v |\lambda|_v$, which equals $1$ by the product formula. We thus obtain a canonical measure on $G(\A)$.

The actual definition of Tamagawa measure in general involves some additional technical complications. For one, the measure (\ref{productmeasure}) does not converge if $G$ admits nontrivial characters (that is, 
the group $\widehat{G}(k)$ of  $k$-homomorphisms $G \rightarrow \Gm$ is nontrivial). For this reason, one must insert certain ``convergence factors'' (and in the actual definition of Tamagawa measure, one inserts specific convergence factors to ensure good behavior under Weil restriction of scalars). One may try to use this measure on $G(\A)$ to define a quotient measure on the coset space $G(k) \backslash G(\A)$. 

If $G$ admits nontrivial characters then the volume of this coset space is automatically infinite, since such a character $\chi$ yields a 
nontrivial composite map
\[
|\!|\chi|\!|:  G(k) \backslash G(\A) \xrightarrow{\chi} k^{\times} \backslash \A^{\times} \xrightarrow{|\!|\cdot|\!|} \mathbf{R}_{>0}
\]
($|\!|\cdot|\!|$ is the norm map sending an idele $(a_v)$ to $\prod_v |a_v|$)
whose image is full or an infinite cyclic group.  We therefore consider the subset 
\[
G(\A)_1 := \underset{\chi \in \widehat{G}(k)}{\bigcap} \ker(|\!|\chi|\!|)
\]
of norm-$1$ adelic points of $G$. This inherits a canonical measure induced from that on $G(\A)$ because $G(\A)/G(\A)_1$ has a canonical Haar measure as a lattice (for function fields) or Euclidean space (for number fields), and we define the {\em Tamagawa number} of $G$, denoted $\tau(G)$, to be the volume of $G(k) \backslash G(\A)_1$.  (For more details on Tamagawa measure, see \cite[Chap. I]{oesterle}.) The choice between using a left-invariant versus right-invariant differential form makes no difference, since the ratio between a nonzero left-invariant form and a nonzero right-invariant form is an algebraic character (the determinant of the adjoint representation), which in turn disappears when we restrict attention to the norm-$1$ adelic points. 

It is by no means apparent that $\tau(G)$ is finite. The finiteness of $\tau(G)$ for reductive groups is due to Borel (building on work of Harish-Chandra) \cite[Thm.\,5.8]{borel} over number fields, using ideas from classical reduction theory and class field theory,
and is due to Harder over function fields \cite{hardertau}. The general finiteness of $\tau(G)$ (for arbitrary connected linear algebraic groups) is an easy reduction to the reductive case in the number field setting, whereas over function fields it was established by Conrad \cite[Thm.\,1.3.6]{conrad} who made essential use of the classification of pseudo-reductive groups in \cite{cgp} and Oesterl\'e's finiteness result in the solvable case \cite[Ch.\,IV, \S 1.3]{oesterle}.

The Tamagawa number of $G$ contains very important arithmetic information. For example, if $q$ is a positive-definite quadratic form over $\Z$, then the Siegel Mass Formula is equivalent (via a  simple argument) to the statement that $\tau({\rm{SO}}(q)) = 2$ \cite[\S 1.1]{lurie}.
As another example, the equality $\tau({\rm{SL}}_n) = 1$ for the group ${\rm{SL}}_n$ over $\Q$ is equivalent to the classical volume computation
\[
{\rm{vol}}({\rm{SL}}_n(\mathbf{R})/{\rm{SL}}_n(\Z)) = \prod_{i = 2}^n \zeta(i).
\]

It is therefore of significant interest to compute Tamagawa numbers of linear algebraic groups. In the late 1950's and early 1960's, Weil \cite{weiladeles} computed the Tamagawa numbers of various classical groups over all global fields, including SL$_n$, SO$(q)$, and Sp$_{2n}$, and formulated his famous conjecture that all simply connected groups over global fields have Tamagawa number $1$ (see the end of \cite{weiladelesetgroupes}). This conjecture was proved over number fields by Langlands, Lai, and Kottwitz (\cite{langlands}, \cite{lai}, \cite{kottwitz}), in conjunction with work of others, and over function fields by Gaitsgory and Lurie \cite{lurie}.

In 1981, Sansuc \cite[Thm.\,10.1]{sansuc} gave a formula for Tamagawa numbers of connected reductive groups over number fields, modulo some deep results which were unknown at the time but which have since been proven, especially Weil's Tamagawa number conjecture mentioned above. Sansuc's formula is
\begin{equation}
\label{sansucformula}
\tau(G) = \frac{\# \Pic(G)}{\# \Sha(G)},
\end{equation}
where $\Pic(G)$ is the Picard group of $G$ (that is, the group of line bundles on $G$ up to isomorphism), and $\Sha(G)$ (also denoted $\Sha^1(G)$ or $\Sha^1(k, G)$) is the Tate-Shafarevich set of $G$, defined as
\[
\Sha(G) := \ker \left( {\rm{H}}^1(k, G) \rightarrow \prod_v {\rm{H}}^1(k_v, G) \right);
\]
i.e., the set of isomorphism classes of principal $G$-torsors over $k$ that have points everywhere locally. This is equivalent to the definition
\[
\Sha(G) = \ker \left( {\rm{H}}^1(k, G) \rightarrow {\rm{H}}^1(\A, G) \right)
\]
by the well-known Proposition \ref{adeliccohomologyprop} below.

The finiteness of $\Pic(G)$ for reductive $G$ (over any field whatsoever) is part of Sansuc's work, whereas the finiteness of $\Sha(G)$ lies much deeper. For abelian varieties, this latter finiteness is a major open problem, but for affine groups $G$ of finite type over global fields, the finiteness of $\Sha(G)$ is known and due to Borel and Serre over number fields (where passage to the reductive case is easy), Harder and Oesterl\'e in the reductive and solvable cases respectively over global function fields, and Conrad in the general case over global function fields; see \cite[\S1.3]{conrad} and references therein for more details and references.

\subsection{Beyond the reductive case}
\label{subsectionbeyondred}

Modulo Weil's Tamagawa number conjecture, recently proven over global function fields by Gaitsgory and Lurie \cite{lurie}, Sansuc's proof works just as well for reductive groups over global function fields. Over number fields, one may easily show that Sansuc's formula (\ref{sansucformula}) holds for {\em all} connected linear algebraic groups, by reducing to the reductive case. Over function fields, the gulf between reductive groups and general (connected) linear algebraic groups is vast. Let us explain why this is.

Given a connected linear algebraic group $G$ over a field $k$, the {\em $k$-unipotent radical} of $G$, denoted $\mathscr{R}_{u, k}(G)$, is defined to be the maximal smooth connected normal unipotent $k$-subgroup of $G$. (That there is a unique such maximal subgroup is a standard result in the structure theory of linear algebraic groups over general fields.) Then we have an exact sequence of connected linear algebraic groups
\begin{equation}
\label{predsequence}
1 \longrightarrow \mathscr{R}_{u, k}(G) \longrightarrow G \longrightarrow H \longrightarrow 1,
\end{equation}
where the group $H$ is {\em pseudo-reductive}, which means (by definition) that $\mathscr{R}_{u, k}(H) = 1$. 

When $k$ is a perfect field, two very nice things happen. First, the group $\mathscr{R}_{u, k}(G)$ is split unipotent; that is, it admits a filtration with successive quotients $k$-isomorphic to $\Ga$, since any smooth connected unipotent group over a perfect field is split \cite[Cor.\,15.5(ii)]{borelalggroups}. Thus, this group is in a suitable sense rather simple to understand, and Sansuc's formula over perfect global fields, i.e., number fields, may be reduced to the case of the pseudo-reductive quotient $H$. Second, over any field, one has (just by definition of the $k$-unipotent radical) an inclusion 
\begin{equation}
\label{unipradicalinclusion}
\mathscr{R}_{u, k}(G)_{\overline{k}} \subset \mathscr{R}_{u, \overline{k}}(G_{\overline{k}}).
\end{equation}
The subgroup $\mathscr{R}_{u, \overline{k}}(G_{\overline{k}})$ of $G_{\overline{k}}$ is characteristic (i.e., preserved by all $\overline{k}$-automorphisms of $G_{\overline{k}}$), hence it is preserved by the action of Aut$(\overline{k}/k)$. If $k$ is perfect, then by Galois descent this $\overline{k}$-subgroup descends to a $k$-subgroup of $G$, hence the inclusion (\ref{unipradicalinclusion}) is an equality. In particular, if $k$ is perfect, then $G$ is pseudo-reductive if and only if it is reductive (which by definition means that $G_{\overline{k}}$ is pseudo-reductive). Thus, if $k$ is perfect, then the group $H$ in (\ref{predsequence}) is reductive and not just pseudo-reductive.

When $k$ is imperfect, however, then both of these nice conclusions may fail: $\mathscr{R}_{u, k}(G)$ can be nonsplit and in fact can be very complicated, and $H$ can be pseudo-reductive but not reductive. Indeed, over every imperfect field, there are many nonsplit smooth connected unipotent groups, as well as many pseudo-reductive groups that are not reductive. (An interesting basic example of the latter is ${\rm{R}}_{k'/k}({\rm{SL}}_p)/{\rm{R}}_{k'/k}(\mu_p)$ for a nontrivial purely inseparable finite extension $k'/k$ in characteristic $p>0$, where ${\rm{R}}_{k'/k}$ denotes Weil restriction of scalars; there are many other examples with more intricate structure.)

In fact, Sansuc's formula (\ref{sansucformula}) fails even for forms of $\Ga$. Indeed, by Theorem \ref{tamagawaformula} below, in order to prove this failure it suffices to give an example of a form $U$ of $\Ga$ over a global function field $k$ such that $\Ext^1(U, \Gm) \neq \Pic(U)$ (see (\ref{extdef}) below for the definition of $\Ext^1(U, \Gm)$). Examples of such $U$ are provided by \cite[Prop.\,5.12]{rospic}.

In order to modify Sansuc's formula in the function field setting, therefore, we introduce a subgroup of $\Pic(G)$ which keeps track of the group structure on $G$ (unlike $\Pic(G)$, which only depends on the scheme structure). Let $m, p_i: G \times G \rightarrow G$ ($i = 1, 2$) denote the multiplication and projection maps. Then we let $\Ext^1(G, \Gm) \subset \Pic(G)$ denote the set of primitive line bundles on $G$. That is,
\begin{equation}
\label{extdef}
\Ext^1(G, \Gm) := \{\mathscr{L} \in \Pic(G) \mid m^*\mathscr{L} \simeq p_1^*\mathscr{L} \otimes p_2^*\mathscr{L}\}.
\end{equation}
The reason for the notation is that any extension of $G$ by $\Gm$ is in particular a $\Gm$-torsor over $G$, hence we get a homomorphism $\Ext^1_{{\rm{Yon}}}(G, \Gm) \rightarrow \Pic(G)$, where the ``Yoneda Ext'' group $\Ext^1_{{\rm{Yon}}}(G, \Gm)$ denotes the set of $k$-isomorphism classes of extensions of $G$ by $\Gm$ made into a group via Baer sum. (Any such extension $E$ is automatically represented by a smooth connected affine $k$-group that is a central extension of $G$ by $\Gm$: the automorphism scheme of $\Gm$ is \'etale, hence the $k$-group map $E \rightarrow {\rm{Aut}}_{\Gm/k}$ induced by conjugation is constant.) This induces an isomorphism $\Ext^1_{{\rm{Yon}}}(G, \Gm) \xrightarrow{\sim} \Ext^1(G, \Gm)$. (This is essentially \cite[Thm.\,4.12]{colliot-thelene}, though the result there is stated only for $G$ of multiplicative type. The proof is the same in general, using Chevalley's Unit Theorem.)

We should also mention that when $G$ is commutative, the notation $\Ext^1(G, \Gm)$ may also be used to denote the derived-functor Ext in the category of fppf abelian sheaves on $\Spec(k)$. This latter Ext group is canonically isomorphic to the group $\Ext^1(G, \Gm)$ defined above \cite[Prop.\,4.3]{rospic}, so there is no ambiguity in the notation. Another nice property of the group $\Ext^1(G, \Gm)$ is that it is finite for any connected linear algebraic group $G$ over a {\em global} field $k$ \cite[Thm.\,1.1]{rospic}. This result is truly arithmetic in nature, as it fails over every local function field and over every imperfect separably closed field \cite[Prop.\,5.9]{rospic}. Let us also remark that if $k$ is a perfect field then the inclusion $\Ext^1(G, \Gm) \subset \Pic(G)$ is an equality for any connected linear algebraic $k$-group $G$ (see the paragraph preceding the statement of Theorem 1.5 in \cite{rospic}). In particular, this equality holds when $k$ is a number field.

A crucial tool for making progress in the study of arithmetic invariants (including Tamagawa numbers) beyond the reductive and number field settings, and which plays a central role throughout this paper, is a substantial generalization of classical Tate duality. Recall that Tate duality, discovered by Poitou and Tate in the 1960's, is a collection of results on the cohomology of finite discrete Galois modules over local and global fields (with a restriction on the order of the module in the positive characteristic setting). These include local duality theorems as well as a global $9$-term exact sequence which in particular involves duality theorems between suitable Tate-Shafarevich groups of $G$ and its fppf $\Gm$-dual $\widehat{G} := \calHom(G, \Gm)$. Recently, \v{C}esnavicius \cite{ces} generalized these results to all finite commutative group schemes over global fields (the classical case being that of finite {\em \'etale} commutative group schemes with \'etale dual), replacing Galois cohomology with fppf cohomology, and replacing various restricted products that appear in the classical duality theorems with adelic cohomology groups. 

Building upon \v{C}esnavi\v{c}ius' work, the author generalized these results further to all affine commutative group schemes of finite type over global fields in \cite{rostateduality}. The essential new feature is to handle positive-dimensional unipotent coefficients and their $\Gm$-dual sheaves over global function fields. These duality results apply in particular to all commutative linear algebraic groups, hence their relevance to the current paper. 

Another essential ingredient for the results in this paper is the classification of pseudo-reductive groups obtained by Conrad, Gabber and Prasad \cite{cgp}, which we review in \S \ref{classificationsection}. This classifies pseudo-reductive groups ``up to the commutative case,'' and commutative pseudo-reductive groups are essentially a black box about which one can say very little. It is therefore crucial that one have an alternative method for understanding commutative groups, and this is provided by global Tate duality for positive-dimensional commutative affine group schemes of finite type (as opposed just the $0$-dimensional -- i.e., finite -- case), henceforward referred to as ``global Tate duality'', or simply ``Tate duality.''

\subsection{Results}
Now let us return to our discussion of global invariants of linear algebraic groups. Unfortunately, Tamagawa numbers of non-commutative unipotent groups exhibit rather pathological behavior. Indeed, many arithmetic quantities exhibit bad behavior beyond the commutative and pseudo-reductive settings. The motto of this work is: ``commutative and pseudo-reductive groups behave nicely, and everything else is pathological.'' Such pathologies with non-commutative unipotent groups will be exhibited in a forthcoming paper which will provide counterexamples to the main results of this work for such groups.

Our first example of nice behavior for commutative and pseudo-reductive groups is the following generalization of Sansuc's formula (\ref{sansucformula}).

\begin{theorem}
\label{tamagawaformula}
Let $G$ be a connected linear algebraic group over a global field $k$. Assume that $G$ is either commutative or pseudo-reductive. Then
\[
\tau(G) = \frac{\# \Ext^1(G, \Gm)}{\# \Sha(G)}.
\]
\end{theorem}

Comparing Theorem \ref{tamagawaformula} with (\ref{sansucformula}), it is natural to ask when 
$\Ext^1(G, \Gm) = \Pic(G)$. As noted in \S \ref{subsectionbeyondred}, this fails even for general $k$-forms $G$ of $\Ga$. So the natural question in light of Theorem \ref{tamagawaformula} is whether this equality holds for pseudo-reductive groups. The answer in general is negative; see Remark \ref{PicneqExtexamples}.

Theorem \ref{tamagawaformula} is the analogue for linear algebraic groups of
Bloch's volume-theoretic formulation of the
Birch and Swinnerton--Dyer Conjecture for abelian varieties as discussed in \cite{bloch}. To explain the connection in the special case of finite Mordell--Weil group
(\cite{bloch} handles any Mordell--Weil rank), 
let $B$ be a $g$-dimensional abelian variety with finite Mordell--Weil group over a global function field $k$, and let $\mathcal{B}$ be its Neron model over the smooth proper geometrically connected curve $X$ with function field $k$. As usual, let $\A$ denote the adele ring of $k$. By the theory of abelian varieties over finite fields, the local $L$-factor for the abelian variety at a place $v$ of $k$ satisfies
\[
L_v(B, 1) = \frac{q_v^g}{\# \mathcal{B}_{\F_v}^0(\F_v)},
\]
where  $\F_v$ is the residue field of at $v$, $q_v := \# \F_v$, and (for globalization purposes) 
$L_v$ is expressed as a function of the parameter $s \in \mathbf{C}$ rather than the more traditional
$q_v^{-s}$. 

The integral structure at $v$ arising
from the N\'eron model defines a canonical Haar measure on $B(k_v)$ for each $v$, and 
a calculation using smoothness of the N\'eron model to compute the volume of the preimage 
in $B(k_v) = \mathcal{B}(\calO_v)$ of $0 \in \mathcal{B}(\F_v)$ shows that relative
to this measure we have 
\[
{\rm{vol}}(B(k_v)) = \frac{\# \mathcal{B}(\F_v)}{q_v^g} = \frac{\# \mathcal{B}^0(\F_v)\cdot \# \Phi_v(\F_v)}{q_v^g} =
L_v(B,1)^{-1} \cdot \#\Phi_v(\F_v),
\]
where $\Phi_v := \mathcal{B}_{\F_v}/\mathcal{B}^0_{\F_v}$ is the finite \'etale component group of the $v$-fiber
(and the second equality uses Lang's Theorem to ensure that $\mathcal{B}(\F_v) \rightarrow \Phi(\F_v)$ is surjective). The quantity $c_v := \# \Phi(\F_v)$ is usually called the Tamagawa factor at $v$; it equals $1$ for all but finitely many $v$. A reasonable definition of
 $\tau(B)$ as a ``volume'' for $B(\A)/B(k)$ is therefore
\[
\tau(B) := \frac{L(B, 1)^{-1} \prod_v c_v}{\# B(k)}
\]
if $L(B,1)\ne 0$. (There is a procedure that puts a canonically normalized Haar measure on $B(\A)$ relative
to which one can instead define $\tau(B)$ as the actual volume of $B(\A)/B(k)$, but we do not discuss that here except for the brief Remark \ref{taurem} below.)

The group $\Ext^1(B, \Gm)$ consists of the translation-invariant line bundles on $B$. The Weil--Barsotti formula identifies this with
$B'(k)$ where $B'$ is the dual abelian variety. Note that $B'(k)$ is finite, since $B(k)$ is finite. The analogue of
Theorem \ref{tamagawaformula} in this situation is therefore
\[
\frac{L(B, 1)^{-1} \prod_v c_v}{\# B(k)} \stackrel{?}{=} \frac{\# B'(k)}{\# \Sha^1(B)}
\]
(assuming finiteness of $\Sha^1(B)$). 
This is precisely the Birch and Swinnerton--Dyer Conjecture for $B$, including the non-vanishing of $L(B,1)$, since $B(k)$ is finite.

\begin{remark}\label{taurem}
In \cite{bloch}, Bloch treats the case of abelian varieties of arbitrary (potentially nonzero) rank, and $\tau(B)$ is defined to be an actual volume (suitably normalized). 
Under that definition and assuming finiteness of $\Sha^1(B)$, by \cite[Thm.\,1.17]{bloch} the equality $\tau(B) = \#B'(k)/\#\Sha^1(B)$
is equivalent to the Birch and Swinnerton--Dyer Conjecture for abelian varieties $B$ of rank $0$.
\end{remark}

\begin{remark}
In the general (possibly higher rank) case of the Birch and Swinnerton--Dyer Conjecture, the analogue of $\Ext^1(G, \Gm)$ is the torsion subgroup $$B'(k)_{{\rm{tor}}} = \Ext^1(B, \Gm)_{{\rm{tor}}} = \Pic(B)_{{\rm{tor}}},$$ the last equality coming from the fact that $\Pic(B)/\Pic^0(B)$ is torsion-free. It is therefore natural to wonder whether we have the equality $\Ext^1(G, \Gm) = \Pic(G)_{{\rm{tor}}}$ for all smooth connected commutative affine group schemes $G$ over the global function field $k$. This is false. Indeed, for an affine group scheme $G$ of finite type over any field, one may show that the group $\Pic(G)$ is torsion, but it is not true in general even for commutative pseudo-reductive $G$ (over global function fields) that $\Ext^1(G, \Gm) = \Pic(G)$; see Remark \ref{PicneqExtexamples}.
\end{remark}

In order to state the next theorem, we recall that an {\em inner form} $G'$ of a $k$-group $G$ is one whose isomorphism class lies in the image of the map ${\rm{H}}^1(k, G/Z_G) \rightarrow {\rm{H}}^1(k, {\rm{Aut}}_{G/k})$ induced by the map of group functors $G/Z_G \rightarrow {\rm{Aut}}_{G/k}$ which sends $g$ to conjugation by $g$, where the latter ${\rm{H}}^1$-set classifies fppf $k$-forms of $G$ up to isomorphism. Here, $Z_G$ is the scheme-theoretic center of $G$. We also say $G'$ is an {\em inner twist} of $G$, or is obtained from $G$ by inner twisting. The relation ``is an inner form of'' is an equivalence relation. If $G$ is {\em smooth} then $G/Z_G$ is smooth, so then classes in ${\rm{H}}^1(k, G/Z_G)$ split over $k_s$
and hence inner twists of $G$ are $k_s/k$-forms.

It is known that for reductive groups $G$, both the Tamagawa number $\tau(G)$ and $\# \Sha(G)$ are invariant under inner twist. Indeed, the invariance of $\# \Sha(G)$ follows by applying \cite[Thm.\,4.3]{sansuc} and twisting by a cohomology class. (For the definition of Galois twisting, as well as twisting of sequences, see \cite[Chap.\,I, \S 5,3]{serre}.) The assumption in \cite[Thm.\,4.3]{sansuc} that $G$ has no $E_8$-factor was stated because various cohomological results on simply connected groups that Sansuc used in his paper were unknown at the time; they have since been proven \cite{chernousov}. 

The invariance of $\tau(G)$ under inner twist for reductive groups over number fields (from which it follows easily for all connected linear algebraic groups over such fields) was proved by Kottwitz \cite{kottwitz}. We will show that the same holds for pseudo-reductive groups over global function fields (there is of course nothing to say for commutative groups):

\begin{theorem}
\label{innerinvariance}
Let $G$ be a pseudo-reductive group over a global function field $k$. Then $\tau(G)$ and $\# \Sha(G)$ are invariant under inner twisting. That is, if $G'$ is an inner form of $G$, then $\tau(G) = \tau(G')$ and $\# \Sha(G) = \# \Sha(G')$.
\end{theorem}

Note that $G'$ is automatically pseudo-reductive when $G$ is, since pseudo-reductivity may be checked over $k_s$ \cite[Prop.\,1.1.9]{cgp}. In fact, thanks to Theorem \ref{tamagawaformula}, in conjunction with Proposition \ref{extinvtinnertwist} of this paper, the invariance of $\tau$ in Theorem \ref{innerinvariance} is equivalent to that of $\# \Sha$. We will prove Theorem \ref{innerinvariance} by proving the invariance of $\tau$, and then use this deduce the invariance of $\# \Sha$. The inner invariance of $\tau$ will depend crucially on the now-established Weil conjecture on Tamagawa numbers (that simply connected groups have Tamagawa number $1$). This order of reasoning, first proving Weil's conjecture and then using it to deduce the invariance of Tamagawa numbers under inner twist, is the exact opposite of the reasoning used in the number field setting. There, Weil's Tamagawa number conjecture, was proved over number fields by showing that Tamagawa numbers are invariant under inner twist (done by Kottwitz \cite{kottwitz} by using the trace formula, modulo the same $E_8$ issue as mentioned above, arising from lack of knowledge of the Hasse principle for such groups, a defect remedied by Chernousov \cite{chernousov}), thus reducing one to the quasi-split case handled by Lai \cite{lai} (building on work of Langlands \cite{langlands}). The proof by Gaitsgory and Lurie \cite{lurie} of Weil's conjecture over function fields proceeds via a totally different geometric route.

Before turning to our next result, we recall a well-known relation between adelic cohomology and the cohomology of the local completions of a global field $k$.

\begin{proposition}
\label{adeliccohomologyprop}
Let $G$ be a smooth connected group scheme over a global field $k$. Then the map ${\rm{H}}^1(\A, G) \rightarrow \prod_v {\rm{H}}^1(k_v, G)$ induced by the projection maps $\A \rightarrow k_v$ induces a bijection of pointed sets
\[
{\rm{H}}^1(\A, G) \xlongrightarrow{\sim} \coprod_v {\rm{H}}^1(k_v, G).
\]
\end{proposition}

\begin{proof}
This is a special case of \cite[Thm.\,2.13]{ces}.
\end{proof}

For a smooth connected $k$-group scheme $G$, consider the map
\[
\theta_G: {\rm{H}}^1(k, G) \longrightarrow \coprod_v {\rm{H}}^1(k_v, G).
\]
By definition, $\Sha(G) = \ker(\theta_G)$, but what about the other fibers? For $x \in {\rm{H}}^1(k, G)$, the fiber $\theta_G^{-1}(\theta_G(x))$ containing $x$ is in bijection with the set $\Sha(G_x)$, where $G_x$ is the $k$-form of $G$ obtained by taking the image of $x$ under the map ${\rm{H}}^1(k, G) \rightarrow {\rm{H}}^1(k, G/Z_G) \rightarrow {\rm{H}}^1(k, {\rm{Aut}}_{G/k})$. In particular, $G_x$ is an inner form of $G$!

Theorem \ref{innerinvariance} therefore implies that when $G$ is pseudo-reductive, all of the nonempty fibers of $\theta_G$ have the same size (though that theorem is strictly stronger than this, since the map ${\rm{H}}^1(k, G) \rightarrow {\rm{H}}^1(k, G/Z_G)$ need not be surjective). It is then natural to ask for which $\alpha \in {\rm{H}}^1(\A, G) = \coprod_v {\rm{H}}^1(k_v, G)$ the fiber $\theta_G^{-1}(\alpha)$ is nonempty.

In order to answer this question, we note that for any smooth connected group scheme $G$ over a global field $k$, we have a complex of pointed sets
\begin{equation}
\label{H^1Extcomplex}
{\rm{H}}^1(k, G) \longrightarrow {\rm{H}}^1(\A, G) \longrightarrow \Ext^1(G, \Gm)^*
\end{equation}
which we will now define. The map ${\rm{H}}^1(k, G) \rightarrow {\rm{H}}^1(\A, G)$ is induced by the diagonal inclusion $k \hookrightarrow \A$. To define the second map, consider an extension
\begin{equation}
\label{extensioneqn1}
1 \longrightarrow \Gm \longrightarrow E \longrightarrow G \longrightarrow 1
\end{equation}
(which is automatically central since connectedness of $G$ implies connectedness of $E$) and an element $\alpha \in {\rm{H}}^1(\A, G)$. We obtain for each place $v$ of $k$ the element $\alpha_v \in {\rm{H}}^1(k_v, G)$; $\alpha_v$ is trivial for all but finitely many $v$ by Proposition \ref{adeliccohomologyprop}. Due to the centrality of the extension (\ref{extensioneqn1}), we get a connecting map ${\rm{H}}^1(k_v, G) \rightarrow {\rm{H}}^2(k_v, \Gm)$, this ${\rm{H}}^2$ being $\Q/\Z$, $\frac{1}{2}\Z/\Z$, or $0$, depending on whether $k_v$ is non-archimedean, $\mathbf{R}$, or $\C$ (by taking local Brauer invariants). 

Thus we get for each place $v$ of $k$ an element of $\Q/\Z$, all but finitely many of which are $0$. Adding these up produces an element of $\Q/\Z$. This procedure yields a map ${\rm{H}}^1(\A, G) \times \Ext^1(G, \Gm) \rightarrow \Q/\Z$. One may check that this map is additive in the second argument, hence induces a map of pointed sets ${\rm{H}}^1(\A, G) \rightarrow \Ext^1(G, \Gm)^*$. This defines the second map in (\ref{H^1Extcomplex}). Further, the image of any element of ${\rm{H}}^1(k, G)$ under this map is $0$, since the sum of the local invariants of a global Brauer class is $0$ by class field theory.

\begin{theorem}
\label{pseudoredcomplexexact}
Let $k$ be a global function field, and let $G$ be a connected linear algebraic $k$-group that is either commutative or pseudo-reductive. Then the complex
\[
{\rm{H}}^1(k, G) \longrightarrow {\rm{H}}^1(\A, G) \longrightarrow \Ext^1(G, \Gm)^*
\]
defined in $(\ref{H^1Extcomplex})$ is an exact sequence of pointed sets. That is, the kernel of the second map is the image of the first.
\end{theorem}

\begin{remark}
\label{compatiblewithtateduality}
The complex of Theorem \ref{pseudoredcomplexexact} is compatible with global Tate duality in the following sense. Given a commutative 
affine group scheme $G$ of finite type over a field $k$, let $$\widehat{G} := \calHom(G, \Gm)$$ denote the fppf $\Gm$-dual sheaf. We have a functorial (in $G$ and $k$) exact sequence 
\begin{equation}
\label{H^1(G^)=kerExt}
0 \longrightarrow {\rm{H}}^1(k, \widehat{G}) \longrightarrow \Ext_{{\rm{cent}}}(G, \Gm) \longrightarrow \Ext_{{\rm{cent}}}(G_{\overline{k}}, (\Gm)_{\overline{k}})
\end{equation}
defined as follows. (When $G$ is disconnected, we have to specify that we are only considering central extensions of $G$ by $\Gm$, as arbitrary such extensions need not be central.) 

A central extension $E$ of $G$ by $\Gm$ splits fppf locally over $k$ if and only if it splits over $\overline{k}$, by the Nullstellensatz and standard spreading-out arguments. Thus, 
\[
\ker \left(\Ext_{{\rm{cent}}}(G, \Gm) \rightarrow \Ext_{{\rm{cent}}}(G_{\overline{k}}, (\Gm)_{\overline{k}})\right)
\]
consists of those extensions that split fppf locally, i.e., the fppf forms of the trivial extension $E = \Gm \times G$ (with the obvious extension structure). But one easily checks that the automorphism functor of the trivial extension (as an extension of $G$ by $\Gm$) is $\widehat{G}$. It follows that the above kernel is canonically (up to a universal choice of sign) isomorphic to ${\rm{H}}^1(k, \widehat{G})$. 

When $k$ is a global field, we have a complex
\begin{equation}
\label{H^1exactsequence1}
{\rm{H}}^1(k, G) \longrightarrow {\rm{H}}^1(\A, G) \longrightarrow {\rm{H}}^1(k, \widehat{G})^*,
\end{equation}
where the first map is induced by the diagonal inclusion $k \hookrightarrow \A$, and the second by cupping everywhere locally and adding the invariants. That is, given $\alpha \in {\rm{H}}^1(\A, G)$ and $\beta \in {\rm{H}}^1(k, \widehat{G})$, we have for each place $v$ of $k$ the cup product pairing 
\[
{\rm{H}}^1(k_v, G) \times {\rm{H}}^1(k_v, \widehat{G}) \rightarrow {\rm{H}}^2(k_v, \Gm) \xrightarrow{\sim} \Q/\Z,
\]
where the last map is once again the Brauer invariant. Thus, by taking the cup product of $\alpha_v$ and $\beta_v$ for each $v$ and then adding the results, we obtain the second map in (\ref{H^1exactsequence1}) above. (By the compatibility of the cohomology with direct limits of rings \cite[Prop.\,D.0.1]{rostateduality}, both $\alpha_v$ and $\beta_v$ lift to ${\rm{H}}^1(\calO_v, \cdot)$ for all but finitely many $v$, hence their cup product lands in ${\rm{H}}^2(\calO_v, \Gm) = 0$, so the sum above contains only finitely many nonzero terms.) The sequence (\ref{H^1exactsequence1}) is a complex, once again because the sum of the local invariants of a global Brauer class is $0$. Part of the statement of global Tate duality is that the sequence (\ref{H^1exactsequence1}) is exact for any affine commutative $k$-group $G$ of finite type \cite[Thm.\,1.2.8]{rostateduality}. 

At any rate, the point we would like to make here is that for connected commutative affine $k$-group schemes $G$ of finite type,  the two complexes (\ref{H^1Extcomplex}) and (\ref{H^1exactsequence1}) are compatible via the first map in (\ref{H^1(G^)=kerExt}) (with the appropriate universal choice of sign). This compatibility may be checked by using the description of cup product (on ${\rm{H}}^1$) in terms of \v{C}ech cohomology. This compatibility, in conjunction with Tate duality \cite[Thm.\,1.2.8]{rostateduality} implies that Theorem \ref{pseudoredcomplexexact} holds for connected commutative linear algebraic groups. The commutative case will play an important role in the proof for pseudo-reductive groups.
\end{remark}

\begin{remark}
Even though Theorem \ref{pseudoredcomplexexact} is true over number fields (in fact, it is true for all connected linear algebraic groups over number fields, by an easy reduction to the reductive case), we have only stated it and will prove it for function fields in order to avoid issues with archimedean places, and because the main focus of this paper is on the function field setting.
\end{remark}

Our last main result is in a certain sense a reinterpretation of Theorem \ref{pseudoredcomplexexact}. We show that that result has implications for weak and strong approximation on quotients by $G$. In order to state this reinterpretation, let us recall what the Brauer--Manin obstruction is. Let $Y$ be a separated scheme of finite type over a global field $k$, let $M_k$ denote the set of places of $k$, and let $S \subsetneq M_k$. Let $\A^S := \prod'_{v \in M_k - S} k_v$ denote the ring of $S$-adeles. Then strong approximation asks about the closure $\overline{Y(k)} \subset Y(\A^S)$, and in particular, we say that strong approximation holds for $Y$ with respect to $S$ if this containment is an equality -- that is, if $Y(k)$ is dense in $Y(\A^S)$. (If $M_k - S$ is finite, then one typically refers to what we have called strong approximation with respect to $S$ as weak approximation with respect to $M_k - S$. One also typically speaks of strong approximation for finite $S$, though we prefer to deal with arbitrary $S$, partially so as not to have to speak about weak and strong approximation separately.)

Manin observed that elements of $\Br(Y) := {\rm{H}}^2_{\et}(Y, \Gm)$ yield constraints on the image of $Y(k) \subset Y(\A)$, and in particular to strong approximation, as follows. Given $\alpha \in \Br(Y)$, we obtain for each element $y \in Y(\A)$ an element $\alpha^*(y) \in \Br(\A)$. For each place $v$ of $k$, local class field theory provides us with the invariant map ${\rm{inv}}: \Br(k_v) \rightarrow \Q/\Z$, and one may show that for any $\beta \in \Br(\A)$, we have ${\rm{inv}}(\beta_v) = 0$ for all but finitely many $v$. We may therefore define the subset
\[
Y(\A)^{\alpha} := \left\{ y \in Y(\A) \, \middle \vert \, \sum_v {\rm{inv}}(\alpha^*(y_v)) = 0\right\} \subset Y(\A),
\]
where the sum is over all places $v$ of $k$. One may show that this is a closed subset of $Y(\A)$, and global class field theory implies that $Y(k) \subset Y(\A)^{\alpha}$. We therefore define
\[
Y(\A)^{\Br} := \underset{\alpha \in \Br(Y)}{\bigcap} Y(\A)^{\alpha},
\]
while for a finite set $S$ of places of $k$, we let $Y(\A^S)^{\Br}$ denote the image of $Y(\A)^{\Br}$ under the projection map $Y(\A) \rightarrow Y(\A^S)$. We then have $\overline{Y(k)} \subset Y(\A)^{\Br}$, and the former set is therefore also contained in $Y(\A^S)^{\Br}$. We say that the Brauer--Manin obstruction is the only obstruction to strong approximation on $Y$ with respect to $S$ if $\overline{Y(k)} = Y(\A^S)^{\Br}$.

\begin{theorem}
\label{weakapprox}
Let $k$ be a global field, let $G$ be a connected linear algebraic $k$-group that is either commutative or pseudo-reductive, and let $H$ be a finite type $k$-group scheme. Further assume that $\Sha^1(k, H) = 1$ and that strong approximation holds for $H$ with respect to the set $S \subsetneq M_k$ of places of $k$. Let $G \hookrightarrow H$ be a $k$-group inclusion, and let $Y := H/G$. Then the Brauer--Manin obstruction is the only obstruction to strong approximation on $Y$ with respect to $S$. That is, $\overline{Y(k)} = Y(\A^S)^{\Br}$.
\end{theorem}

Let us summarize the contents of this paper. In \S \ref{classificationsection} we review the 
classification of pseudo-reductive groups obtained in \cite{cgp}. In \S \S \ref{sectionexoticgeneral}-
\ref{exotictam=1section} we show that ``basic exotic'' and ``basic non-reduced'' pseudo-reductive 
groups have Tamagawa number $1$; this is necessary in order to prove Theorem 
\ref{tamagawaformula} in characteristics $2$ and $3$ (so the reader who is willing to ignore 
characteristics $2$ and $3$ may skip these sections or may take Proposition 
\ref{basicexotictam=1} as a black box and then read the proof of Theorem \ref{tamagawaformula}, 
given in \S \ref{sectiontamagawaformula}, without ignoring characteristics $2$ and $3$). In \S 
\ref{sectiontamagawaformula} we prove Theorems \ref{tamagawaformula} and 
\ref{innerinvariance}. In \S \ref{sectionexactsequence} we prove Theorem 
\ref{pseudoredcomplexexact}. Finally, in \S \ref{approximationsection} we prove Theorem \ref{weakapprox}.

\subsection{Notation and Conventions} 

Throughout this paper, $k$ denotes a field and, when it appears, $p$ denotes a prime number equal to the characteristic of $k$.

A {\em linear algebraic group} over $k$ means a smooth affine $k$-group scheme. When $k$ is a global field, $k_v$ denotes the completion of $k$ at a place $v$, $\calO_v$ the ring of integers of $k_v$ when $v$ is non-archimedean, and $\A_k$ (or just $\A$ when there can be no confusion) denotes the ring of adeles of $k$.

For any affine $k$-group scheme $G$ of finite type, we define $\widehat{G}$ to be the functor on $k$-algebras given by
$$\widehat{G}(A) = \Hom_{A-{\rm{gp}}}(G_A, \Gm)$$
(so $\widehat{G}(k)$ is the group of $k$-homomorphisms $G \to \Gm$).

We must also make some remarks about the behavior of cohomology in exact sequences. Given an affine group scheme $G$ of finite type over a field $k$, one may define ${\rm{H}}^1(k, G)$ as the set of fppf $G$-torsors over Spec$(k)$ up to isomorphism; this is a pointed set, and if $G$ is commutative then it is even an abelian group. When $G$ is commutative, one may also define the higher cohomology groups. Note that the affineness of $G$ implies that the torsor sheaves classified by this ${\rm{H}}^1$-set are all representable due to the effectivity of descent for affine schemes.

If $G$ is smooth over $k$, then the \'etale and fppf $G$-torsors agree, so we may define ${\rm{H}}^1(k, G)$ to be the set of isomorphism classes of \'etale or fppf $G$-torsors over Spec$(k)$. The \'etale and fppf cohomology groups ${\rm{H}}^i(k, G)$ agree for all $i$ when $G$ is smooth and commutative \cite[Thm.\,11.7]{briii}. When $G$ is smooth, therefore, all of the defined cohomology groups may be defined in terms of Galois cohomology by using the language of cocycles and coboundaries. This is the language used in \cite{serre}.

On the other hand, for some purposes it is necessary to work with cohomology over general base schemes, or even fppf cohomology over fields, in which case the language of Galois cocycles is insufficient. One may sometimes replace these with \v{C}ech cohomology, but it is also useful to develop the entire theory in a more intrinsic manner, using the language of torsors.

Given an exact sequence
\begin{equation}
\label{asampleseq}
1 \longrightarrow G' \longrightarrow G \longrightarrow G'' \longrightarrow 1
\end{equation}
of smooth affine $k$-group schemes, one may compute the associated cohomology exact sequence using Galois cohomology. In this language, much of the familiar formalism of long exact sequences which comes out of (\ref{asampleseq}) when the groups are commutative remains true in the non-commutative setting. This is discussed in \cite[Chap.\,I, \S 5]{serre}. 

Much of the discussion in \cite{serre} is done in the language of torsors as well as Galois cocycles, though not all of it. We will at various points throughout this work require results for cohomology sets of the form ${\rm{H}}^1(\A_k, G)$, where $G$ is a smooth connected affine group scheme over a global field $k$. Strictly speaking, the results in \cite{serre} do not apply to these sets. There are, however, two ways around this.
The first is to simply invoke Proposition \ref{adeliccohomologyprop} to reduce assertions for adelic cohomology to the case of fields, where one may use Galois cohomology and apply the results in \cite{serre} directly. The second approach is to state and prove all of the results of \cite[Chap.\,I, \S 5]{serre} in the more general context of torsors over an arbitrary base rather than merely Galois (i.e., \'etale) cohomology over fields. This is essentially done in \cite[Appendix B]{conrad}. 

Strictly speaking, \cite[Appendix B]{conrad} only treats the case of fields, but all of the techniques and arguments used there for deriving properties of the long exact sequence associated to the short exact sequence (\ref{asampleseq}) work over a general base ring for affine groups, since affineness ensures the effectivity of all descent datum, which is necessary if one wishes to ensure that all fppf sheaf torsors are in fact representable by schemes. (If one does not care about such representability, and is satisfied with working just with sheaf torsors, then even this assumption is unnecessary.)

Throughout this work, we will refer to \cite{serre} and invoke Proposition \ref{adeliccohomologyprop}, but we wanted to make the reader aware of the more general results essentially proved in \cite[Appendix B]{conrad} which allow one to work directly over the adele ring rather than invoking this ``trick''.

\subsection{Acknowledgements}

I would like to thank the referee, who provided several useful suggestions, especially that Theorem \ref{pseudoredcomplexexact} might have an interpretation in terms of the Brauer--Manin obstruction to weak and strong approximation on quotient spaces. It is a pleasure to thank Brian Conrad, who worked tirelessly to provide me with many helpful suggestions and improvements, on both the presentation and the arguments throughout this paper.

\section{The classification of pseudo-reductive groups}
\label{classificationsection}

The purpose of this section is to review for the reader the classification of pseudo-reductive groups obtained in \cite{cgp}, which will play a major role in our work in this paper.

A central role in this classification is played by the so-called ``standard construction'' (and the ``generalized standard construction'' in characteristics $2$ and $3$), which we now recall. Let $k$ be a field, and let $k'$ be a nonzero finite reduced $k$-algebra (that is, a product of finitely many finite field extensions of $k$). Let $G'$ be a $k'$-group such that its fiber over each factor field $k_i'$ of $k'$ is a connected semisimple group that is absolutely simple and simply connected. Let $T' \subset G'$ be a maximal $k'$-torus (i.e., its $k_i'$-fiber is a maximal $k_i'$-torus of the $k_i'$-fiber of $G'$). Then $T'$ acts on $G'$ by conjugation, and this action factors uniquely through an action $\xi$ by $T'/Z_{G'}$, where $Z_{G'}$ is the (scheme-theoretic) center of $G'$. Therefore, $\mathscr{C} := \R_{k'/k}(T')$ acts on $\R_{k'/k}(G')$ by conjugation, and this action factors naturally through the evident action $\R_{k'/k}(\xi)$ of $\R_{k'/k}(T'/Z_{G'})$. Beware that the map 
$$f: \R_{k'/k}(T') \rightarrow \R_{k'/k}(T'/Z_{G'})$$
is {\em not} surjective when $Z_{G'}$ has non-\'etale fiber over a factor field of $k'$ that is not \'etale over $k$. 

Suppose that we are given a factorization of $f$ as
\[
\R_{k'/k}(T') =: \mathscr{C} \xrightarrow{\phi} C \xrightarrow{\psi} \R_{k'/k}(T'/Z_{G'})
\]
with $C$ commutative pseudo-reductive. (Neither $\phi$ nor $\psi$ need be surjective.) Then $C$ acts on $\R_{k'/k}(G')$ through $\psi$ and the action of $\R_{k'/k}(T'/Z_{G'})$. Let $j: \mathscr{C} \hookrightarrow \R_{k'/k}(G')$ be the inclusion. One easily checks that the ``anti-diagonal'' map $\alpha: \mathscr{C} \rightarrow C \ltimes \R_{k'/k}(G')$ given by $(-\phi, j)$ is an isomorphism onto a central subgroup. Thus, we may form the cokernel
\[
G := (C \ltimes \R_{k'/k}(G'))/\mathscr{C}.
\]
This $G$ is pseudo-reductive \cite[Prop.\,1.4.3]{cgp}. A pseudo-reductive group is called {\em standard pseudo-reductive} if it is either commutative or constructed in the above manner \cite[Definition 1.4.4, Thm.\,4.1.1]{cgp}. If ${\rm{char}}(k) > 3$, then every pseudo-reductive group over $k$ is standard pseudo-reductive \cite[Thm.\,5.1.1(1)]{cgp}, so the reader who wishes to avoid extra complications in characteristics $2$ and $3$ may just concentrate on the standard case.

\begin{remark}
The $k$-group $G$ determines the data $(k'/k, G')$ and $\R_{k'/k}(G') \rightarrow G$ uniquely up to unique isomorphism by \cite[Props.\,4.1.4(1), 4.2.4, 5.1.7(1), A.5.14]{cgp}.
\end{remark}

In characteristics $2$ and $3$, there exist ``exotic'' pseudo-reductive groups beyond the standard case. Let us describe these now. The first class of such groups are the basic exotic pseudo-reductive groups, defined as follows. Pick $p \in \{2, 3\}$, and let $k$ be a field of characteristic $p$. Let $G$ be a connected semisimple $k$-group that is absolutely simple and simply connected with an edge of multiplicity $p$ in its absolute Dynkin diagram. Then there is a unique nontrivial factorization of the Frobenius isogeny $F_{G/k}: G \rightarrow G^{(p)}$ as
\[
G \xrightarrow{\pi} \overline{G} \rightarrow G^{(p)},
\]
for a specific non-central isogeny $\pi$ called the {\em very special isogeny}.
The group $\overline{G}$ is simply connected with root system dual to that of $G$ (so $\overline{G}$ is of the same type as $G$ except in the case when $G$ has type B$_n$, C$_n$ respectively with $n \geq 3$, in which case $\overline{G}$ has type C$_n$, B$_n$, respectively);
see \cite[\S7.1]{cgp} for more details.

Given a nontrivial finite extension $k'/k$ of fields of characteristic $p$ such that $(k')^p \subset k$, we let $f := \R_{k'/k}(\pi_{k'})$.
This map is {\em not} surjective, and its kernel is non-smooth and positive-dimensional. The canonical  $k$-subgroup $\overline{G} \hookrightarrow \R_{k'/k}(\overline{G}_{k'})$ lies in the image of $f$
(as it lies in the image of the $k$-subgroup $G \subset \R_{k'/k}(G_{k'})$ that is a Levi $k$-subgroup \cite[Prop.\,7.3.1]{cgp}), and 
although $f$ is not smooth it turns out that 
the preimage $\mathscr{G} := f^{-1}(\overline{G})$ is smooth and in fact even pseudo-reductive \cite[Thm.\,7.2.3]{cgp}. 
This is almost the general definition of a basic exotic pseudo-reductive group:

\begin{definition}\cite[Definition 7.2.6]{cgp}
\label{basicexotic}
Let $k$ be an imperfect field of characteristic $p \in \{2, 3\}$. A {\em basic exotic pseudo-reductive group} is a $k$-group $\mathscr{G}$ that is $k$-isomorphic to a $k$-group scheme of the form $f^{-1}(\overline{G})$ where:
\begin{itemize}
\item[(i)] $f = \R_{k'/k}(\pi')$ for a nontrivial finite extension $k'/k$ such that $(k')^p \subset k$, where $\pi': G' \rightarrow \overline{G}'$ is the very special isogeny for a connected semisimple $k'$-group $G'$ that is absolutely simple and simply connected with Dynkin diagram containing an edge of multiplicity $p$;
\item[(ii)] the subgroup $\overline{G} \subset \R_{k'/k}(\overline{G}')$ is a Levi $k$-subgroup;
\item[(iii)] $f^{-1}(\overline{G})_{k_s}$ contains a Levi $k_s$-subgroup of the group $\R_{k'/k}(G')_{k_s} = \R_{k'_s/k_s}(G'_{k'_s})$.
\end{itemize}
\end{definition}

The pseudo-reductivity of such $\mathscr{G}$ reduces to that of the construction preceding Definition \ref{basicexotic} due to \cite[Lem.\,7.2.1]{cgp}.

\begin{remark}
\label{canonicalsurjection}
Condition (iii) in Definition \ref{basicexotic} is the most subtle. Several alternative equivalent formulations of this condition are given in \cite[Prop.\,7.3.1]{cgp}, two of which are that $f^{-1}(\overline{G})$ is $k$-smooth and that $\overline{G} \subset {\rm{im}}(f)$. In particular, that result tells us that the pseudo-reductive group $\mathscr{G}$ in Definition \ref{basicexotic} admits a natural surjective map 
\begin{equation}\label{fggexotic}
f: \mathscr{G} \twoheadrightarrow \overline{G}
\end{equation}
 to a simply connected group; this map has  kernel $\R_{k'/k}(\ker(\pi'))$ that is non-smooth and  positive dimensional; see the remarks just before the statement of \cite[Thm.\,7.2.3]{cgp}.
\end{remark}

\begin{remark}
The $k$-group $G$ uniquely determines the data $(k'/k, G', \overline{G}, f)$ up to unique isomorphism, by \cite[Prop.\,7.2.7(1), (2)]{cgp}.
\end{remark}

Before we go on, let us prove a result about tori in basic exotic pseudo-reductive groups that we will need in the sequel. Let $\mathscr{G}$ be a basic exotic pseudo-reductive group over a field $k$. Then, as observed in Remark \ref{canonicalsurjection} above, there is a canonical surjection $f: \mathscr{G} \twoheadrightarrow \overline{G}$ to a simply connected group with positive-dimensional non-smooth kernel. There is a bijection between the sets of maximal $k$-tori of $\overline{G}$ and maximal $k$-tori of $\mathscr{G}$ defined as follows. Given a maximal $k$-torus $\overline{T} \subset \overline{G}$, take $\mathscr{T}$ to be the unique maximal $k$-torus in $f^{-1}(\overline{T})$. That there is a unique such torus and that the map $\overline{T} \mapsto \mathscr{T}$ is a bijection between the sets of maximal $k$-tori is \cite[Cor.\,7.3.4]{cgp}. That corollary also says that the map $\mathscr{T} \rightarrow \overline{T}$ is an isogeny. Now we need the following result, which is in the same spirit as \cite[Prop.\,7.3.3]{cgp}.

\begin{lemma}
\label{pointsofcartan}
Let $k$ be an imperfect field of characteristic $p \in \{2, 3\}$ such that $[k: k^p]$ is finite. Let $\mathscr{G}$ be a basic exotic pseudo-reductive group over $k$ arising from the finite nontrivial purely inseparable extension $k' := k^{1/p}$, and let $f: \mathscr{G} \twoheadrightarrow \overline{G}$ be the canonical map onto a simply connected group over $k$. 

Let $\overline{T} \subset \overline{G}$ be a maximal $k$-torus, and let $\mathscr{T} \subset f^{-1}(\overline{T})$ be the corresponding maximal $k$-torus of $\mathscr{G}$. Then the smooth connected commutative Cartan subgroup $Z_{\mathscr{G}}(\mathscr{T})$ maps onto $\overline{T}$ and the induced map ${\rm{H}}^i(k, Z_{\mathscr{G}}(\mathscr{T})) \rightarrow {\rm{H}}^i(k, \overline{T})$ is an isomorphism for all $i$.
\end{lemma}

Note that if $[k: k^p] = p$, then every basic exotic pseudo-reductive $k$-group arises as in this lemma. That $Z_{\mathscr{G}}(\mathscr{T})$ is smooth connected commutative is part of \cite[Prop.\,1.2.4]{cgp}.

\begin{proof}
Since the maps $\mathscr{G} \rightarrow \overline{G}$ and $\mathscr{T} \rightarrow \overline{T}$ are surjective, it follows that $f(Z_{\mathscr{G}}(\mathscr{T})) = Z_{\overline{G}}(\overline{T}) = \overline{T}$. (Surjectivity holds because surjective maps of connected linear algebraic groups always map Cartan subgroups onto Cartan subgroups.) 
Since the map $Z_{\mathscr{G}}(\mathscr{T})(k_s) \rightarrow \overline{T}(k_s)$ is Galois-equivariant, in order to prove the lemma it suffices to show that this map is an isomorphism; clearly this implies that the two groups have the same Galois cohomology. 

It is harmless to rename $k_s$ as $k$ (since the map $k^{1/p}\otimes_k k_s \rightarrow k_s^{1/p}$ is an isomorphism), so we may and do assume that $k = k_s$. By \cite[Prop.\,7.3.3(1)]{cgp} (this is where we use that $k' = k^{1/p}$), the map $\mathscr{G}(k) \rightarrow \overline{G}(k)$ is an isomorphism, so the map $Z_{\mathscr{G}}(\mathscr{T})(k) \rightarrow \overline{T}(k)$ is injective. In order to show that it is surjective, we need to show that if $g \in \mathscr{G}(k)$ maps to $\overline{T}(k)$, then $g$ centralizes $\mathscr{T}$. Since $\mathscr{T}(k)$ is Zariski dense in $\mathscr{T}$ (because $k = k_s$), it suffices to show that $g$ centralizes $\mathscr{T}(k)$. Let $t \in \mathscr{T}(k)$. We need to show that $g$ centralizes $t$. We have $f(gtg^{-1}) = f(g)f(t)f(g)^{-1} = 1$ because $f(g), f(t) \in \mathscr{T}(k)$. Since $f$ is injective on $k$-points, this implies that $gtg^{-1} = 1$; that is, $g$ centralizes $t$.
\end{proof}

\begin{remark}
While the $k$-subgroup $Z_{\mathscr{G}}(\mathscr{T}) \subset f^{-1}(\overline{T})$ surjects onto $\overline{T}$, it is {\em not} true that $Z_{\mathscr{G}}(\mathscr{T}) = f^{-1}(\overline{T})$, since $\ker(f) \not \subset Z_{\mathscr{G}}(\mathscr{T})$, as can be seen using \cite[Rem.\,7.2.8]{cgp} over $k_s$.
\end{remark}

The definition of a {\em generalized standard} pseudo-reductive group is analogous to that of a standard pseudo-reductive group as follows. Let $k$ be a field, $k'/k$ a nonzero finite reduced $k$-algebra (that is, a finite product of finite field extensions of $k$). Let $G'$ be a $k'$-group, each of whose fibers over the factor fields of $k'$ are either connected semisimple, absolutely simple, and simply connected, or (this only happens when char$(k) \in \{2, 3\}$) are basic exotic pseudo-reductive. Let $T' \subset G'$ be a maximal $k'$-torus, and let $\mathscr{C} := \R_{k'/k}(Z_{G'}(T'))$ be the Weil restriction of the associated Cartan subgroup of $G'$. Then $\mathscr{C}$ acts on $\R_{k'/k}(G')$ by conjugation, and this action is trivial on the commutative group $\mathscr{C}$. 

Suppose we are given a commutative pseudo-reductive group $C$ acting on $\R_{k'/k}(G')$ so that the action of $C$ on $\mathscr{C}$ is trivial and a map $\phi: \mathscr{C} \rightarrow C$ such that the actions of $\mathscr{C}$ and $C$ on $\R_{k'/k}(G')$ are compatible via $\phi$. Let $j: \mathscr{C} \hookrightarrow \R_{k'/k}(G')$ denote the inclusion. One easily checks that the map $(-\phi, j): \mathscr{C} \rightarrow C \ltimes \R_{k'/k}(G')$ is an isomorphism onto a central subgroup of $C \ltimes \R_{k'/k}(G')$. Thus we may form the quotient $G := (C \ltimes \R_{k'/k}(G'))/\mathscr{C}$. Then $G$ is pseudo-reductive \cite[Prop.\,1.4.3]{cgp}.

\begin{definition}\cite[Definition 10.1.9]{cgp}
\label{generalized standard}
Let $k$ be a field, and assume that $[k: k^2] \leq 2$ if char$(k) = 2$. A {\em generalized standard pseudo-reductive group} is a pseudo-reductive group that is either commutative or constructed in the above manner.
\end{definition}

The assumption on $[k: k^2]$ when char$(k) = 2$ is made because the ``correct'' notion (given in \cite[Def.\,9.1.5, Def.\,9.1.7]{cp}) is different if char$(k) = 2$ and $[k: k^2] > 2$, even though the above definition makes perfectly good sense without this assumption. This assumption will not matter for us, since it holds for all local and global function fields. We have kept it, however, in order to be consistent with \cite{cgp}. 

Let us also remark that one can make the definition more analogous to the definition of standard groups, by noting that our assumptions on $\mathscr{C}$ and $C$ are equivalent to the existence of a factorization
\[
\mathscr{C} \xrightarrow{\phi} C \xrightarrow{\psi} Z_{\R_{k'/k}(G'), \mathscr{C}},
\]
where $Z_{\R_{k'/k}(G'), \mathscr{C}}$
is the maximal smooth $k$-subgroup scheme (see \cite[Lemma C.4.1]{cgp} for the notion of the maximal smooth subgroup) of the affine finite type $k$-group scheme $\underline{{\rm{Aut}}}_{\R_{k'/k}(G'), \mathscr{C}}$ which represents the functor of automorphisms of $\R_{k'/k}(G')$ that are trivial on $\mathscr{C}$. (See \cite[Lemma 10.1.7]{cgp} for the representability of this latter functor by an affine $k$-group scheme of finite type, and \cite[Thm.\,1.3.9]{cgp} for its identification with $\R_{k'/k}(T'/Z_{G'})$ in the standard setting.) We have avoided this alternative formulation, however, because we will not need it. If char$(k) > 2$ then every pseudo-reductive group is generalized standard \cite[Thm.\,5.1.1]{cgp}.

\begin{remark}
\label{choiceoftorus}
The standard and generalized standard constructions above depend upon beginning with a $k'$-group $G'$ whose fibers over the factor fields of $k'$ are each either (i) connected semisimple, absolutely simple, and simply connected, or (ii) basic exotic, and then making a choice of maximal $k'$-torus $T' \subset G'$ in order to carry out the construction. 

An important point is that having fixed the $k'$-group $G'$ (which, together with $k'/k$, is actually unique up to unique isomorphism due to \cite[Prop.\,10.2.4]{cgp}), we may choose {\em whichever} maximal $k'$-torus we want in the sense that any choice admits a canonically associated variant of $\phi: \mathscr{C} \rightarrow C$ also giving rise to $G$, by \cite[Prop.\,10.2.2(3)]{cgp}. This flexibility in choosing the torus $T'$ will be extremely useful to us, particularly in the proofs of Theorems \ref{tamagawaformula} and \ref{pseudoredcomplexexact}.
\end{remark}

Finally, there is one extra complication that can show up when char$(k) = 2$ (with $[k: k^2] = 2$; the situation is much more complicated when $[k: k^2] > 2$, but fortunately this never happens for local or global function fields): the possibility of ``basic non-reduced'' pseudo-reductive groups \cite[Def.\,10.1.2]{cgp}. We will not say much about the structure of these groups here, because we will not need them too often. For such $G$, 
with $n$ the dimension of its maximal tori, there is a surjective $k$-homomorphism
\begin{equation}\label{fggnonred}
f: G \twoheadrightarrow \overline{G}:= \R_{k^{1/2}/k}({\rm{Sp}}_{2n})
\end{equation}
for which the associated $k^{1/2}$-map $G_{k^{1/2}} \to {\rm{Sp}}_{2n}$ descends the maximal geometric reductive quotient over
$\overline{k}$ and satisfies some other properties that we shall refer to as they arise later, giving appropriate references to \cite{cgp}. 
This map $f$ will serve a role for us similar to that of the map (\ref{fggexotic}) in the basic exotic case. 
The term ``non-reduced'' refers to the fact that the root system of $G_{k_s}$ relative to a maximal $k_s$-torus is the non-reduced root system BC$_n$. 

Here is the main classification result for pseudo-reductive groups in \cite{cgp} allowing char$(k) = 2$. A more general structure theorem, without any restrictions on $[k: k^2]$ when char$(k) = 2$, is given by \cite[\S 1.6, Structure Theorem]{cp}. We will have no need for that more general result, since the condition $[k: k^2] = 2$ holds for all local and global function fields of characteristic $2$.

\begin{theorem}{\rm{\cite[Thm.\,5.1.1 and Prop.\,10.1.4]{cgp}}}
\label{predclassification}
Let $G$ be a pseudo-reductive group over a field $k$, and assume that $[k: k^2] \leq 2$ if ${\rm{char}}(k) = 2$. Then $G$ is generalized standard except possibly if $k$ is imperfect of characteristic $2$, and in the latter case, if $G$ is not generalized standard then $G = G_1 \times G_2$ where $G_1$ is a generalized standard pseudo-reductive group and $G_2 = \R_{K/k}(H)$ for a nonzero finite reduced $k$-algebra $K$ and a $K$-group $H$ whose fiber over each factor field is basic non-reduced.
\end{theorem}

\section{Tamagawa numbers of exotic groups: a general approach}
\label{sectionexoticgeneral}

The purpose of the present section and the next one is to show that basic non-reduced groups (these only exist in characteristic $2$) and basic exotic groups (these only exist in characteristics $2$ and $3$) have Tamagawa number $1$. This will play a crucial role in proving Theorems \ref{tamagawaformula} and \ref{innerinvariance} in these low characteristics, since such groups form some of the elemental building blocks of pseudo-reductive groups over such fields. The reader who does not mind ignoring characteristics $2$ and $3$ (or who simply first wishes to read the proof of Theorems \ref{tamagawaformula} and \ref{innerinvariance}) may skip this section and the next one.

The goal of this section and the next is to prove the following result.

\begin{proposition}
\label{basicexotictam=1}
Let $k$ be a global function field of characteristic $p \in \{2, 3\}$, and let $G$ be either a basic exotic $k$-group or a basic non-reduced $k$-group $($${\rm{char}}(k) = 2$ in the latter case$)$. Then $\tau(G) = 1$.
\end{proposition}

In order to prove Proposition \ref{basicexotictam=1}, we leverage the canonical map $f: G \twoheadrightarrow \overline{G}$
from (\ref{fggexotic}) and (\ref{fggnonred}). This map induces an isomorphism on $k$-points and a topological isomorphism on $k_v$-points and on $\A_k$-points \cite[Props.\,7.3.3, 9.9.4]{cgp}. Let $f_{\A}: G(\A) \xrightarrow{\sim} \overline{G}(\A)$ denote the isomorphism on adelic points induced by $f$. Let $\mu_G$ denote the Tamagawa measure on $G(\A)$, and similarly for $\mu_{\overline{G}}$. The group $G$ is its own
derived group: this follows from \cite[Thm.\,7.2.3(1)]{cgp} applied over $k_s$ in the basic exotic case,
and from \cite[Thm.\,9.8.1(4), Rem.\,9.8.2]{cgp} in the basic non-reduced case.
Hence, $\widehat{G}(k_s) = 0$, so $G$ is algebraically unimodular and thus $G(\A)$ is unimodular as a topological group \cite[Ch.\,I, \S 5.8]{oesterle}.
The measure $f^{-1}(\mu_{\overline{G}})$ is a positive multiple of $\mu_G$, since both are Haar measures on $G(\A)$. 

The Tamagawa number of $G$ is defined to be the volume of $G(k)\backslash G(\A)$ with respect to the quotient measure induced by $\mu_G$, and similarly for $\overline{G}$. We know $\tau(\overline{G}) = 1$ (by the now-established conjecture of Weil on Tamagawa numbers of simply connected groups \cite{lurie} for the basic exotic case, and the invariance of Tamagawa numbers under Weil restriction \cite[Chap.\,II, \S 1.3, Thm.\,(d)]{oesterle} together with the more classical evaluation $\tau({\rm{Sp}}_{2n})=1$ over global function fields 
for the basic non-reduced case). Thus, Proposition \ref{basicexotictam=1} is equivalent to the assertion that this multiplier is $1$; i.e., that $f^{-1}(\mu_{\overline{G}}) = \mu_G$. This is the content of the following proposition, which therefore implies Proposition \ref{basicexotictam=1}.

\begin{proposition}
\label{measuresagreebasic} 
Let $k$ be a global function field of characteristic $p \in \{2, 3\}$.
For $G$ as in Proposition $\ref{basicexotictam=1}$, 
via the isomorphism $f_{\A}: G(\A) \xrightarrow{\sim} \overline{G}(\A)$ the Tamagawa measure on $\overline{G}$ pulls back to the Tamagawa measure on $G$. That is, $f^{-1}(\mu_{\overline{G}}) = \mu_G$. In particular, $\tau(G) = 1$.
\end{proposition}

Since the two measures $f^{-1}(\mu_{\overline{G}})$ and $\mu_G$ are positive multiples of one another, in order to prove Proposition \ref{measuresagreebasic}, it would suffice to show that they agree on some nonempty open subset of $G(\A)$. This is roughly the idea, though not strictly what we will do. Rather, we will compare local measures on the $k_v$-points corresponding to the ``open cell'' inside $G$ coming from a suitable split $k$-torus $S \subset G$ and then use our knowledge of the Tamagawa numbers of the factors in the open cell decomposition to compare the corresponding product measures. In fact, one cannot hope to directly compare the measures by restricting to the adelic points of an open subscheme because given a smooth connected affine $k$-group $G$ and a nonempty affine open subscheme $U \subset G$, the subset $U(\A) \subset G(\A)$ is {\em not} generally open, and may even have measure $0$, and its natural topology may be distinct from the subspace topology inherited from $G(\A)$. (For example, take $G = \Ga$, $U = \Gm \subset \Ga$.) Nevertheless, we will be able to work on each local factor and then compute a product over all places to compare the two measures $f^{-1}(\mu_{\overline{G}})$ and $\mu_G$ on $G(\A)$.
In fact, we will reduce this comparison to the computation of the Tamagawa number of the centralizer of $S$ as in the following proposition.

\begin{proposition}
\label{toruscentralizertam=1}
Let $k$ be a global function field of characteristic $p \in \{2, 3\}$, and let $G$ be either a basic exotic $k$-group or a basic non-reduced $k$-group $($${\rm{char}}(k) = 2$ in the latter case$)$. Let $S \subset G$ be a split $k$-torus, and let $\overline{S} := f(S) \subset \overline{G}$ denote the associated split $k$-torus in $\overline{G}$. Suppose that $\tau(Z_G(S)) = \tau(Z_{\overline{G}}(\overline{S}))$. Then Proposition $\ref{measuresagreebasic}$ holds for $G$. In particular, $\tau(G) = 1$.
\end{proposition}

Note that if we take $S = 1$, then $Z_G(S) = G$ and $Z_{\overline{G}}(\overline{S}) = \overline{G}$, so we obtain a tautology. Thus, for Proposition \ref{toruscentralizertam=1} to be useful, we need a large split torus so that the group $Z_G(S)$ is something we already understand. Fortunately, over global function fields basic exotic groups and basic non-reduced groups always have relatively large split tori, as we shall see in \S \ref{exotictam=1section}. 

(In fact, as we shall see, basic exotic groups over global fields which are built out of simply connected groups of type F$_4$ or G$_2$ always have a split maximal torus. This is not generally true for groups built out of types B$_n$ or C$_n$, but these still have rather large $k$-rank over a global field $k$; see Proposition \ref{suitablesplittorus}. Basic non-reduced groups always have a split maximal $k$-torus \cite[Thm.\,9.9.3(1)]{cgp}.)

We now turn to the proof of Proposition \ref{toruscentralizertam=1}, which we will use to prove the equivalent Propositions \ref{basicexotictam=1} and \ref{measuresagreebasic} in \S \ref{exotictam=1section}. Note that the map $S \twoheadrightarrow \overline{S}$ is an isogeny, since $\ker(f)$ contains no nontrivial torus by \cite[Thms.\,1.6.2(2), 9.8.1(3), Prop.\,7.2.7(2)]{cgp}. The informal idea of the proof is to use ``open cell'' decompositions associated to $(G, S)$ and $(\overline{G}, \overline{S})$ in order to compare the Tamagawa measures on $G(\A)$ and $\overline{G}(\A)$ by comparing them on these ``open cells'', but strictly speaking this is not quite what we will do.

To define these open cells, we need to review some ``dynamic constructions''. 
Let $K$ be an arbitrary field, $H$ a linear algebraic $K$-group (i.e., a smooth affine $K$-group). Fix a cocharacter $\beta: \Gm \rightarrow H$ (possibly trivial). Then we define a subfunctor $U_H(\beta) \subset H$ by the formula
\[
U_H(\beta)(R) := \{h \in H(R) \mid \lim_{t\to 0} \beta(t)h\beta(t)^{-1} = 1\}
\]
for $K$-algebras $R$. Let us explain what we mean by this. The map $(\Gm)_R \rightarrow H_R$ defined by $t \mapsto \beta(t)h\beta(t)^{-1}$ may or may not extend to a map $c_{\beta}: \A^1_R \rightarrow H_R$. If it does, then it extends uniquely, since $H_R$ is separated over $R$. Then we define $\lim_{t \to 0} \beta(t)h\beta(t)^{-1} := c_{\beta}(0)$.

The subfunctor $U_H(\beta) \subset H$ is represented by a split smooth connected unipotent closed $K$-subgroup scheme of $H$ (which we also denote by $U_H(\beta)$) \cite[Lemma 2.1.5, Prop.\,2.1.10]{cgp}. Let $Z_H(\beta)$ denote the (scheme-theoretic) centralizer in $H$ of $\beta(\Gm)$.
Then the multiplication map
\begin{equation}\label{opencellH}
m_H: U_H(\beta^{-1}) \times Z_H(\beta) \times U_H(\beta) \rightarrow H
\end{equation}
is an open immersion \cite[Prop.\,2.1.8(3)]{cgp}. The image of this map is called the {\em open cell} in $H$ associated to $\beta$, or the {\em open cell} of the pair $(H, \beta)$.

We need to generalize our situation somewhat. Recall that in the basic exotic case over a {\em general} field $k$ of characteristic $p \in \{2, 3\}$, $G$ is constructed from a triple of data $(k'/k, G', \overline{G})$, where $k'/k$ is a nontrivial finite extension contained in $k^{1/p}$, $G'$ is a connected semisimple $k'$-group that is absolutely simple and simply connected with Dynkin diagram containing an edge of multiplicity $p$, and $\overline{G} \subset \R_{k'/k}(\overline{G}')$ is a Levi $k$-subgroup with $\overline{G}'$ the very special quotient of $G'$. For $k$ of characteristic $2$ with $[k: k^2] = 2$, a basic non-reduced group $G$ over $k$ depends only on a positive integer $n$ (and $G$ is then pseudo-split with root system ${\rm{BC}}_n$) \cite[Thm.\,9.9.3(1)]{cgp}.

Let $k$ be an {\em arbitrary} imperfect field of characteristic $p \in \{2, 3\}$ such that $k' := k^{1/p}$ is a {\em finite} extension of $k$. In the basic exotic case, let $G$ denote the basic exotic group built out of a triple $(k'/k, G', \overline{G})$, so $f: G \twoheadrightarrow \overline{G}$ is the quotient map onto a simply connected semisimple $k$-group. In the basic non-reduced case in characteristic $2$ with $k' = k^{1/2}$ quadratic over $k$, let $G$ be the basic non-reduced $k$-group with root system of type ${\rm{BC}}_n$ (such $G$ is unique up to isomorphism). This data behaves well with respect to suitable separable extensions as follows. In the basic exotic case, if $K/k$ is a separable extension field such that $k^{1/p} \otimes_k K \rightarrow K^{1/p}$ is an isomorphism (e.g. $K = k_v$ for a global field $k$), then $G_K$ is the basic exotic group built from the data $(K'/K, G'_{K'}, \overline{G}_K)$. In the basic non-reduced case, under the same hypothesis $G_K$ is the basic non-reduced $K$-group with root system ${\rm{BC}}_n$. 

Now we discuss the ``open cell'' associated to a split torus $S \subset G$ (really to a suitable cocharacter of that torus) in this general setting. Let $S \subset G$ be a split $k$-torus, and $\overline{S} := f(S) \subset \overline{G}$ the associated split torus in $\overline{G}$. We may choose a cocharacter $\beta: \Gm \rightarrow S$ so that $Z_G(\beta) = Z_G(S)$ (and then automatically $Z_{\overline{G}}(\overline{\beta}) = Z_{\overline{G}}(\overline{S})$, where $\overline{\beta} := f \circ \beta$). Indeed, this amounts to choosing $\beta$ to be ``sufficiently generic''; namely, one chooses $\beta$ that is
not killed by any of the finitely many nontrivial $S$-weights for the adjoint action of $S$ on Lie$(G)$. 

Let us denote the groups $U_G(\beta)$ and $U_G(\beta^{-1})$ by $U^+$ and $U^-$ respectively, and define $\overline{U}^+, \overline{U}^- \subset \overline{G}$ similarly, so the $k$-groups $U^+, U^-, \overline{U}^+, \overline{U}^-$ are split unipotent. As instances of (\ref{opencellH}), the multiplication maps
\[
m_G: U^- \times Z_G(S) \times U^+ \rightarrow G
\]
\[
m_{\overline{G}}: \overline{U}^- \times Z_{\overline{G}}(\overline{S}) \times \overline{U}^+ \rightarrow \overline{G}
\]
are open immersions. We refer to these open subschemes of $G$ and $\overline{G}$ as the respective
{\em open cells} in $G$ and $\overline{G}$ (they typically depend on a choice of cocharacter $\beta: \Gm \rightarrow S$, though for maximal split $S$ one can give descriptions in terms of minimal pseudo-parabolic $k$-subgroups containing $S$
\cite[C.2.2--C.2.5]{cgp}).

The maps $U^- \rightarrow \overline{U}^-$, $U^+ \rightarrow \overline{U}^+$, $Z_G(S) \rightarrow Z_{\overline{G}}(\overline{S})$ induced by $f$
are surjective \cite[Cor.\,2.1.9]{cgp}.  Here is a refinement, an analogue for these subgroups of the results \cite[Props.\,7.3.3, 9.9.4]{cgp} at the level of the entire group:

\begin{lemma}
\label{opencellpoints}
Let $k$ be an imperfect field of characteristic $p \in \{2, 3\}$ such that $[k : k^p]$ is finite, and let $G$ be either a basic exotic $k$-group built of the data $(k'/k, G', \overline{G})$ with $k' := k^{1/p}$ or else suppose that $p = 2$, $[k: k^2] = 2$, and that $G$ is a basic non-reduced $k$-group. 

Let $f: G \twoheadrightarrow \overline{G}$ denote the canonical map onto a simply connected quotient (see Remark $\ref{canonicalsurjection}$). Let $\beta: \Gm \rightarrow G$ be a cocharacter, $\overline{\beta} := f \circ \beta: \Gm \rightarrow \overline{G}$, and let $Z := Z_G(\beta)$, $U^+ := U_G(\beta)$, $U^- := U_G(\beta^{-1})$ and define $\overline{Z}, \overline{U}^+, \overline{U}^- \subset \overline{G}$ analogously using $\overline{\beta}$.
\begin{itemize}
\item[(i)] The maps $Z(k) \rightarrow \overline{Z}(k)$, $U^+(k) \rightarrow \overline{U}^+(k)$, and $U^-(k) \rightarrow \overline{U}^-(k)$ induced by $f$ are isomorphisms.
\item[(ii)] If $k$ is complete with respect to an absolute value, then the maps in (i) are topological isomorphisms.
\item[(iii)] If $k$ is a global function field then the maps $Z(k_v) \rightarrow \overline{Z}(k_v)$, $U^+(k_v) \rightarrow \overline{U}^+(k_v)$, and $U^-(k_v) \rightarrow \overline{U}^-(k_v)$ are topological isomorphisms, as are the maps $Z(\A_k) \rightarrow \overline{Z}(\A_k)$, $U^+(\A_k) \rightarrow \overline{U}^+(\A_k)$, and $U^-(\A_k) \rightarrow \overline{U}^-(\A_k)$.
\end{itemize}
\end{lemma}

\begin{proof}
First we note that the topological assertions are all automatic once we prove that the maps are set-theoretic bijections. For example, suppose that the map $Z(k) \rightarrow \overline{Z}(k)$ is a bijection. If $k$ is complete with respect to an absolute value, then since the topology on $Z(k)$ is the subspace topology inherited from $G(k)$, and likewise for $\overline{Z}(k) \subset \overline{G}(k)$, the fact that $G(k) \rightarrow \overline{G}(k)$ is a homeomorphism \cite[Prop.\,7.3.3(ii), 9.9.4(2)]{cgp} will then imply that the map $Z(k) \rightarrow \overline{Z}(k)$ is a topological isomorphism. Similarly, in (iii) one merely has to show that the maps are bijections (since a closed immersion $Y \hookrightarrow X$ of affine $k$-schemes of finite type induces a closed embedding of topological spaces $Y(\A) \rightarrow X(\A)$).

Let us also observe that all of the maps in (i)--(iii) are injective because the analogous maps for the entire group are bijections by \cite[Props.\,7.3.3, 9.9.4]{cgp}. Let $R := \beta(\Gm)$, a split torus in $G$, and let $\overline{R} := f(R) = \overline{\beta}(\Gm) \subset \overline{G}$.

We first prove (i) and (ii). We may replace $k$ with a finite Galois extension $K$ and thereby assume that $G$ contains a split maximal torus $T$ containing $R$. Indeed, this preserves all hypotheses because $k^{1/p} \otimes_k K \xrightarrow{\sim} K^{1/p}$, and showing that the maps are bijective on $K$-points implies the same on $k$-points by taking Galois invariants.

We first prove that the map $Z(k) \rightarrow \overline{Z}(k)$ is surjective. Suppose that $\overline{z} \in \overline{Z}(k)$. By \cite[Props.\,7.3.3, 9.9.4]{cgp}, there exists a unique $g \in G(k)$ such that $f(g) = \overline{z}$. We need to show that $g \in Z(k)$. That is, we need to show that $g$ centralizes $R := \beta(\Gm)$. Since $R$ is smooth, its $k_s$-points are Zariski dense, so by replacing $k$ with $k_s$, it suffices to show that $g$ centralizes $R(k)$. So let $r \in R(k)$. Then $f(grg^{-1}) = f(g)f(r)f(g)^{-1} = 1$ because $f(g)$ centralizes $f(r) \in \overline{R}$. But $f$ induces an inclusion on $k$-points, so $grg^{-1} = 1$. That is, $g$ centralizes $r$.

Next we show that the maps on $k$-points of the ``positive'' and ``negative'' unipotent groups are bijections. For the chosen split maximal torus $T \subset G$ containing $R$, let $\overline{T} := f(T)$ be the associated split maximal torus of $\overline{G}$. Let $\Phi := \Phi(G, T)$ and $\overline{\Phi} := \Phi(\overline{G}, \overline{T})$ denote the corresponding root systems. 

Let us first treat the basic exotic case. The map $f$ induces a bijection between the dual root systems $\Phi$ and $\overline{\Phi}$ that switches long and short roots. Indeed, let $\Phi_>$ and $\Phi_<$ denote the long and short roots of $\Phi$, respectively, and similarly for $\overline{\Phi}$. For $\overline{a} \in \overline{\Phi}_<$, the root $a := \overline{a} \circ f$ of $T$ lies in $\Phi_>$, while for $\overline{a} \in \overline{\Phi}_>$, we have $\overline{a} \circ f = pa$ for a root $a \in \Phi_<$, and this correspondence yields a bijection between the roots of $\Phi$ and those of $\overline{\Phi}$ \cite[Prop.\,7.1.5]{cgp}.

By \cite[Cor.\,3.3.12]{cgp}, multiplication induces an isomorphism of $k$-schemes (not necessarily of $k$-groups)
\begin{equation}
\label{k-schemeU^+isomeqn1}
\prod_{\substack{a \in \Phi \\ \langle a, \beta \rangle > 0}} U_a \rightarrow U^+
\end{equation}
for {\em any} ordering of the elements of $\{ a \in \Phi \mid \langle a, \beta \rangle > 0\}$, where $U_a$ denotes the $a$-root group for $(G, T)$. Taking the corresponding ordering of $\{ \overline{a} \in \overline{\Phi} \mid \langle \overline{a}, \overline{\beta} \rangle > 0\}$ similarly yields a $k$-scheme isomorphism
\begin{equation}
\label{k-schemeU^+isomeqn2}
\prod_{\substack{\overline{a} \in \overline{\Phi} \\ \langle \overline{a}, \overline{\beta} \rangle > 0}} \overline{U}_{\overline{a}} \rightarrow \overline{U}^+
\end{equation}
(where $\overline{U}_{\overline{a}} \subset \overline{G}$ is the $\overline{a}$-root group), and the isomorphisms (\ref{k-schemeU^+isomeqn1}) and (\ref{k-schemeU^+isomeqn2}) are compatible via $f$. It therefore suffices to show that $f$ induces bijections $U_a(k) \rightarrow \overline{U}_{\overline{a}}(k)$ for each $a \in \Phi$ (where $a = \overline{a} \circ f$).

If $\pi': G' \rightarrow \overline{G}'$ denotes the very special isogeny, then $\ker(\pi')$ is a height-$1$ normal $k'$-subgroup scheme whose Lie algebra intersects ${\rm{Lie}}(U'_a)$ trivially if $a \in \Phi$ is long and contains ${\rm{Lie}}(U'_a)$ if $a$ is short. If $a$ is a long root, therefore, then $f$ induces an isomorphism $U_a \rightarrow \overline{U}_{\overline{a}}$, whereas if $a$ is short, then the map $U_a = \R_{k'/k}(U'_a) \rightarrow \overline{U}_{\overline{a}}$ induced by $f$ is identified with the map $\R_{k'/k}(\Ga) \rightarrow \Ga$ that on $k$-points is the $p$th power map $k' = k^{1/p} \rightarrow k$ (see \cite[Prop.\,7.1.5(1), (2)]{cgp}). This latter map is clearly an isomorphism. This proves (i) and (ii) for the map $U^+(k) \rightarrow \overline{U}^+(k)$ when $G$ is basic exotic. The proof for $U^-(k) \rightarrow \overline{U}^-(k)$ is exactly the same. 

The proof for basic non-reduced groups is similar, once again breaking up the positive and negative unipotent groups as a product of root groups to reduce the assertion to proving that one obtains isomorphisms between rational points of root groups. By \cite[Thm.\,9.8.1(3)]{cgp}, the map induced by $f$ between each pair of root groups is either an isomorphism or else is the map 
\[
\R_{k'/k}(\Ga) \times \Ga \rightarrow \R_{k'/k}(\Ga)
\]
defined by $(x, y) \mapsto \alpha x^2 + y$ for some $\alpha \in k' - k$. Since $[k': k] = 2$, this map is a bijection on $k$-points.

To prove (iii), we first note that the assertion about $k_v$-points follows from (ii) once we note that all hypotheses are preserved when we extend scalars from $k$ to $k_v$, since $k_v \otimes_k k^{1/p} \xrightarrow{\sim} k_v^{1/p}$. The proof that the map on $\A_k$-points is bijective is exactly the same as the proof of (i), since $H(\A_k) = H(\A_L)^{{\rm{Gal}}(L/k)}$ for any finite Galois extension $L/k$ and any finite type affine $k$-scheme $H$.
\end{proof}

Now we return to the setting when $k$ is a global function field. Before continuing further, let us recall precisely how the Tamagawa measure and Tamagawa number of a connected linear algebraic group are defined. For a more detailed discussion, see \cite[Chap.\,I]{oesterle}. Let $H$ be a connected linear algebraic over the global function field $k$. Choose a nonzero top-degree left-invariant differential form $\omega_H$ on $H$. For each place $v$ of $k$, let $\mu_v$ denote the Haar measure on $k_v$ which gives $\calO_v$ measure $1$. Then $\omega_H$ together with $\mu_v$ induces a left Haar measure $\mu_H^v$ on $H(k_v)$.

Consider the continuous finite-dimensional $\C$-linear representation $V := \C \otimes \widehat{H}(k_s)$
of ${\rm{Gal}}(k_s/k)$. Then we have for each place $v$ of $k$ the local L-factor L$_v(\widehat{H}, s)$ associated to this representation:
\[
{\rm{L}}_v(\widehat{H}, s) := {\rm{det}}(1 - q_v^{-s}F_v\mid V^{I_v})
\]
where $q_v$ is the size of the residue field of $\calO_v$, $F_v$ is a lift of the conjugacy class of any Frobenius element at $v$, and $I_v$ is the inertia group of $v$. Then the product $L(\widehat{H}, s) := \prod_v {\rm{L}}_v(\widehat{H}, s)$ is an Artin L-function, and it converges for Re$(s) \gg 0$ and extends to a meromorphic function on all of $\C$ with a pole at $s = 1$ of order $r := {\rm{dim}}_{\C} V^{{\rm{Gal}}(k_s/k)} = {\rm{rk}}(\widehat{H}(k))$. Let $\rho_H := \lim_{s \to 1} (s-1)^rL(\widehat{H}, s)$.

For $a := (a_v) \in \A^{\times}$, define $|\!|a|\!| := \prod_v |a_v|_v$. Note that $|\!|a|\!| \in q^{\Z}$, where $q$ is the size
of the finite constant field of $k$. In order to define the Tamagawa number of $H$, note first that any $\chi \in \widehat{H}(k)$ induces a map $H(\A) \rightarrow \A^{\times}$, whence a map $|\!|\chi|\!|: H(\A) \rightarrow q^{\Z} \subset \mathbf{R}_{>0}$. In this way, we obtain a continuous map
\[
\theta_H: H(\A) \rightarrow \Hom(\widehat{H}(k), q^{\Z})
\]
whose image has finite index in $\Hom(\widehat{H}(k), q^{\Z})$. Let $H(\A)_1 := \ker(\theta_H)$. Note that $H(\A)_1$ is an open subset of $H(\A)$, hence we may restrict a left Haar measure on $H(\A)$ to $H(\A)_1$, and the measure is determined by this restriction.

Then the Tamagawa measure on $H(\A)$ is defined to be the product measure
\begin{equation}
\label{Tammeasuredef}
\mu_H := \frac{(q^{1-g})^{{\rm{dim}}(H)} \rho_H^{-1}} {({\rm{log}}(q))^{{\rm{rk}}(\widehat{H}(k))} \cdot \# \coker (\theta_H)} \prod_v \left({\rm{L}}_v(\widehat{H}, 1) \cdot \mu_H^v\right),
\end{equation}
where $k$ is the function field of the smooth proper geometrically connected genus-$g$ curve $X$ over $\F_q$. This definition is independent of our choice $\omega_H$ of nonzero left-invariant top-degree form on $H$, since any other such form is of the form $\lambda \cdot \omega_H$ for some $\lambda \in k^{\times}$ (so the two resulting measures have ratio $\prod_v |\lambda|_v = 1$ by the product formula). The infinite product measure on the right side of (\ref{Tammeasuredef}) converges absolutely, so that the order of the product is irrelevant. By this absolute convergence we mean the following. Spread $H$ out to a smooth affine $\calO_{\Sigma}$-model $\mathscr{H}$ for some finite nonempty set $\Sigma$ of places of $k$. Then the infinite product
\[
\prod_{v \notin \Sigma} {\rm{L}}_v(\widehat{H}, 1) \cdot \mu_H^v(\mathscr{H}(\calO_v))
\]
converges absolutely. This follows from \cite[Ch.\,I, \S 1.7 and \S 2.5, Prop.\,]{oesterle}.
The {\em Tamagawa number} $\tau(H)$ of $H$ is defined to be the volume of $H(k)\backslash H(\A)_1$ with respect to the quotient measure induced by the measure $\mu_H$ on the unimodular group $H(\A)_1$. ($H(\A)_1$ is unimodular by \cite[Ch.\,I, \S 5.8]{oesterle}.)

This definition simplifies greatly when $\widehat{H}(k_s) = 0$: in that case, we have
\begin{equation}
\label{Tammeasuredefnochars}
\mu_H = (q^{1-g})^{{\rm{dim}}(H)} \prod_v \mu_H^v,
\end{equation}
and $\tau(H)$ is the volume of $H(k)\backslash H(\A)$ with respect to the quotient measure induced by the measure $\mu_H$ on the unimodular group $H(\A)$.

Before proceeding further, we require the following lemma.

\begin{lemma}
\label{productmeasurepullback}
Let $X, Y$ be affine schemes of finite type over a global field $k$. Let $g: X \rightarrow Y$ be a $k$-scheme morphism that induces a homeomorphism $g_{\A}: X(\A_k) \xrightarrow{\sim} Y(\A_k)$. Assume that $X(\A_k) \neq \emptyset$. Suppose given for each place $v$ of $k$ a Borel measure $\mu_v$ on $Y(k_v)$ such that the restricted product measure $\prod_v \mu_v$ on $Y(\A_k)$ is absolutely convergent. For each place $v$ of $k$, let $g_v: X(k_v) \rightarrow Y(k_v)$ denote the natural map induced by $g$. 

Then $g_v$ is a homeomorphism for all $v$, the restricted product measure $\prod_v g_v^{-1}(\mu_v)$ on $X(\A_k)$ is absolutely convergent, and 
\[
g_{\A}^{-1}\left(\prod_v \mu_v\right) = \prod_v g_v^{-1}(\mu_v)
\]
as measures on $X(\A_k)$.
\end{lemma}

\begin{proof}
For some finite nonempty set $\Sigma$ of places of $k$, we can choose affine $\calO_{\Sigma}$-schemes $\mathscr{X}$ and $\mathscr{Y}$ of finite type with respective generic fibers $X$ and $Y$ such that $g$ spreads out to a map of $\calO_{\Sigma}$-schemes $g': \mathscr{X} \rightarrow \mathscr{Y}$. Both assertions of the lemma follow immediately from the definition of restricted product measures once we show that $g_v$ is a homeomorphism for all $v$ and $\mathscr{X}(\calO_v) \rightarrow \mathscr{Y}(\calO_v)$ is bijective for all but finitely many $v$. Since $g_{\A}$ is a homeomorphism and $X(\A) \neq \emptyset$, each $g_v$ is clearly a homeomorphism. Since $\mathscr{X}, \mathscr{Y}$ are affine, therefore, $g'$ induces inclusions $\mathscr{X}(\calO_v) \hookrightarrow \mathscr{Y}(\calO_v)$ for all $v \notin \Sigma$, and so we only need to show that these maps are surjective for all but finitely many $v$.

In order to show this, choose $x \in X(\A)$. (There exists such $x$ because $X(\A) \neq \emptyset$ by assumption.) The point $x_v \in X(k_v)$ lies in $\mathscr{X}(\calO_v)$ for all but finitely many $v \notin \Sigma$, so we can choose some finite set $\Sigma'$ of places of $k$ that contains $\Sigma$ 
such that $x \in W := \prod_{v \in \Sigma'} X(k_v) \times \prod_{v \notin \Sigma'} \mathscr{X}(\calO_v)$. Since $g_{\A}$ is a homeomorphism, $g_{\A}(W) \subset Y(\A)$ is nonempty open, so 
\[
g(x) \in \prod_{v \in \Sigma''} U_v \times \prod_{v \notin \Sigma''} \mathscr{Y}(\calO_v) \subset g_{\A}(W)
\]
for some finite set $\Sigma''$ of places of $k$ containing $\Sigma$ and some nonempty open subsets $U_v \subset Y(k_v)$ $(v \in \Sigma''$). It follows that $\mathscr{Y}(\calO_v) \subset g'(\mathscr{X}(\calO_v))$ for $v \notin \Sigma' \cup \Sigma''$.
\end{proof}

Now fix nonzero top-degree left-invariant differential forms $\omega_{Z_G(S)}$, $\omega_{U^-}$, $\omega_{U^+}$, $\omega_{Z_{\overline{G}}(\overline{S})}$, $\omega_{\overline{U}^-}$, $\omega_{\overline{U}^+}$ on the groups $Z_G(S)$, $U^-$, $U^+$, $Z_{\overline{G}}(\overline{S})$, $\overline{U}^-$, $\overline{U}^+$, respectively. Let $f_{+, v}: U^+(k_v) \xrightarrow{\sim} \overline{U}^+(k_v)$ denote the topological isomorphism induced by $f$, and similarly for $f_{-, v}: U^-(k_v) \xrightarrow{\sim} \overline{U}^-(k_v)$ and $f_{Z, v}: Z_G(S)(k_v) \xrightarrow{\sim} Z_{\overline{G}}(\overline{S})(k_v)$.

Since any two left Haar measures on a locally compact group are positive multiples of one another, we have for each place $v$ of $k$
\[
f_{+, v}^{-1}(\mu_{\overline{U}^+}^v) = \lambda_{+, v} \cdot \mu_{U^+}^v,
\]
\[
f_{-, v}^{-1}(\mu_{\overline{U}^-}^v) = \lambda_{-, v} \cdot \mu_{U^-}^v,
\]
\begin{equation}
\label{lambda_vdef}
f_{Z, v}^{-1} (\mu_{Z_{\overline{G}}(\overline{S})}^v) = \lambda_{Z, v} \cdot \mu_{Z_G(S)}^v.
\end{equation}
for some constants $\lambda_{+, v}$, $\lambda_{-, v}$, $\lambda_{Z, v} > 0$. We first analyze the constants $\lambda_{\pm, v}$.

\begin{lemma}
\label{prodlambda_pmv}
We have 
\[
\prod_v \lambda_{+, v} = (q^{1-g})^{{\rm{dim}}(U^+) - {\rm{dim}}(\overline{U}^+)},
\]
\[
\prod_v \lambda_{-, v} = (q^{1-g})^{{\rm{dim}}(U^-) - {\rm{dim}}(\overline{U}^-)},
\]
and the products converge absolutely.
\end{lemma}

\begin{proof}
These equations are basically just saying that $\tau(U^+) = \tau(\overline{U}^+)$ and $\tau(U^-) = \tau(\overline{U}^-)$, and both of these equalities hold because all of these groups are split unipotent, hence have Tamagawa number $1$.

Indeed, let us treat the ``$+$'' equality, the other being exactly the same but replacing ``$+$'' everywhere with ``$-$''. Since $f$ induces a topological isomorphism $f_{+, \A}: U^+(\A) \xrightarrow{\sim} \overline{U}^+(\A)$ that restricts to a bijection on $k$-points, the assertion that $\tau(U^+) = \tau(\overline{U}^+)$ is the same as saying that $f_{+, \A}^{-1}(\mu_{\overline{U}^+}) = \mu_{U^+}$. By using the definition (\ref{Tammeasuredefnochars}) of Tamagawa measure for groups with no nontrivial $k_s$-characters (such as unipotent groups), the definition (\ref{lambda_vdef}) of the $\lambda_{+, v}$, Lemma \ref{productmeasurepullback}, and the fact that the Tamagawa measure converges absolutely, we see that this in turn is equivalent to the first equality in the lemma.
\end{proof}

Next we turn to a study of the constants $\lambda_{Z, v}$. This is somewhat more complicated because the groups $Z_G(S)$ and $Z_{\overline{G}}(\overline{S})$ in general may have nontrivial characters. The key point is to identify their character groups. We turn to this task now.

\begin{lemma}
\label{Rk'/kmupchars}
Let $F'/F$ be a finite nontrivial purely inseparable extension of fields of characteristic $p$. Then $\R_{F'/F}(\mu_p)$ has no nontrivial $F$-characters.
\end{lemma}

\begin{proof}
We first show that the cokernel of the canonical inclusion $i: \mu_p \hookrightarrow \R_{F'/F}(\mu_p)$ is unipotent. We have an inclusion $\R_{F'/F}(\mu_p)/\mu_p \hookrightarrow \R_{F'/F}(\Gm)/\Gm$, so it suffices to show that $\R_{F'/F}(\Gm)/\Gm$ is unipotent. But this group is $p$-power torsion, since for any $F$-algebra $A$, 
$$[p^n](\R_{F'/F}(\Gm)(A)) = ((A \otimes _F F')^{\times})^{p^n} \subset A^{\times} = \Gm(A)$$
if $n$ is chosen so large that $(F')^{p^n} \subset F$. Since the
smooth connected affine group 
$$\R_{F'/F}(\Gm)/\Gm$$
is killed by a power of $p$, it is unipotent.

Let $U := \R_{F'/F}(\mu_p)/\mu_p$, so $\widehat{U}(F) = 0$ because unipotent groups have no nontrivial characters over a field. It follows that the map $$\widehat{i}: \widehat{\R_{F'/F}(\mu_p)}(F) \rightarrow \widehat{\mu_p}(F) = \Z/p\Z$$ is an inclusion, so it suffices to show that this map is $0$. Suppose to the contrary that it is not $0$, so it is surjective. 
Any homomorphism $\R_{F'/F}(\mu_p) \to \Gm$ lands in $\mu_p$,
and now the character of $\mu_p$ given by the inclusion $\mu_p \hookrightarrow \Gm$ lies in the image of $\widehat{i}$,
so $i$ has a section.

This means that we have an $F$-group splitting $\R_{F'/F}(\mu_p) = \mu_p \times U$. Thus we get an inclusion $U \hookrightarrow \R_{F'/F}(\mu_p)$, which via the universal property of Weil restriction corresponds to an
$F'$-homomorphism $U_{F'} \rightarrow \mu_p$. Since $U_{F'}$ is unipotent, this
latter map must vanish. Hence, the inclusion of $U$ into $\R_{F'/F}(\mu_p)$ vanishes,
so $U=1$ and thus $i$ is an isomorphism. This is false because $i$ isn't an equality on $F'$-points: 
$\mu_p(F') = 1$ but $\R_{F'/F}(\mu_p)(F') = \mu_p(F' \otimes_F F') \neq 1$ because $F' \otimes_F F'$ is non-reduced.
\end{proof}

Now we may identify the character groups of $Z_G(S)$ and $Z_{\overline{G}}(\overline{S})$.

\begin{lemma}
\label{ZGcharssame}
Let $f_Z: Z_G(S) \twoheadrightarrow Z_{\overline{G}}(\overline{S})$ denote the map induced by $f$. Then the maps on character groups $\widehat{Z_{\overline{G}}(\overline{S})}(k) \rightarrow \widehat{Z_G(S)}(k)$ and $\widehat{Z_{\overline{G}}(\overline{S})}(k_s) \rightarrow \widehat{Z_G(S)}(k_s)$ induced by $f_Z$ are isomorphisms.
\end{lemma}

\begin{proof}
It suffices to prove the assertion about characters over $k_s$, since the result over $k$ then follows by taking Galois-invariants. Let $H := \ker(f_Z) = Z_{\ker(f)}(S)$, so we have an exact sequence
\[
1 \longrightarrow H \longrightarrow Z_G(S) \xlongrightarrow{f_Z} Z_{\overline{G}}(\overline{S}) \longrightarrow 1.
\]
This induces an exact sequence
\[
0 \longrightarrow \widehat{Z_{\overline{G}}(\overline{S})}(k_s) \xlongrightarrow{\widehat{f_Z}} \widehat{Z_G(S)}(k_s) \longrightarrow \widehat{H}(k_s),
\]
so it suffices to show that $\widehat{H}(k_s) = 0$.

First suppose that $G$ is basic non-reduced. Let $T$ be a maximal $k_s$-torus containing $S$. Then $(\ker f)_{k_s}$ is directly spanned (under multiplication) by unipotent group schemes normalized by $T_{k_s}$, 
due to \cite[Thm.\,9.8.1(4)]{cgp} (since in the notation of that result, $f$ is the map $\xi_G$ that is identified with $i_G$
by \cite[Thm.\,9.8.1(3)]{cgp}).  Hence, $H = (\ker f) \cap Z_G(S)$ is directly spanned over $k_s$ by the $S_{k_s}$-centralizers of
those unipotent group schemes, so $H_{k_s}$ is also directly spanned by unipotent group schemes. 
The group scheme $H_{k_s}$ therefore has no nontrivial characters, as desired.

Now suppose that $G$ is basic exotic. Recall that for $k' := k^{1/p}$ we have $\ker(f) = \ker(\R_{k'/k}(\pi'))$, where $\pi': G' \twoheadrightarrow \overline{G}'$ is the very special isogeny associated to the simply connected group $G'$ (and $G$ is the preimage of a Levi $k$-subgroup $\overline{G} \subset \R_{k'/k}(\overline{G}')$). 
Since $\ker \pi'$ lies inside the kernel of the relative Frobenius isogeny $F_{G'/k'}:G' \to {G'}^{(p)}$, 
which in turn is killed by the $p$-power map
(as this is inherited from an elementary computation for ${\rm{GL}}_N$ and its relative Frobenius isogeny), it follows
that $\ker(f)$ is killed by the $p$-power map and hence the same holds for $H = (\ker f) \cap Z_G(S)$.

We want to show that any $k_s$-homomorphism $\chi:H_{k_s} \to \Gm$ is trivial. There is no harm in checking the analogous
vanishing over any finite separable extension of $k$, which in turn we can increase further and rename as $k$
to arrive at the case that $\chi$ is a $k$-homomorphism and $G$ admits a split maximal $k$-torus $T \subset G$
containing $S$.  Let $T'$ be the corresponding maximal $k'$-torus $T_{k'} \subset G_{k'}^{\rm{red}} = G'$,
which identifies $\Phi := \Phi(G, T)$ with $\Phi(G', T')$ via the equality $\widehat{T}(k) = \widehat{T'}(k)$
\cite[Prop.\,7.2.7(1),(2)]{cgp}. 

It suffices to show that an open subscheme 
$\Omega$ of $H$ containing the identity point is contained in $\ker(\chi)$. Indeed, by translation
over $\overline{k}$ it then follows that $H^0 \subset \ker(\chi)$, so
$\chi$ factors through the finite \'etale $k$-group $H/H^0$. 
But $H$ is $p$-torsion, so the same holds for the \'etale quotient $H/H^0$.
Hence, this quotient has no nontrivial characters over $k_s$ since $\Gm[p] = \mu_p$ is
infinitesimal, so $\chi = 1$ as desired.  Thus, it suffices to show that $\chi$ vanishes on an open subscheme of $H$ containing 1.
 
By definition, $\ker(\pi')$ is the $k'$-subgroup scheme of $G'$ of height 1 with $p$-Lie algebra inside $\mathfrak{g}'$ 
given by 
\[
{\rm{Lie}}(\ker \pi') = \left(\bigoplus_{b \in \Delta_<} {\rm{Lie}}(b^{\vee}(\Gm))\right) \bigoplus \left(\bigoplus_{b \in \Phi_<} \mathfrak{g}'_b\right),
\]
where $\mathfrak{g}'_b$ is the $b$-weight space for the adjoint action of $T'$ on $\mathfrak{g}'$, 
$\Phi_<$ denotes the set of short roots, and $\Delta_< = \Delta \cap \Phi_<(G', T')$ for some basis $\Delta$ of the root system $\Phi$.

Multiplication yields a $T$-equivariant (with respect to the $T$-conjugation action) open immersion
\[
\prod_{b \in \Delta_<} \R_{k'/k}(b^{\vee}(\mu_p)) \times \prod_{b \in \Phi_<} \R_{k'/k}(\ker({\rm{F}}_{U_b/k})) \hookrightarrow \ker(f),
\]
where $\Ga \simeq U_b \subset G'$ is the $b$-root group (i.e., the maximal smooth connected $k'$-subgroup of $G'$ normalized by $T'$ for which the adjoint action of $T'$ is the character $b$).  The $S$-centralizers are exactly the scheme-theoretic intersections
with $Z_G(S)$, so the induced map between $S$-centralizers is an open immersion.  We want to compute this latter map.

The $S$-centralizer of the right side is $(\ker f) \cap Z_G(S) = H$, and on the left side the $S$-centralizer may be computed factor by factor (since each factor is preserved by $T$-conjugation).
The first product on the left side clearly commutes with the $S$-action (since even $T$-conjugation on it is trivial).
In the second product, the factors are unipotent group schemes, hence have unipotent $S$-centralizers. (In fact, the factors indexed by roots that are trivial on $S$ have trivial $S$-action and 
hence are their own $S$-centralizers, while the other factors have trivial $S$-centralizers, but we do not need this.)

We now have an open immersion
\[
\prod_{\substack{b \in \Delta_<}} \R_{k'/k}(b^{\vee}(\mu_p)) \times \prod_{\substack{b \in \Phi_<}} \left( Z_G(S) \cap \R_{k'/k}(\ker({\rm{F}}_{U_b/k})) \right) \hookrightarrow H, 
\]
so it suffices to show that $\chi$ vanishes on this open subscheme of $H$. Such vanishing is clear on the unipotent groups $Z_G(S) \cap \R_{k'/k}(\ker({\rm{F}}_{U_b/k}))$, and the vanishing on each $\R_{k'/k}(b^{\vee}(\mu_p))$ follows from Lemma \ref{Rk'/kmupchars}.
\end{proof}

We are now ready to prove the analogue of Lemma \ref{prodlambda_pmv} for the $\lambda_{Z, v}$ under the assumption that $\tau(Z_G(S)) = \tau(Z_{\overline{G}}(\overline{S}))$.

\begin{lemma}
\label{prodlambdazv}
If $\tau(Z_G(S)) = \tau(Z_{\overline{G}}(\overline{S}))$, then
\[
\prod_v \lambda_{Z, v} = (q^{1-g})^{{\rm{dim}}(Z_G(S)) - {\rm{dim}}(Z_{\overline{G}}(\overline{S}))}
\]
and the product converges absolutely.
\end{lemma}

\begin{proof}
The map $f_Z: Z_G(S) \rightarrow Z_{\overline{G}}(\overline{S})$ induces an isomorphism on $k$-points and topological isomorphisms on $k_v$- and $\A$-points. This last isomorphism restricts to a topological isomorphism $Z_G(S)(\A)_1 \xrightarrow{\sim} Z_{\overline{G}}(\overline{S})(\A)_1$ on norm-$1$ adelic points, thanks to Lemma \ref{ZGcharssame}. Since $\tau(Z_G(S)) = \tau(Z_{\overline{G}}(\overline{S}))$, the two left Haar measures $\mu_{Z_G(S)}$ and $f_Z^{-1}(\mu_{Z_{\overline{G}}(\overline{S})})$ on $Z_G(S)(\A)$ must agree, as they are multiples of one another that give the coset space $Z_G(S)(k) \backslash Z_G(S)(\A)_1$ of the {\em open} subset $Z_G(S)(\A)_1 \subset Z_G(S)(\A)$ the same finite volume with respect to the corresponding quotient measures, since $\tau(Z_G(S)) = \tau(Z_{\overline{G}}(\overline{S}))$. That is,
\begin{equation}
\label{Tamcomparisoneqn1}
\mu_{Z_G(S)} = f_Z^{-1}\left(\mu_{Z_{\overline{G}}(\overline{S})}\right).
\end{equation}

Now we compute both sides of (\ref{Tamcomparisoneqn1}). We use the notation of (\ref{Tammeasuredef}). We have absolutely convergent measures on $Z_G(S)(\A)$
\begin{equation}
\label{tamzgeqn1}
\mu_{Z_G(S)} = \frac{(q^{1-g})^{{\rm{dim}}(Z_G(S))} \rho_{Z_G(S)}^{-1} } {({\rm{log}}(q))^{{\rm{rk}}(\widehat{{Z_G(S)}}(k))} \cdot \# \coker \theta_{Z_G(S)}} 
\prod_v {\rm{L}}_v(\widehat{Z_G(S)}, 1) \cdot \mu_{Z_G(S)}^v,
\end{equation}
and using Lemma \ref{productmeasurepullback},
\begin{align}
\label{tamzgeqn2}
f_Z^{-1}\left(\mu_{Z_{\overline{G}}(\overline{S})}\right) 
 & = & \frac{(q^{1-g})^{{\rm{dim}}(Z_{\overline{G}}(S))} \rho_{Z_{\overline{G}}(S)}^{-1} } {({\rm{log}}(q))^{{\rm{rk}}(\widehat{Z_{\overline{G}}(\overline{S})}(k))}\cdot \# \coker \theta_{Z_{\overline{G}}(\overline{S})}}
 \prod_v {\rm{L}}_v(\widehat{Z_{\overline{G}}(S)}, 1) \cdot f_{Z, v}^{-1}\left( \mu_{Z_{\overline{G}}(S)}^v\right) \notag \\
& = & \frac{(q^{1-g})^{{\rm{dim}}(Z_{\overline{G}}(S))} \rho_{Z_{\overline{G}}(S)}^{-1} } {({\rm{log}}(q))^{{\rm{rk}}(\widehat{Z_{\overline{G}}(\overline{S})}(k))}\cdot \# \coker \theta_{Z_{\overline{G}}(\overline{S})}}
\prod_v {\rm{L}}_v(\widehat{Z_{\overline{G}}(S)}, 1) \cdot \lambda_{Z, v} \cdot \mu_{Z_G(S)}^v,
\end{align}
where in the last equality we have used the definition (\ref{lambda_vdef}) of the $\lambda_{Z, v}$.

By Lemma \ref{ZGcharssame}, we have a commuting diagram in which the vertical isomorphisms are induced by $f_Z$ and $\widehat{f_Z}$
\[
\begin{tikzcd}
Z_G(S)(\A) \isoarrow{d} \arrow{r}{\theta_{Z_G(S)}} & \Hom(\widehat{Z_G(S)}(k), q^{\Z}) \isoarrow{d} \\
Z_{\overline{G}}(\overline{S})(\A) \arrow{r}{\theta_{Z_{\overline{G}}(\overline{S})}} & \Hom(\widehat{Z_{\overline{G}}(\overline{S})}(k), q^{\Z})
\end{tikzcd}
\]
This shows that
\[
\# \coker \theta_{Z_G(S)} = \# \coker \theta_{Z_{\overline{G}}(\overline{S})}.
\]
Lemma \ref{ZGcharssame} also shows that 
$$\rho_{Z_G(S)} = \rho_{Z_{\overline{G}}(\overline{S})},\,\,\, {\rm{rk}}(\widehat{Z_G(S)}(k)) = {\rm{rk}}(\widehat{Z_{\overline{G}}(\overline{S})}(k)),\,\,\,
{\rm{L}}_v(\widehat{Z_G(S)}, 1) = {\rm{L}}_v(\widehat{Z_{\overline{G}}(\overline{S})}, 1)$$ 
for all $v$. Thus, using (\ref{tamzgeqn1}) and (\ref{tamzgeqn2}), (\ref{Tamcomparisoneqn1}) implies that
\begin{align*}
(q^{1-g})^{{\rm{dim}}(Z_G(S)) - {\rm{dim}}(Z_{\overline{G}}(\overline{S}))} \prod_v {\rm{L}}_v(\widehat{Z_G(S)}, 1) \cdot \mu_{Z_G(S)}^v 
= \prod_v \lambda_{Z, v} \cdot {\rm{L}}_v(\widehat{Z_G(S)}, 1) \cdot \mu_{Z_G(S)}^v
\end{align*}
as measures on $Z_G(S)(\A)$. (Even though the factors ${\rm{L}}_v(\widehat{Z_G(S)}, 1)$ appear on both sides, we cannot cancel them because then the product measures will not converge, so the resulting measures would not make sense.) Since the product on the right converges {\em without} the factors $\lambda_{Z, v}$, we may ``cancel'' the product measure $\prod_v {\rm{L}}_v(\widehat{Z_G(S)}, 1) \cdot \mu_{Z_G(S)}^v$ to deduce the lemma.
\end{proof}

We may now prove Proposition \ref{toruscentralizertam=1}. 

\begin{proof}[Proof of Proposition $\ref{toruscentralizertam=1}$]
Choose nonzero top-degree left-invariant differential forms $\omega_G$ and $\omega_{\overline{G}}$ on $G$ and $\overline{G}$ respectively. Since $f: G \rightarrow \overline{G}$ induces a topological isomorphism $f_v: G(k_v) \xrightarrow{\sim} \overline{G}(k_v)$ on $k_v$-points for each place $v$ of $k$, we have
\begin{equation}
\label{cvdef}
f_v^{-1}(\mu_{\overline{G}}^v) = c_v\cdot \mu_G^v
\end{equation}
for some constants $c_v > 0$, since both sides of this equation are left Haar measures on $G(k_v)$. Since $G$ and $\overline{G}$ are perfect groups, they have no nontrivial characters over $k_s$. Hence, using the definition (\ref{Tammeasuredefnochars}) of Tamagawa measure in this case as well as Lemma \ref{productmeasurepullback}, we see that $\prod_v c_v$ converges absolutely, and to show that $f_{\A}^{-1}(\mu_{\overline{G}}) = \mu_G$ (and therefore prove Proposition \ref{toruscentralizertam=1}) is equivalent to the equality
\begin{equation}
\label{doesthisholdtam}
\prod_v c_v \stackrel{?}{=} (q^{1-g})^{{\rm{dim}}(G) - {\rm{dim}}(\overline{G})}.
\end{equation}

Let $X, \overline{X}$ denote the respective open cells that we have defined on $G, \overline{G}$ in the discussion preceding Lemma \ref{opencellpoints}. Since any nowhere-vanishing top-degree differential form on a smooth $k$-scheme generates the invertible sheaf of top-degree differential forms, there are nowhere-vanishing regular functions $F$ and $\overline{F}$ on $X$ and $\overline{X}$ respectively, such that
\[
\omega_G|_X = F\cdot \omega_{U^-} \cdot \omega_{Z_G(S)} \cdot \omega_{U^+},
\]
\[
\omega_{\overline{G}}|_{\overline{X}} = \overline{F} \cdot \omega_{\overline{U}^-} \cdot \omega_{Z_{\overline{G}}(\overline{S})} \cdot \omega_{\overline{U}^+}.
\]
Thus, using the definition (\ref{lambda_vdef}) of the constants $\lambda_{+, v}, \lambda_{-, v}, \lambda_{Z, v} > 0$, we have an equality of measures on the nonempty open subset $X(k_v) \subset G(k_v)$, hence on all of $G(k_v)$ since both sides are left Haar measures on $G(k_v)$:
\[
f_v^{-1}(\mu_{\overline{G}}^v) = \left|\frac{\overline{F}\circ f}{F}\right|_v\cdot \lambda_{+, v} \cdot \lambda_{Z, v} \cdot \lambda_{-, v} \cdot \mu_G^v.
\]
Therefore, by the definition (\ref{cvdef}) of the $c_v$, we obtain
\[
c_v = \left|\frac{\overline{F}\circ f}{F}\right|_v\cdot \lambda_{+, v} \cdot \lambda_{Z, v} \cdot \lambda_{-, v}.
\]
In particular, for each $v$ the function $\left|\frac{\overline{F}\circ f}{F}\right|_v$ is a nonzero constant on $G(k_v)$. But $F$ and $\overline{F}$ are nowhere-vanishing, so we can evaluate at $1 \in G(k_v)$ to compute this constant. That is, $|(\overline{F} \circ f)/F|_v = |r|_v$, where $r := \overline{F}(1)/F(1) \in k^{\times}$; note that this is {\em independent} of $v$! We have found $r \in k^{\times}$ such that
\[
c_v = |r|_v\cdot \lambda_{+, v} \cdot \lambda_{-, v} \cdot \lambda_{Z, v}
\]
for every place $v$ of $k$. Using Lemmas \ref{prodlambda_pmv}, \ref{prodlambdazv}, and the product formula, we therefore obtain
\begin{align*}
\prod_v c_v & = \prod_v |r|_v \cdot \prod_v \lambda_{+, v} \cdot \prod_v \lambda_{Z, v} \cdot \prod_v \lambda_{-, v} \\
& = 1 \cdot (q^{1-g})^{{\rm{dim}}(U^+) - {\rm{dim}}(\overline{U}^+)} (q^{1-g})^{{\rm{dim}}(Z_G(S)) - {\rm{dim}}(Z_{\overline{G}}(\overline{S}))} (q^{1-g})^{{\rm{dim}}(U^-) - {\rm{dim}}(\overline{U}^-)} \\
& = (q^{1-g})^{{\rm{dim}}(G) - {\rm{dim}}(\overline{G})}.
\end{align*}
This is (\ref{doesthisholdtam}), so the proof of Proposition \ref{toruscentralizertam=1} is complete. 
\end{proof}

\section{Tamagawa numbers of exotic groups: computations}
\label{exotictam=1section}

In this section we will prove Proposition \ref{measuresagreebasic} by constructing suitable split tori as in Proposition \ref{toruscentralizertam=1}. First suppose that $G$ is a basic non-reduced group over the global function field $k$ of characteristic $2$. Let $k' := k^{1/2}$ and $\overline{G} = \R_{k'/k}({\rm{Sp}}_{2n})$ with $f: G \twoheadrightarrow \overline{G}$ the canonical map as in \S \ref{sectionexoticgeneral}. There is a bijection between the sets of
maximal $k$-tori of $G$ and maximal $k'$-tori of ${\rm{Sp}}_{2n}$ defined by sending a maximal $k'$-torus $T' \subset {\rm{Sp}}_{2n}$ to the unique maximal $k$-torus $T \subset f^{-1}(\R_{k'/k}(T'))$, and the map $T \rightarrow \R_{k'/k}(T')$ induced by $f$ yields an isomorphism onto the maximal $k$-torus of $\R_{k'/k}(T')$ \cite[Cor.\,9.9.5]{cgp}.

In conjunction with Proposition \ref{toruscentralizertam=1}, the following result completes the proof of Proposition \ref{basicexotictam=1} (or equivalently, Proposition \ref{measuresagreebasic}) for basic non-reduced $G$.

\begin{proposition}
\label{Z_G(T)nonreduced}
Let $G$ be a basic non-reduced group over a global function field $k$ of characteristic $2$, with $f: G \twoheadrightarrow \overline{G}$ the canonical map as above. For a split maximal $k$-torus $T \subset G$,
the natural map $Z_G(T) \to Z_{\overline{G}}(\overline{T})$ is
an isomorphism.  In particular, $\tau(Z_G(T)) = \tau(Z_{\overline{G}}(\overline{T}))$.
\end{proposition}

\begin{proof}
Since all split maximal tori in $G$ are $G(k)$-conjugate to each other, and likewise for $\overline{G}$ and 
for ${\rm{Sp}}_{2n}$ over $k'$, with both maps $G(k) \to \overline{G}(k) \to {\rm{Sp}}_{2n}(k^{1/2})$ bijective,
we can arrange that $\overline{T} = {\rm{R}}_{k^{1/2}/k}(D)$ for the diagonal split maximal $k^{1/2}$-torus $D \subset {\rm{Sp}}_{2n}$.
By \cite[Thm.\,9.8.1(2),(4)]{cgp} and \cite[Rem.\,9.8.2]{cgp}, the surjective map $Z_G(T) \to Z_{\overline{G}}(\overline{T})$
between Cartan $k$-subgroups is an isomorphism.
\end{proof}

Now we prove Proposition \ref{measuresagreebasic} for basic exotic groups by constructing split tori satisfying the hypotheses of Proposition \ref{toruscentralizertam=1}.
First we fix the same notation from the start of \S\ref{sectionexoticgeneral}. So $k$ is a global function field of characteristic $p \in \{2, 3\}$, $k' := k^{1/p}$, $G'$ is a connected semisimple $k'$-group that is absolutely simple and simply connected
with absolute Dynkin diagram containing an edge of multiplicity $p$, and we have the very special isogeny $\pi': G' \rightarrow \overline{G}'$ and the induced map $\R_{k'/k}(\pi'): \R_{k'/k}(G') \rightarrow \R_{k'/k}(\overline{G}')$ denoted as $f$. 
Let  $\overline{G} \subset \R_{k'/k}(\overline{G}')$ be a Levi $k$-subgroup (necessarily simply connected with root system dual to that of $G'$)
such that for $G := f^{-1}(\overline{G})$, the $k_s$-group $G_{k_s}$ contains a Levi $k_s$-subgroup of
$\R_{k'/k}(G')_{k_s}$ (it is equivalent that $\overline{G} \subset {\rm{im}}(f)$, and also equivalent that $G$ is $k$-smooth
 \cite[Prop.\,7.3.1]{cgp}), so $G$ is a basic exotic pseudo-reductive $k$-group.

The following proposition will allow us to treat the various possible types of such $G'$ almost uniformly.

\begin{proposition}
\label{suitablesplittorus}
Let $G$ be a basic exotic group over a global function field $k$, and $f: G \twoheadrightarrow \overline{G}$ its canonical surjection onto a simply connected $k$-group $($see Remark $\ref{canonicalsurjection}$$)$. Assume that for some maximal $k$-torus $\overline{T} \subset \overline{G}$ and basis $\overline{\Delta}$
of the root system $\Phi(\overline{G}_{k_s}, \overline{T}_{k_s})$, there is a $($possibly empty$)$ subset $\overline{A} \subset \overline{\Delta}$ of pairwise orthogonal roots such that the $($scheme-theoretic$)$ center $Z_{\overline{G}}$ of $\overline{G}$ is contained in $\prod_{\overline{a} \in \overline{A}} \overline{a}^{\vee}(\Gm)$.

Then there is such a triple $(\overline{T}, \overline{\Delta}, \overline{A})$ for which the $k_s$-torus $(\cap_{\overline{a} \in \overline{A}} \ker(\overline{a}))^0_{{\rm{red}}} \subset \overline{T}_{k_s}$ descends to a 
split $k$-subtorus $\overline{S} \subset \overline{T}$. For the corresponding split subtorus $S \subset T$
via the isogeny $T \to \overline{T}$, we have $\tau(Z_G(S)) = \tau(Z_{\overline{G}}(\overline{S})) = 1$. 
\end{proposition}

\begin{remark}
All maximal tori in $\overline{G}$ are $\overline{G}(k_s)$-conjugate, and
$N_{\overline{G}}(\overline{T})(k_s)$-conjugation (which has no effect on $Z_{\overline{G}}$)
 acts transitively on the set of bases of the root system, so the hypothesis concerning
the existence of $\overline{A}$ is independent of the choice of $\overline{T}$ and $\overline{\Delta}$.
Note in particular that the existence of $\overline{A}$ only involves $\overline{G}_{k_s}$ and so is sufficient to check
after making a preliminary finite separable extension of $k$ to split $\overline{T}$.
\end{remark}

Before turning to the proof of Proposition \ref{suitablesplittorus}, let us show how it allows us to quickly complete the proof of Proposition \ref{basicexotictam=1} (or equivalently, Proposition \ref{measuresagreebasic}) for basic exotic groups. 

\begin{proof}[Proof of Proposition $\ref{measuresagreebasic}$ $($assuming Proposition $\ref{suitablesplittorus}$$)$]
The proposition for basic non-reduced groups, as we have already noted, follows from Propositions \ref{toruscentralizertam=1} and \ref{Z_G(T)nonreduced}. It remains to treat the basic exotic case. 

We maintain the notation of Proposition \ref{suitablesplittorus}. The absolute Dynkin diagram of $\overline{G}$ is of type G$_2$ (in characteristic $3$) or F$_4$, B$_n$, or C$_n$ (in characteristic $2$; $n \geq 2$ in the last two cases). By Propositions \ref{toruscentralizertam=1} and \ref{suitablesplittorus}, it suffices to prove the existence of $\overline{A}$ as in Proposition \ref{suitablesplittorus} for the split group $\overline{G}_{k_s}$.

 In the F$_4$ and G$_2$ cases we have $Z_{\overline{G}} = 1$, so we may take $\overline{A} = \emptyset$. In the case that $\overline{G}$ is of type B$_n$, we have $Z_{\overline{G}} = \overline{a}^{\vee}(\mu_2)$ where $\overline{a}$ is the unique short root in any basis
 $\overline{\Delta}$, so we may take $\overline{A} = \{ \overline{a} \}$. Finally, when $\overline{G}$ is of type C$_n$
 we can take $\overline{A} \subset \overline{\Delta}$ to be the set of alternating roots in the associated Dynkin diagram beginning with the unique short root at the end of the diagram; the center is then $Z_{\overline{G}} = (\prod_{\overline{a} \in \overline{A}} \overline{a}^{\vee})(\mu_2) \simeq \mu_2$. (Note: this is {\em not} the same as $\prod_{\overline{a} \in \overline{A}} \overline{a}^{\vee}(\mu_2) \simeq \mu_2^{\# \overline{A}}$ when $n \geq 3$.)
 \end{proof}

It remains to prove Proposition \ref{suitablesplittorus}. This will occupy us for the remainder of this section. Before we do this, let us prove some general results on the cohomology of algebraic groups over global function fields that we shall require now and in the sequel. These results are presumably ``well-known'', but we prove them here for the reader's convenience.

\begin{proposition}
\label{torustrivialSha^2}
Let $H$ be a connected semisimple group over a global field $k$, and let $\Sigma$ be a finite set of places of $k$. If $\Sigma$ contains any real places, then further assume that $H$ is simply connected. Then there exists a maximal $k$-torus $T \subset H$ such that $\Sha^2(k, T) = 0$ and such that ${\rm{H}}^2(k_v, T) = 0$ for all $v \in \Sigma$.
\end{proposition}

\begin{proof}
Enlarge $\Sigma$ if necessary in order to assume that it contains a non-archimedean place of $k$. First we find for each $v \in \Sigma$ a maximal $k_v$-torus $T_v \subset H_{k_v}$ such that ${\rm{H}}^2(k_v, T_v) = 0$. If $k_v = \mathbf{R}$, then the existence of such $T_v$ is \cite[\S 6.5, Lemma 6.18]{platonov}. (This is where we need the simply connected assumption if $\Sigma$ contains real places.) If $v$ is non-archimedean, then $H_{k_v}$ contains an anisotropic maximal $k_v$-torus $T_v$ \cite[Lemma 2.4.1]{debacker}. By Tate local duality \cite[Thm.\,1.2.1]{rostateduality} (although for tori this result has been known for a long time), the anisotropicity of $T_v$ implies that ${\rm{H}}^2(k_v, T_v) = 0$.

By \cite[Lemma 5.5.3]{harder} (it is stated for number fields, but the proof works just as well over function fields), there exists a maximal $k$-torus $T \subset H$ such that for each $v \in \Sigma$, $T_{k_v}$ is $H(k_v)$-conjugate to $T_v$, so ${\rm{H}}^2(k_v, T) = 0$  for all $v \in \Sigma$ and $T_{k_v}$ is $k_v$-anisotropic for the non-archimedean places $v \in \Sigma$. By the comment on the top of page 408 of \cite{harder} (just after \cite[Lemma 5.5.4]{harder}; once again, the comment is made for number fields, but goes over without change to the function field setting), this anisotropicity implies that $\Sha^2(T) = 0$.
\end{proof}

The following result seems to have first been explicitly stated and proved as \cite[Thm.\,A]{thang}. We give a simpler proof below.

\begin{proposition}
\label{connectingmap=bijection}
Let $G$ be a connected semisimple group over a global function field $k$, and consider the central extension
\begin{equation}
\label{simplyconnectedcoverseq1}
1 \longrightarrow Z \longrightarrow \widetilde{G} \xlongrightarrow{\pi} G \longrightarrow 1
\end{equation}
with $\pi$ the simply connected cover of $G$. Then the connecting map $\delta: {\rm{H}}^1(k, G) \rightarrow {\rm{H}}^2(k, Z)$ is a bijection.
\end{proposition}

\begin{proof}
For the injectivity, suppose we have $x, y \in {\rm{H}}^1(k, G)$ such that $\delta(x) = \delta(y)$. Since $Z \subset \widetilde{G}$ is central, $G$ acts on the whole sequence by conjugation, so we may twist the sequence (\ref{simplyconnectedcoverseq1}) by $x$ (more precisely, by a cocycle representing $x$) to obtain a new sequence
\begin{equation}
\label{simplyconnectedcoverseq2}
1 \longrightarrow Z \longrightarrow \widetilde{G}_x \longrightarrow G_x \longrightarrow 1.
\end{equation}
(See \cite[Chap.\,I, \S 5.3]{serre} for details on twisting.) Using the natural bijection ${\rm{H}}^1(k, G) \xrightarrow{\sim} {\rm{H}}^1(k, G_x)$, which sends $x$ to $1$ \cite[Chap.\,I, \S 5.3, Prop.\,35]{serre}, if we denote the image of $y$ by $y_x$, then we see that we have $\delta(y_x) = \delta(x_x) = \delta(1) = 1$, so we may assume that $x = 1$ and $\delta(y) = 1$ (since all hypotheses are preserved when passing to the twisted sequence (\ref{simplyconnectedcoverseq2})), and we need to show that $y = 1$. This follows from the vanishing of ${\rm{H}}^1(k, \widetilde{G}_x)$, which holds because $\widetilde{G}_x$ is simply connected \cite[Satz A]{harder1}.

To prove that the connecting map is surjective, suppose that $x \in {\rm{H}}^2(k, Z)$. Then $x_v \in {\rm{H}}^2(k_v, Z)$ vanishes for all but finitely many places $v$ of $k$ by \cite[Prop.\,5.2.2]{rostateduality}. Let $\Sigma$ denote the finite set of places $v$ such that $x_v \neq 0$. By Proposition \ref{torustrivialSha^2}, there is a maximal $k$-torus $\widetilde{T} \subset \widetilde{G}$ such that $\Sha^2(k, \widetilde{T}) = 0$ and ${\rm{H}}^2(k_v, \widetilde{T}) = 0$ for all $v \in \Sigma$. Let $T := \pi(\widetilde{T})$ be the associated maximal $k$-torus of $G$.

Consider the exact sequence
\[
0 \longrightarrow Z \xlongrightarrow{j} \widetilde{T} \longrightarrow T \longrightarrow 0.
\]
Then $0 = j(x) \in {\rm{H}}^2(k, \widetilde{T})$. Indeed, since $\Sha^2(\widetilde{T}) = 0$, it suffices to show that $j(x_v) = 0$ for all places $v$ of $k$. For $v \notin \Sigma$, we have $x_v = 0$, so this is clear, while for $v \in \Sigma$, we have ${\rm{H}}^2(k_v, \widetilde{T}) = 0$, so this is once again clear. It follows that $x = \delta'(\alpha)$ for some $\alpha \in {\rm{H}}^1(k, T)$, where $\delta': {\rm{H}}^1(k, T) \rightarrow {\rm{H}}^2(k, Z)$ is the connecting map. Using the commutative diagram
\[
\begin{tikzcd}
0 \arrow{r} & Z \arrow[d, equals] \arrow{r} & \widetilde{T} \arrow{d} \arrow{r} & T \arrow{d}{i} \arrow{r} & 0 \\
1 \arrow{r} & Z \arrow{r} & \widetilde{G} \arrow{r} & G \arrow{r} & 1
\end{tikzcd}
\]
we then see that $x = \delta(i(\alpha)) \in \delta({\rm{H}}^1(k, G))$.
\end{proof}

Now we turn to the proof of Proposition \ref{suitablesplittorus}. Let $\overline{G}_0$ be the split form of $\overline{G}$. Fix a split maximal $k$-torus $\overline{T}_0 \subset \overline{G}_0$ and a choice of basis $\overline{\Delta}_0 \subset \Phi(\overline{G}_0, \overline{T}_0)$, and let $\overline{A}_0 \subset \overline{\Delta}_0$ be a set of mutually orthogonal simple roots such that $Z_{\overline{G}_0} \subset \prod_{\overline{a}_0 \in \overline{A}_0} \overline{a}_0^{\vee}(\Gm)$. For $\overline{a}_0 \in \Phi(\overline{G}_0, \overline{T}_0)$, let SL$_{\overline{a}_0} := \mathscr{D}(Z_{\overline{G}_0}((\overline{T}_0)_{\overline{a}_0}))$, where $(\overline{T}_0)_{\overline{a}_0} := \ker(\overline{a}_0)^0_{{\rm{red}}} \subset \overline{T}_0$. The group SL$_{\overline{a}_0}$ is isomorphic to SL$_2$, since $\overline{G}_0$ is simply connected.

Since the elements of $\overline{A}_0$ are mutually orthogonal, the root groups $U_{\pm \overline{a}_0}$ commute with the root groups $U_{\pm \overline{b}_0}$ for distinct $\overline{a}_0, \overline{b}_0 \in \overline{A}_0$. Since SL$_{\overline{a}_0}$ is generated by 
the corresponding root groups $U_{\pm \overline{a}_0}$, it follows that the groups $\{{\rm{SL}}_{\overline{a}_0}\}_{\overline{a}_0 \in \overline{A}_0}$
commute with  each other. Therefore, the multiplication map
\[
m: \prod_{ \overline{a}_0 \in \overline{A}_0} {\rm{SL}}_{\overline{a}_0} \rightarrow \overline{G}_0
\]
is a homomorphism. It restricts to an inclusion on an open neighborhood of 1 due to the open cell decomposition of the
{\em simply connected} $\overline{G}_0$ (since $\overline{T}_0$ is the direct product of the 1-parameter subgroups
arising from the coroots associated to $\overline{\Delta}_0$), so $m$ is birational onto its image and hence is an inclusion of $k$-groups. 

The $k$-group $Z_{\overline{G}_0} \subset \overline{G}_0$ is either trivial or $\mu_2$, as may be checked in the separate cases when $\overline{G}_0$ is of type ${\rm{F}}_4$ or ${\rm{G}}_2$ (trivial) or ${\rm{B}}_n$ or ${\rm{C}}_n$ ($\mu_2$). We have 
\[
Z_{\overline{G}_0} \subset \prod_{ \overline{a}_0 \in \overline{A}_0} {\rm{SL}}_{\overline{a}_0},
\]
so
\[
Z_{\overline{G}_0} \subset \prod_{ \overline{a}_0 \in \overline{A}_0} \overline{a}_0^{\vee}(\mu_2).
\]
The $k$-group $Z_{\overline{G}_0}$ is either trivial or $\mu_2$, and in the latter case it is the diagonally embedded
$\mu_2$ along the factors corresponding to the subset $\overline{R}_0 \subset \overline{A}_0$ of elements $\overline{a}_0 \in \overline{A}_0$
for which the projection from $Z_{\overline{G}_0}$ to the $\overline{a}_0$-factor is nontrivial.  
Hence, if we define $\overline{H}_0$ to be the copy of ${\rm{SL}}_2$ diagonally embedded
along $\overline{R}_0$ when $\overline{G}_0$ is of type B or C, and define
$\overline{H}_0 = 1$ for types F$_4$ and G$_2$ (in effect, $\overline{R}_0$ is empty), we have $Z_{\overline{H}_0} = Z_{\overline{G}_0}$.

Now the $k$-forms of $\overline{G}_0$ (such as $\overline{G}$) are classified by the set ${\rm{H}}^1(k, {\rm{Aut}}_{\overline{G}_0/k})$. The Dynkin diagram of $\overline{G}_0$, which is either F$_4$, G$_2$, B$_n$, or C$_n$, has no nontrivial automorphisms, hence ${\rm{Aut}}_{\overline{G}_0/k} = \overline{G}_0^{{\rm{ad}}} := \overline{G}_0/Z_{\overline{G}_0}$. Consider the commutative diagram of exact sequences
\[
\begin{tikzcd}
1 \arrow{r} & Z_{\overline{H}_0} \arrow[d, equals] \arrow{r} & \overline{H}_0 \arrow[d, hookrightarrow] \arrow{r} & \overline{H}_0^{{\rm{ad}}} \arrow{r} \arrow{d} & 1 \\
1 \arrow{r} & Z_{\overline{G}_0} \arrow{r} & \overline{G}_0 \arrow{r} & \overline{G}_0^{{\rm{ad}}} \arrow{r} & 1
\end{tikzcd}
\]
The connecting maps associated to the top and bottom sequences induce a commutative diagram
\[
\begin{tikzcd}
{\rm{H}}^1(k, \overline{H}_0^{{\rm{ad}}}) \arrow{r}{\delta} \arrow{d} & {\rm{H}}^2(k, Z_{\overline{H}_0}) \arrow[d, equals] \\
{\rm{H}}^1(k, \overline{G}_0^{{\rm{ad}}}) \arrow{r}{\delta} & {\rm{H}}^2(k, Z_{\overline{G}_0})
\end{tikzcd}
\]
Since $\overline{H}_0$ and $\overline{G}_0$ are simply connected and $k$ is a global function field, the well-known Proposition \ref{connectingmap=bijection} implies that the connecting maps above are bijections. Therefore, the left vertical map is also a bijection.

It follows that {\em any} $k$-form of $\overline{G}_0$, such as $\overline{G}$, is obtained by twisting $\overline{G}_0$ by the conjugation action of a $1$-cocycle valued in $\overline{H}_0$. Since every class in ${\rm{H}}^1(k, \overline{H}_0^{{\rm{ad}}})$
is represented by one valued in the normalizer of the split maximal torus \cite[Ch.\,III, \S2.2, Lemma 1]{serre}, 
 which in turn lies inside $N_{\overline{G}_0^{\rm{ad}}}(\overline{T}_0^{\rm{ad}})$ (since the reflections
 for the elements in $\overline{A}_0$ pairwise commute due to the pairwise-orthogonality condition), 
 there is a continuous $1$-cocycle $\zeta: {\rm{Gal}}(k_s/k) \rightarrow \overline{H}_0(k_s)$ normalizing $\overline{T}_0$ such that 
$(\overline{G}, \overline{T})$ is the $k_s/k$-form of $(\overline{G}_0, \overline{T}_0)$ using the $\zeta$-twisted Galois action
\[
{}^{\sigma'}g := \zeta_{\sigma} \cdot {}^{\sigma} g \cdot \zeta_{\sigma}^{-1}.
\]
In particular, for any split $k$-torus $\overline{S}_0 \subset \overline{T}_0$ that {\em commutes} with the $k$-group $\overline{H}_0$, $(\overline{S}_0)_{k_s}$ descends to a {\em split} $k$-torus $\overline{S} \subset \overline{G}$. We claim that the torus 
$$\overline{S}_0 := \left(\bigcap_{\overline{a}_0 \in \overline{A}_0} \ker(\overline{a}_0)\right)^0_{{\rm{red}}}$$
is such a $k$-torus. Assuming this, we get a corresponding split torus $S \subset T$ via the isogeny $T \to \overline{T}$,
and it remains to check that $Z_G(S)$ and $Z_{\overline{S}}(\overline{S})$ both have Tamagawa number $1$. 

In order to prove that $\overline{S}_0$ commutes with $\overline{H}_0$, recall that $\overline{H}_0$ is generated by the root groups $U_{\pm \overline{a}_0}$ for $\overline{a}_0 \in \overline{A}_0$. Thus, it suffices to show that $\overline{S}_0$ commutes with each of these
root groups. But the conjugation action of $\overline{T}_0$ on the root group $U_{\pm \overline{a}_0} \simeq \Ga$ is given by the character $\pm \overline{a}_0$ which kills $\overline{S}_0$ by definition when $\overline{a}_0 \in \overline{A}_0$, so $\overline{S}_0$ centralizes $U_{\pm \overline{a}_0}$ for $a \in \overline{A}_0$ as desired.  Let $\overline{A} \subset \overline{\Delta} \subset X(\overline{T}_{k_s})$
correspond to $\overline{A}_0 \subset \overline{\Delta}_0 \subset X(\overline{T}_0) = X((\overline{T}_0)_{k_s})$ 
via the identification $\overline{T}_{k_s} \simeq (\overline{T}_0)_{k_s}$.  
We have now built the split $k$-subtori $\overline{S} \subset \overline{T}$ and $S \subset T$,
and we will show that $\tau(Z_G(S)) = 1 = \tau(Z_{\overline{G}}(\overline{S}))$. 

Let us first show that $\tau(Z_{\overline{G}}(\overline{S})) = 1$.
As a first step, we show that $\overline{S} \subset Z_{\overline{G}}(\overline{S})$ is the maximal central torus. 
Any central torus $\overline{S}' \subset Z_{\overline{G}}(\overline{S})$ is contained in every maximal torus, hence in 
$T$. The root groups $U_{\overline{a}} \subset \overline{G}_{k_s}$ are contained in $Z_{\overline{G}}(\overline{S})_{k_s}$ for $\overline{a} \in \overline{A}$, so $\overline{S}'_{k_s} \subset \ker(\overline{a})$ for such $\overline{a}$. Therefore, 
$$\overline{S}'_{k_s} \subset (\bigcap_{\overline{a} \in \overline{A}} \ker(\overline{a}))^0_{{\rm{red}}} = \overline{S}_{k_s},$$
so $\overline{S}$ is indeed the maximal central torus of $Z_{\overline{G}}(\overline{S})$. 
Since $Z_{\overline{G}}(\overline{S})$ is reductive, we therefore have $$Z_{\overline{G}}(\overline{S}) = \overline{S} \cdot \mathscr{D}(Z_{\overline{G}}(\overline{S})).$$

It follows that the abelianization $Z_{\overline{G}}(\overline{S})^{{\rm{ab}}}$ of $Z_{\overline{G}}(\overline{S})$ is a split torus. Further, the group $\mathscr{D}(Z_{\overline{G}}(\overline{S}))$ is simply connected (this is well-known, but isn't discussed in standard references;
a proof is given in the extensive hints for \cite[Exer.\,6.5.2(iv)]{conrad2}). Consider the exact sequence
\[
1 \longrightarrow \mathscr{D}(Z_{\overline{G}}(\overline{S})) \longrightarrow Z_{\overline{G}}(\overline{S}) \longrightarrow Z_{\overline{G}}(\overline{S})^{{\rm{ab}}} \longrightarrow 1.
\]
Since $\mathscr{D}(Z_{\overline{G}}(\overline{S}))$ is simply connected, the following properties hold: it has
no nontrivial characters, ${\rm{H}}^1(k, \mathscr{D}(Z_{\overline{G}}(\overline{S}))) = 1$ \cite[Satz A]{harder1}, and ${\rm{H}}^1(\A, \mathscr{D}(Z_{\overline{G}}(\overline{S}))) = 1$ (by Proposition \ref{adeliccohomologyprop} and \cite[Thm.\,4.7]{bruhattits}). The map $Z_{\overline{G}}(\overline{S})(\A) \rightarrow Z_{\overline{G}}(\overline{S})^{{\rm{ab}}}(\A)$ is therefore surjective and by \cite[Chap.\,III, \S 5.3, Thm.]{oesterle}, we have
\[
\tau(Z_{\overline{G}}(\overline{S})) = \tau(\mathscr{D}(Z_{\overline{G}}(\overline{S}))) \cdot \tau(Z_{\overline{G}}(\overline{S})^{{\rm{ab}}}).
\]
But $\tau(\mathscr{D}(Z_{\overline{G}}(\overline{S}))) = 1$ because $\mathscr{D}(Z_{\overline{G}}(\overline{S}))$ is simply connected, and $\tau(Z_{\overline{G}}(\overline{S})^{{\rm{ab}}}) = 1$ because $Z_{\overline{G}}(\overline{S})^{{\rm{ab}}}$ is a split torus. Therefore, $\tau(Z_{\overline{G}}(\overline{S})) = 1$.

It remains to show that $\tau(Z_G(S)) = 1$. Let $T_{k'} \subset G' = G_{k'}^{\rm{red}}$ be the maximal $k'$-torus arising from $T \subset G$. Then we have a natural identification $\Phi := \Phi(G'_{k'_s}, T'_{k'_s}) = \Phi(G_{k_s}, T_{k_s})$ \cite[Prop.\,7.2.7(2)]{cgp}. Let $\Delta \subset \Phi$ be the basis corresponding to $\overline{\Delta} \subset \overline{\Phi} := \Phi(\overline{G}_{k_s}, \overline{T}_{k_s})$ in the dual root system 
via the isogeny $T \to \overline{T}$ \cite[Prop.\,7.1.5]{cgp}, and let $A \subset \Delta$ be the subset corresponding to $\overline{A} \subset \overline{\Delta}$. 

We first claim that the only roots $b \in \Phi$ that are trivial on $S$ are those of the form $\pm a$ with $a \in A$. 
We may suppose that $b$ is positive with respect to $\Delta$, and we want to show $b \in A$. Since 
$\Delta$ is a $\Z$-basis of $X(T_{k_s})$ (as $G'_{k_s}$ is simply connected and $X(T_{k_s}) \simeq X(T'_{k_s})$) and $b$ is trivial on 
$$S = (\bigcap_{a \in A} \ker(a))^0_{{\rm{red}}},$$ 
we have $b = \sum_{a \in A} n_a a \in \Phi$ with $n_a \in \Z_{\geq 0}$. 
There is some $a_1$ for which $n_{a_1} > 0$. For any $a_2 \in A$ distinct from $a_1$, apply the reflection $r_{a_2}$ to get the root $r_{a_2}(b) \in \Phi$.
Since the elements of $A$ are mutually orthogonal, we have 
$$r_{a_2}(b) = -n_{a_2} a_2 + \sum_{a \in A-\{a_2\}} n_a a \in \Phi$$
with $a_1$-coefficient $n_{a_1} > 0$, so all coefficients must be non-negative.  This forces $-n_{a_2} \ge 0$, so $n_{a_2} \le 0$,
yet by design $n_{a_2} \ge 0$.  Thus, $n_{a_2} = 0$ for all $a_2 \in A$ distinct from $a_1$.
This yields $b = n_{a_1} a_1$ with $n_{a_1}$ a positive integer.  By reducedness of $\Phi$, we conclude that $b = a_1 \in A$ as desired.

For a root $a \in \Phi$, let $U'_a \subset G'_{k_s}$ denote the corresponding root group. Let $A_< := A \cap \Phi_<$ denote the set of 
short roots in $A$ and $A_> := A \cap \Phi_>$ denote the set of long roots in $A$. Using the $T_{k_s}$-equivariant
description of the open cell 
for $(G_{k_s}, T_{k_s})$ associated to $\Delta$ given in \cite[Remark 7.2.8]{cgp}, together with the fact that the only roots killing $S_{k_s}$ are the elements of $A$ and their negatives (as shown above), an elementary $S_{k_s}$-centralizer calculation
factorwise in the open cell for $(G_{k_s}, T_{k_s}, \Delta)$ yields that the open cell for $(Z_G(S)_{k_s}, T_{k_s}, \Delta)$ is
\[
\R_{k'_s/k_s}(U'_{-A_{<}}) \times U'_{-A_{>}} \times \prod_{b \in \Delta_<} \R_{k'_s/k_s}(b^{\vee}(\Gm)) \times \prod_{b \in \Delta_>} b^{\vee}(\Gm) \times \R_{k'_s/k_s}(U'_{A_{<}}) \times  U'_{A_{>}}
\]
for the product varieties (also subgroups, by pairwise orthogonality of roots in $A$) 
$$U'_{A_{>}} = \prod_{a \in A_{>}} U'_a,\,\,\,U'_{A_{<}} = \prod_{a \in A_{<}} U'_a,\,\,\,
U'_{-A_{>}} = \prod_{a \in A_{>}} U'_{-a},\,\,\,U'_{-A_{<}} = \prod_{a \in A_{<}} U'_{-a}.$$

Since root groups for a maximal torus in a pseudo-reductive group are normalized by the corresponding Cartan subgroup
and the groups $\{{\rm{SL}}_a\}_{a \in A}$ pairwise commute (by pairwise orthogonality of the roots in $A$), the semidirect product
\begin{equation}
\label{Z_G(S)opencelleqn1}
(\prod_{b \in \Delta_< - A} \R_{k'_s/k_s}(b^{\vee}(\Gm)) \times \prod_{b \in \Delta_> - A} b^{\vee}(\Gm)) \ltimes 
(\prod_{a \in A_<} \R_{k'_s/k_s}({\rm{SL}}_a) \times \prod_{a \in A_>} {\rm{SL}}_a)
\end{equation}
makes sense as a $k_s$-subgroup of $G_{k_s}$ yet its own open cell is an {\em open subvariety} of $Z_G(S)_{k_s}$ and hence this group coincides
with $Z_G(S)_{k_s}$!

The groups ${\rm{SL}}_a \subset G'$ are perfect, and so likewise for each $\R_{k'/k}({\rm{SL}}_a)$ \cite[Prop.\,2.2.4]{cpsurvey}, so since (\ref{Z_G(S)opencelleqn1}) equals $Z_G(S)_{k_s}$, we get 
\begin{equation}
\label{DZ_G(S)opencelleqn2}
\mathscr{D}(Z_G(S))_{k_s} = \prod_{a \in A_<} \R_{k'_s/k_s}({\rm{SL}}_a) \times \prod_{a \in A_>} {\rm{SL}}_a.
\end{equation}
If we define the finite reduced $k$-algebra $F := {k'}^{A_{<}} \times k^{A_{>}}$ (which is 0 when $A$ is empty), then this says 
that the $k$-groups $\mathscr{D}(Z_G(S))$ and ${\rm{R}}_{F/k}({\rm{SL}}_2)$ are $k_s/k$-forms of each other. We can thereby get a description of $\mathscr{D}(Z_G(S))$: it is $k$-isomorphic to $\R_{K/k}(H)$ for a finite reduced $k$-algebra $K$
and a smooth affine $K$-group $H$ whose fiber over each factor field of $K$ is a form of ${\rm{SL}}_2$. Indeed, we need to show that any $k_s/k$-descent datum on $\R_{F/k}({\rm{SL}}_2)$ arises uniquely from a $k_s/k$-descent datum on the $k$-algebra $F$ and a descent datum on ${\rm{SL}}_2$ over that. This is immediate from \cite[Prop.\,A.5.14]{cgp}.

Summarizing, we have shown that
\begin{equation}
\label{DZ=R(SL2)}
\mathscr{D}(Z_G(S)) = \R_{K/k}(H)
\end{equation}
for some (possibly $0$) finite reduced $k$-algebra $K$ and \'etale $K$-form $H$ of ${\rm{SL}}_2$. Next we turn to analyzing the group $Z_G(S)^{{\rm{ab}}} := Z_G(S)/\mathscr{D}(Z_G(S))$.

By (\ref{Z_G(S)opencelleqn1}) and (\ref{DZ_G(S)opencelleqn2}), we have a $k_s$-isomorphism
\begin{equation}
\label{psik_sisomeqn}
\psi: ((\Gm)^{n_1} \times (\R_{k'/k}(\Gm))^{n_2})_{k_s} \simeq Z_G(S)^{{\rm{ab}}}_{k_s}
\end{equation}
for some nonnegative integers $n_1, n_2$. The maximal torus of the commutative pseudo-reductive $k$-group $(\Gm)^{n_1} \times (\R_{k'/k}(\Gm))^{n_2}$ is split (since the maximal torus of $\R_{k'/k}(\Gm)$ is the canonically embedded $\Gm$, as $k'/k$ is purely inseparable). We claim that the maximal $k$-torus in the commutative pseudo-reductive $k$-group $Z_G(S)^{{\rm{ab}}}$ is also split. (This abelianization is pseudo-reductive by (\ref{psik_sisomeqn}) because pseudo-reductivity may be checked over $k_s$.)

In fact, we claim that the image of the split torus $S$ in this commutative group is its maximal $k$-torus. 
By (\ref{Z_G(S)opencelleqn1}) and (\ref{DZ_G(S)opencelleqn2}), it is equivalent to show that the image of the projection map from $S_{k_s} \subset Z_G(S)_{k_s}$ into the direct factor 
\[
\prod_{\substack{b \in \Delta_< -A_<}} \R_{k_s'/k_s}(b^{\vee}(\Gm)) \times \prod_{\substack{b \in \Delta_> - A_>}} b^{\vee}(\Gm) \subset \prod_{b \in \Delta_<} \R_{k'_s/k_s}(b^{\vee}(\Gm)) \times \prod_{b \in \Delta_>} b^{\vee}(\Gm) = Z_G(T)_{k_s}
\]
is its maximal $k_s$-torus. The maximal torus of the left side has dimension $\# \Delta - \# A$,
which is the same as the dimension of $S$, so it suffices to show that the kernel of the projection from $S_{k_s}$ into the factor group on the left is finite. This kernel is 
$$S_{k_s} \bigcap \prod_{a \in A} a^{\vee}(\Gm).$$ 
For any point $\prod_{a \in A} a^{\vee}(t_a) \in T$ with $t_a \in \Gm$, applying $b \in A$ to this element yields $b(b^{\vee}(t_b)) = t_b^2$
by mutual orthogonality of the elements of $A$. Hence, when that point belongs to $S$
we conclude from the triviality of $b|_S$ that $t_b^2 = 1$ for each $b \in A$.  This proves the desired finiteness of the intersection. 

We have shown that (\ref{psik_sisomeqn}) yields an isomorphism over $k_s$ between two commutative pseudo-reductive $k$-groups with split maximal $k$-tori. The following result ensures that this isomorphism is actually defined over $k$.

\begin{lemma}
Let $k$ be a field, and $C, C'$ commutative pseudo-reductive $k$-groups whose maximal $k$-tori are split. 
Then any $k_s$-homomorphism $\phi: C_{k_s} \rightarrow C'_{k_s}$ descends to a $k$-homomorphism $C \rightarrow C'$.
\end{lemma}

\begin{proof}
We need to show that ${}^{\sigma}\phi = \phi$ for any $\sigma \in {\rm{Gal}}(k_s/k)$. But a $k_s$-homomorphism between commutative pseudo-reductive $k_s$-groups is determined by its restriction between maximal $k_s$-tori \cite[Prop.\,1.2.2]{cgp}, so we only need to show that the map that $\phi$ induces between maximal $k_s$-tori is defined over $k$. But any $k_s$-homomorphism between split $k$-tori is defined over $k$,
so we are done.
\end{proof}

Summarizing, we have that
\[
Z_G(S)^{{\rm{ab}}} \simeq \Gm^{n_1} \times (\R_{k'/k}(\Gm))^{n_2}
\]
for some nonnegative integers $n_1, n_2$. In particular, 
\begin{equation}
\label{ZGabTam=1eqn1}
\tau(Z_G(S)^{{\rm{ab}}}) = 1.
\end{equation}

By (\ref{DZ=R(SL2)}), \cite[Chap.\,IV, \S 2.3, Cor.]{oesterle}, and \cite[Satz A]{harder1}, we have ${\rm{H}}^1(k, \mathscr{D}(Z_G(S))) = 1$. By \cite[Prop.\,2.2.4]{cpsurvey}, $\mathscr{D}(Z_G(S))$ is perfect, hence has no nontrivial characters. Finally, (\ref{Z_G(S)opencelleqn1}) and (\ref{DZ_G(S)opencelleqn2}) show that the map $Z_G(S)(\A) \rightarrow Z_G(S)^{{\rm{ab}}}(\A)$ is surjective. By \cite[Chap.\,III, \S 5.3, Thm.]{oesterle}, we therefore have
\[
\tau(Z_G(S)) = \tau(\mathscr{D}(Z_G(S))) \cdot \tau(Z_G(S)^{{\rm{ab}}}),
\]
so (\ref{ZGabTam=1eqn1}) therefore reduces us to showing that $\tau(\mathscr{D}(Z_G(S))) = 1$. Since Tamagawa numbers are invariant under Weil restriction \cite[Chap.\,II, \S 1.3, Thm.\,(d)]{oesterle} and simply connected groups have Tamagawa number $1$, 
we may conclude by (\ref{DZ=R(SL2)}). The proof of Proposition \ref{suitablesplittorus}, hence also that of Proposition \ref{basicexotictam=1}, is complete. \hfill \qed

\section{Tamagawa number formula and inner-twist invariance}
\label{sectiontamagawaformula}

In this section we prove Theorems \ref{tamagawaformula} and \ref{innerinvariance}.

\subsection{Tamagawa number formula}

First we treat the commutative case of Theorem \ref{tamagawaformula}, which is a consequence of Tate duality for affine commutative group schemes of finite type. Before giving the proof, we recall a couple of definitions. 

Let $G$ be a commutative group scheme over a global field $k$. Then we define $\Sha^i(G)$ and $\Che^i(G)$ to be the kernel and cokernel, respectively, of the map ${\rm{H}}^i(k, G) \rightarrow {\rm{H}}^i(\A, G)$. If $G$ is smooth and connected, then by \cite[Prop.\,5.2.2]{rostateduality}, ${\rm{H}}^i(\A, G) = \oplus_v {\rm{H}}^i(k_v, G)$ for $i > 0$, the sum being over all places $v$ of $k$. These definitions of $\Sha$ and $\Che$ therefore agree (for $i > 0$) with the ones given in \cite{oesterle}. This agreement is important because we shall use results from \cite{oesterle} involving these quantities in the proof below. We will actually give two formulas for Tamagawa numbers of commutative linear algebraic groups, one of them a simplification of a formula of Oesterl\'e's.

\begin{theorem}
\label{tamagawaformulacomm}
Let $G$ be a connected commutative linear algebraic group over a global field $k$. Then
\[
\tau(G) = \frac{\# \Che^1(G)\cdot \# \Sha^2(G)}{\# \Sha^1(G)} = \frac{\# \Ext^1(G, \Gm)}{\# \Sha^1(G)}.
\]
\end{theorem}

\begin{proof}
Define $\widehat{G}$ to be the (possibly non-representable) Hom-functor $\mathscr{H}om(G, \Gm)$ on the category of $k$-schemes. By \cite[Ch.\,IV, Thm.\,3.2]{oesterle}, we have
\[
\tau(G) = \frac{\# \Che^1(G) \cdot \# \Sha^2(G)}{\# \Sha^1(G)\cdot \# \ker(\lambda_G)},
\]
where $\lambda_G: \Che^2(G) \rightarrow \widehat{G}(k)^*$ is the map obtained from the last three nonzero terms in the complex of \cite[Thm.\,1.2.8]{rostateduality}. 
The exactness of this complex (which is the content of \cite[Thm.\,1.2.8]{rostateduality}) yields that $\lambda_G$ is an isomorphism! This proves the first equality.

For the second equality, we need to show that $\# \Ext^1(G, \Gm) = \# \Che^1(G)\cdot \# \Sha^2(G)$. As we have already remarked, $\Ext^1(G, \Gm)$ agrees with the derived-functor $\Ext$ \cite[Prop.\,4.3]{rospic}. By \cite[Cor.\,2.3.4]{rostateduality}, which concerns the derived-functor Ext,
we have 
$$\Ext^1(G, \Gm) \simeq {\rm{H}}^1(k, \widehat{G}).$$ Hence, ${\rm{H}}^1(k, \widehat{G})$ is finite by \cite[Thm.\,1.1]{rospic}, and it suffices 
to show $\# {\rm{H}}^1(k, \widehat{G}) = \# \Che^1(G)\cdot \# \Sha^2(G)$. 
But \cite[Thm.\,1.2.8]{rostateduality} provides an exact sequence
\[
0 \longrightarrow \Che^1(G) \longrightarrow {\rm{H}}^1(k, \widehat{G})^* \longrightarrow \Sha^2(G) \longrightarrow 0,
\]
so the $\Q/\Z$-dual ${\rm{H}}^1(k, \widehat{G})^*$ of the finite group ${\rm{H}}^1(k, \widehat{G})$
has size $ \# \Che^1(G)\cdot \# \Sha^2(G)$.  Since any finite abelian group has the same size as its $\Q/\Z$-dual, we are done!
\end{proof}

Now we turn to the proof of Theorem \ref{tamagawaformula} for pseudo-reductive groups, which actually requires us to already know the result for commutative groups. In fact, the classification theorem (Theorem \ref{predclassification}) for pseudo-reductive groups explicitly classifies such groups modulo the commutative case, about which one can say very little. It is therefore essential that we have Theorem \ref{tamagawaformulacomm} for very general commutative groups; we shall use it below for groups that may not even be pseudo-reductive.

Now let $G$ be a pseudo-reductive group over a global function field $k$. Since the commutative case has already been handled by Theorem \ref{tamagawaformulacomm}, Theorem \ref{predclassification} allows us to assume that $G$ is non-commutative and is either generalized standard or a Weil restriction of a basic non-reduced group over a finite extension $k'$ of $k$ (in which case necessarily char$(k) = 2$ and $[k: k^2] = 2$). If ${\rm{char}}(k) > 3$, then $G$ must actually be a standard pseudo-reductive group \cite[Thm.\,5.1.1(1)]{cgp}, so the reader who wishes to avoid extra complications in characteristics $2$ and $3$ may just concentrate on the standard case. In particular, he may ignore all mention of basic exotic and basic non-reduced groups, as well as any invocations of Propositions \ref{basicexotictam=1} and \ref{measuresagreebasic}.

We first treat the generalized standard case. We keep the notation of \S \ref{classificationsection}, so $k'$ is a nonzero finite reduced $k$-algebra, $G'$ is a $k'$-group scheme all of whose fibers over the factor fields of $k'$ are either connected semisimple, absolutely simple, and simply connected, or else basic exotic pseudo-reductive (the latter case is only a possibility when char$(k) = p \in \{2, 3\}$, so the reader willing to ignore small characteristics may ignore it and assume that all fibers are of the former type), $T' \subset G'$ is a maximal $k'$-torus (and we are free to choose $T'$ to be any maximal $k'$-torus in $G'$ that we wish; see Remark \ref{choiceoftorus}), $\mathscr{C} := \R_{k'/k}(Z_{G'}(T'))$  is the Weil restriction of the corresponding Cartan subgroup of $G'$ (so $\mathscr{C}$ and $\R_{k'/k}(T')$ agree for the contribution from factor fields of $k'$ with simply connected fiber), $C$ is a commutative pseudo-reductive group acting on $\R_{k'/k}(G')$ so that the action is trivial on $\mathscr{C} \subset \R_{k'/k}(G')$, and there is a map
\[
\mathscr{C} \xrightarrow{\phi} C
\]
such that the $C$-action on $\R_{k'/k}(G')$ is compatible with the conjugation action of $\mathscr{C}$ via $\phi$. 

We have the anti-diagonal map $\mathscr{C} \rightarrow C \ltimes \R_{k'/k}(G')$ defined by $(-\phi, j)$, where $j: \mathscr{C} \hookrightarrow \R_{k'/k}(G')$ is the inclusion. This map is an isomorphism of $\mathscr{C}$ onto a central subgroup, and $G$ is the quotient
\[
G = (C \ltimes \R_{k'/k}(G'))/\mathscr{C}.
\]
As we remarked above, we are free to choose the maximal torus $T'$, which we do as follows. We have a $k'$-group $\overline{G}'$ that on each fiber over a factor field of $k'$ is simply connected, and such that there is a specific surjective map of $k'$-group schemes 
\begin{equation}
\label{map1}
f: G' \rightarrow \overline{G}'.
\end{equation}
Indeed, for any simply connected fiber $H'$ of $G'$, we take the corresponding fiber of $\overline{G}'$ to also be $H'$ and $f$ to be the identity map. For a basic exotic fiber $H'$ of $G'$, we have the canonical surjection $H' \twoheadrightarrow \overline{H}'$ to a simply connected group $\overline{H}'$ (Remark \ref{canonicalsurjection}), and we take the corresponding fiber of $G'$ to be $\overline{H}'$ and $f$ to be the above surjection. There is a bijection between maximal $k'$-tori of $G'$ and those of $\overline{G}'$ given by sending a maximal torus $\overline{T}'$ of $\overline{G}'$ to the unique maximal torus $T' \subset f^{-1}(\overline{T}')$, and the map $T' \rightarrow \overline{T}'$ is an isogeny. This assertion may be checked fiber by fiber, so there is only something to check on the basic exotic fibers, where it is \cite[Cor.\,7.3.4]{cgp} and was also discussed in \S \ref{classificationsection}. 

Now we use Proposition \ref{torustrivialSha^2} to choose a maximal $k'$-torus $\overline{T}' \subset \overline{G}'$ such that $\Sha^2(k', \overline{T}') = 0$, and we take $T' \subset G'$ to be the maximal torus corresponding to $\overline{T}'$ under the above bijection. The key point to note is that
\begin{equation}
\label{Sha^2=0forourcartan}
\Sha^2(k', Z_{G'}(T')) = 0.
\end{equation}
Indeed, this may be checked fiber by fiber. On the simply connected fibers, the centralizer of $T'$ agrees with  the fiber $T'$, which was chosen to have trivial $\Sha^2$. On the basic exotic fibers, the centralizer has the same (vanishing!) $\Sha^2$ as the torus $\overline{T}'$ in $\overline{G}'$, by Lemma \ref{pointsofcartan} (and the isomorphism $k^{1/p} \otimes_k k_v \simeq k_v^{1/p}$ and compatibility of the basic exotic construction with separable field extensions, such as $k_v/k$).

Before continuing, we need a lemma.

\begin{lemma}
\label{tamagawasimplekernel}
Suppose one has an exact sequence
\[
1 \longrightarrow H' \longrightarrow H \longrightarrow H'' \longrightarrow 1
\]
of connected linear algebraic groups over a global function field $k$ such that ${\rm{H}}^1(k, H') = 1$; ${\rm{H}}^1(k_v, H') = 1$ for all places $v$ of $k$; $\Ext^1(H', \Gm) = 1$; and $\tau(H') = 1$. Then $\tau(H)/\# \Ext^1(H, \Gm) = \tau(H'')/\# \Ext^1(H'', \Gm)$.
\end{lemma}

The list of hypotheses on $H'$ looks rather peculiar, but will be fulfilled in the cases we need. This lemma should be thought of as saying that if the kernel $H'$ is sufficiently ``trivial'' (in the variety of ways listed in the lemma), then the Tamagawa numbers of $H$ and $H''$ are related in a very simple manner.

\begin{proof}
Our assumption that ${\rm{H}}^1(k_v, H') = 1$ for all places $v$ of $k$ implies by Proposition \ref{adeliccohomologyprop} that ${\rm{H}}^1(\A, H') = 1$, so the map $H(\A) \rightarrow H''(\A)$ is surjective. Since ${\rm{H}}^1(k, H') = 1$, we have $\Sha(H') = 1$. Applying \cite[Chap.\,III, Thm.\,5.3]{oesterle}, therefore, together with our assumption that $\tau(H') = 1$, implies that
\begin{equation}
\label{tamsimplepfeqn1}
\tau(H) \cdot \# \coker(\widehat{j}) = \tau(H''),
\end{equation}
where $j: H' \rightarrow H$ is the inclusion and $\widehat{j}: \widehat{H}(k) \rightarrow \widehat{H'}(k)$ is the induced map on $k$-rational character groups. Due to the vanishing of $\Ext^1(H', \Gm)$, \cite[Lemma 3.3]{rospic} furnishes an exact sequence
\[
0 \longrightarrow \coker(\widehat{j}) \longrightarrow \Ext^1(H'', \Gm) \longrightarrow \Ext^1(H, \Gm) \longrightarrow 0.
\]
This implies that $\# \coker(\widehat{j}) = \# \Ext^1(H'', \Gm)/\# \Ext^1(H, \Gm)$, which together with (\ref{tamsimplepfeqn1}) proves the lemma.
\end{proof}

Returning to the proof of Theorem \ref{tamagawaformula}, we now construct an inclusion $\mathscr{C} \hookrightarrow B$ of connected commutative linear algebraic $k$-groups such that $B$ satisfies all of the requirements of the $H'$ in Lemma \ref{tamagawasimplekernel}. That is: $\Ext^1(B, \Gm) = 1$, $\tau(B) = 1$, ${\rm{H}}^1(k, B) = 1$, and ${\rm{H}}^1(k_v, B) = 1$ for all places $v$ of $k$. We will also construct $B$ so that $\Sha^2(B) = 0$. To see that there exists an inclusion into such a $B$, note that for some finite extension $k'/k$, $\mathscr{C}_{k'}$ is the product of a split torus and a split unipotent group, since this holds over $\overline{k}$ (for any connected commutative linear algebraic group) and hence over some finite extension of $k$. The canonical inclusion $\mathscr{C} \hookrightarrow \R_{k'/k}(\mathscr{C}_{k'})$ does the job. Indeed, first note that the latter group has vanishing Pic, hence vanishing $\Ext^1(\cdot, \Gm)$. Since $\tau$, ${\rm{H}}^1(k, \cdot), {\rm{H}}^1(k_v, \cdot)$, and $\Sha^2$ are invariant under Weil restriction ($\tau$ by \cite[Chap.\,II, Thm.\,1.3(d)]{oesterle}, and the cohomology groups due to the exactness of pushforward through a finite map between categories of \'etale abelian sheaves), we only need to show that $\Gm$ and commutative split unipotent groups have Tamagawa number $1$, vanishing ${\rm{H}}^1$ (over any field), and vanishing $\Sha^2$. 

The vanishing of the ${\rm{H}}^1$'s is well-known and completely general, as is the vanishing of ${\rm{H}}^2(k', U')$ for commutative unipotent $U'$ (\cite[Prop.\,2.5.4]{rostateduality}, though we only require the split case, which is even more basic), while the vanishing of $\Sha^2(k', \Gm)$ follows from class field theory. The normalization factors in the definition of Tamagawa numbers are essentially defined so that $\tau(\Gm) = 1$ \cite[Chap.\,VII, \S \S 5-6]{weil}; indeed, this equality is equivalent to the formula for the pole of the Dedekind zeta function of $k$ at $1$ in the number field case, while in the function field case that we are interested in now, it is equivalent to computing the pole at $1$ of the zeta function of the (smooth proper geometrically connected) curve $X$ of which $k$ is the function field. To show that $\tau(U') = 1$ for $U'$ split unipotent, it suffices by \cite[Chap.\,III, Thm.\,5.3]{oesterle} to show that $\tau(\Ga) = 1$, which also holds essentially by definition \cite[Chap.\,I, \S 5.14, Example 1]{oesterle}.

Now consider the following pushout diagram, which makes sense because the top extension is central
\begin{equation}
\label{diagram1}
\begin{tikzcd}
1 \arrow{r} & \mathscr{C} \arrow{r} \arrow[d, hookrightarrow] & C \ltimes \R_{k'/k}(G') \arrow[d, hookrightarrow, "t"] \arrow{r}{\psi} & G \arrow{r} \arrow[d, equals] & 1 \\
1 \arrow{r} & B \arrow{d} \arrow{r}{i} & F \arrow{d} \arrow{r} & G \arrow{r} & 1 \\
& D \arrow[r, equals] & D &&
\end{tikzcd}
\end{equation}
where the central pushout $F$ and the commutative $k$-group $D$ are defined by this diagram, and the columns are short exact. Lemma \ref{tamagawasimplekernel} applied to the second row of (\ref{diagram1}) yields
\begin{equation}
\label{tamproofeqn2}
\frac{\tau(F)}{\# \Ext^1(F, \Gm)} = \frac{\tau(G)}{\# \Ext^1(G, \Gm)}.
\end{equation}
We claim that the map $\Sha(F) \rightarrow \Sha(G)$ is a bijection. Assuming this, (\ref{tamproofeqn2}) reduces Theorem \ref{tamagawaformula} for $G$ to the same theorem for $F$. So let us prove this claim.

First we prove injectivity. Since the extension
\begin{equation}
\label{bfgseq1}
1 \longrightarrow B \longrightarrow F \longrightarrow G \longrightarrow 1
\end{equation}
is central, the abelian group ${\rm{H}}^1(k, B)$ acts on the set ${\rm{H}}^1(k, F)$ and the orbits for this action are exactly the fibers of the map ${\rm{H}}^1(k, F) \rightarrow {\rm{H}}^1(k, B)$ \cite[Ch.\,I, \S 5.7, Prop.\,42]{serre}. Since ${\rm{H}}^1(k, B) = 0$ by construction, the map ${\rm{H}}^1(k, F) \rightarrow {\rm{H}}^1(k, B)$ is therefore injective. To prove surjectivity of $\Sha(F) \rightarrow \Sha(G)$, we first note that the centrality of (\ref{bfgseq1}) implies that there is a connecting map ${\rm{H}}^1(k, B) \rightarrow {\rm{H}}^2(k, B)$ inducing the usual long exact cohomology sequence of pointed sets up to ${\rm{H}}^2(k, B)$ \cite[Ch.\,I, \S 5.7, Prop.\,43]{serre}. The vanishing of $\Sha^2(B)$ therefore implies that any element of $\Sha^1(G)$ lifts to ${\rm{H}}^1(k, F)$. But since ${\rm{H}}^1(k_v, B) = 0$ for every place $v$ of $k$, this lift must actually lie in $\Sha^1(F)$, hence the map $\Sha(F) \rightarrow \Sha(G)$ is surjective. This proves the claim.

We are therefore reduced to proving Theorem \ref{tamagawaformula} for the group $F$. The group $\R_{k'/k}(G')$ sits as a normal subgroup of $F$ with commutative cokernel. That is, we have an exact sequence
\begin{equation}
\label{exactseq2}
1 \longrightarrow \R_{k'/k}(G') \longrightarrow F \longrightarrow E \longrightarrow 1
\end{equation}
with $E$ a connected commutative linear algebraic group.

\begin{lemma}
\label{Ext^1vanishessc}
Let $k$ be a field, $H$ a connected linear algebraic $k$-group that is perfect $($i.e., $H = \mathscr{D}H$$)$, and assume that the geometric reductive quotient $(H_{\overline{k}})^{{\rm{red}}}$ is simply connected. Then $\Ext^1(H, \Gm) = 1$.
\end{lemma}

\begin{proof}
Suppose that we have a (necessarily central) extension
\[
1 \longrightarrow \Gm \longrightarrow E \longrightarrow H \longrightarrow 1.
\]
We will show that it admits a section, hence splits. Indeed, we will show that the map $\mathscr{D}E \rightarrow H$ is an isomorphism. For this, we may extend scalars to $\overline{k}$. First, since $H = \mathscr{D}H$, we have an exact sequence
\[
1 \longrightarrow \mu \longrightarrow \mathscr{D}E \xlongrightarrow{\pi} H \longrightarrow 1.
\]
We claim that $\mu$ is finite. Indeed, more generally, for any connected linear algebraic group $G$ over $k$, and any central $\mathbf{G}_m \subset G$, we claim $\mathscr{D}G \cap \mathbf{G}_m$ is finite. To see this, let $U$ be the unipotent radical of $G$. Then $\mathscr{D}G \cap \mathbf{G}_m \subset \mathbf{G}_m \hookrightarrow G/U$, hence we may assume that $G$ is reductive. But then the claim follows from the well-known fact that for any reductive $G$ and its maximal central torus $Z$, $\mathscr{D}G \cap Z$ is finite. This proves the finiteness of $\mu$.

It only remains to show that $\mu = 1$. But since $\mu \subset \Gm$, we have $\mu \hookrightarrow \mathscr{D}E/U$, where $U \subset \mathscr{D}E$ is the unipotent radical. Hence we have an exact sequence
\[
1 \longrightarrow \mu \longrightarrow \mathscr{D}E/U \xlongrightarrow{\pi} H/\pi(U) \longrightarrow 1.
\]
Since $\mathscr{D}E/U$ is reductive, so is $H/\pi(U)$, which is therefore the geometric reductive quotient of $H$, hence simply connected by assumption. It therefore has no nontrivial finite central covers by a connected linear algebraic group, so $\mu = 1$ as claimed.
\end{proof}

\begin{lemma}
\label{R_{k'/k}(G')trivial}
Let $k$ be a global function field, $k'/k$ a finite reduced $k$-algebra, $G'$ a $k'$-group such that the fiber of $G'$ over each factor field of $k'$ is either connected, semisimple, absolutely simple, and simply connected, or else basic exotic pseudo-reductive. Then $\R_{k'/k}(G')$ is equal to its own derived group, $\tau(\R_{k'/k}(G')) = 1$, ${\rm{H}}^1(k, \R_{k'/k}(G')) = 1$, ${\rm{H}}^1(k_v, \R_{k'/k}(G')) = 1$ for each place $v$ of $k$, and $\Ext^1(\R_{k'/k}(G'), \Gm) = 1$.
\end{lemma}

\begin{proof}
Let us verify the various assertions one by one. We first need to show that $\tau(\R_{k'/k}(G'))$ $= 1$. We may work on each factor field of $k'$ and thereby assume that $k'$ is a field. Since Tamagawa numbers are invariant under Weil restriction \cite[Chap.\,II, Thm.\,1.3(d)]{oesterle}, it suffices to show that $\tau(G') = 1$ when $G'$ is either simply connected or basic exotic pseudo-reductive. For simply connected $G'$, this is the content of Weil's Tamagawa number conjecture, recently resolved over function fields by Gaitsgory and Lurie \cite{lurie}. For basic exotic groups, this is part of Proposition \ref{basicexotictam=1}.

Next we check the vanishing of ${\rm{H}}^1(k, \R_{k'/k}(G'))$ and 
\[
{\rm{H}}^1(k_v, \R_{k'/k}(G')) = \prod_{w \mid v} {\rm{H}}^1(k'_w, \R_{k'_w/k_v}(G'))
\]
where the product is over all places $w$ of $k'$ lying above $v$. By \cite[Chap.\,IV, \S 2.3, Cor.]{oesterle}, ${\rm{H}}^1(k, \R_{k'/k}(G')) = {\rm{H}}^1(k', G')$ and similarly for the cohomology groups of the local fields, so we may assume that $k' = k$ in order to show that these sets are trivial. We may also work fiber by fiber and thereby assume that $G'$ is either simply connected or basic exotic pseudo-reductive. In the simply connected case, the desired vanishing is \cite[Thm.\,4.7]{bruhattits} for non-archimedean local fields, and \cite[Satz A]{harder1} for global function fields. The basic exotic case follows from the simply connected one and \cite[Prop.\,7.3.3(1)]{cgp}.

Next we check that $\R_{k'/k}(G')$ is perfect (i.e., equal to its own derived group). For this, we may assume that $k'/k$ is a finite extension of fields, and that $G'$ is either simply connected or basic exotic pseudo-reductive. The perfectness of $\R_{k'/k}(G')$ for simply connected $G'$ is \cite[Prop.\,2.2.4]{cpsurvey}. The perfectness of this group for basic exotic $G'$ is part of \cite[Prop.\,8.1.2]{cgp}.

Finally, it remains to check that $\Ext^1(\R_{k'/k}(G'), \Gm) = 1$. Since $\R_{k'/k}(G')$ is perfect, it suffices by Lemma \ref{Ext^1vanishessc} to show that $\R_{k'/k}(G')$ has simply connected geometric reductive quotient. This may be checked fiberwise, so we may assume that $k'/k$ is a finite extension of fields, and that $G'$ is either simply connected or basic exotic pseudo-reductive. We may extend scalars to $k_s$ and work on fibers to thereby arrange that $k'/k$ is a finite purely inseparable extension. 

Consider the map $\R_{k'/k}(G')_{k'} \rightarrow G'$ defined functorially on $k'$-algebras $R'$ by the map $G'(R' \otimes_k k') \rightarrow G'(R')$ induced by the ring homomorphism $R' \otimes_k k' \rightarrow R'$ defined on pure tensors by $r' \otimes \lambda' \mapsto \lambda'r'$. This map is surjective with smooth connected unipotent kernel \cite[\S A.3.6]{oesterle}, so $\R_{k'/k}(G')$ has the same geometric reductive quotient as $G'$ (when $k'/k$ is finite and purely inseparable). Thus, we are done when $G'$ is simply connected, and it only remains to show that a basic exotic group $H$ over a field $F$ has simply connected geometric reductive quotient. But since by definition (Definition \ref{basicexotic}(iii)), $H_{F_s}$ contains a Levi $F_s$-subgroup of $\R_{F'_s/F_s}(H'_{F'_s})$, with $H'$ simply connected and $F'_s/F_s$ a purely inseparable finite extension, it suffices to show that $\R_{F'_s/F_s}(H'_{F'_s})$ has simply connected geometric reductive quotient, and this follows from the argument given above.
\end{proof}

We record the following lemma here for future use.

\begin{lemma}(\cite[Prop.\,6.4.1]{conrad})
\label{standardetaletwist}
Let $k$ be a field, $k'/k$ a nonzero finite reduced $k$-algebra, and $G'$ a $k'$-group such that each fiber $G_i'$ of $G'$ on a factor field $k'_i$ of $k'$ is either connected semisimple, absolutely simple, and simply connected, or else basic exotic pseudo-reductive. Let $G := \R_{k'/k}(G')$. Then any \'etale $k$-form $H$ of $G$ is of the same type. That is, $H \simeq \R_{k''/k}(G'')$ for some nonzero finite reduced $k$-algebra $k''$ and $k''$-group $G''$ all of whose fibers are either connected semisimple, absolutely simple, and simply connected, or else basic exotic pseudo-reductive.
\end{lemma}

Returning to the proof of Theorem \ref{tamagawaformula} for the group $F$ appearing in (\ref{exactseq2}), Lemmas \ref{R_{k'/k}(G')trivial} and \ref{tamagawasimplekernel} imply that
\begin{equation}
\label{tau/exteqn3}
\frac{\tau(F)}{\# \Ext^1(F, \Gm)} = \frac{\tau(E)}{\# \Ext^1(E, \Gm)}.
\end{equation}
Since $E$ is commutative, Theorem \ref{tamagawaformulacomm} therefore implies that in order to show that $\tau(F) = \# \Ext^1(F, \Gm)/\# \Sha(F)$, and thereby complete the proof of Theorem \ref{tamagawaformula} except for the yet-to-be-treated basic non-reduced case, it suffices to prove that the map $\Sha(F) \rightarrow \Sha(E)$ is a bijection. 

The injectivity of this map is easy (though based on deep results), as follows. By the usual twisting argument (see, e.g., the first paragraph of the proof of Proposition \ref{connectingmap=bijection}), it suffices to show that ${\rm{H}}^1(k, H) = 1$ for any \'etale $k$-form $H$ of $\R_{k'/k}(G')$. In fact, by Lemma \ref{standardetaletwist}, any such $H$ is itself of the form $\R_{k''/k}(H'')$ for some finite reduced $k$-algebra $k''$ and some $k''$-group $H''$ all of whose fibers are either simply connected or basic exotic. Thus $H$ has vanishing ${\rm{H}}^1$ by an argument we have already seen: we may work on each factor field of $k''$ and thereby assume that $k''$ is a field. \cite[Ch.\,IV, \S 2.3, Cor.\,]{oesterle} then reduces us to showing that ${\rm{H}}^1(k'', H'') = 1$ so we may assume that $k'' = k$. If $H''$ is simply connected, then the desired vanishing is \cite[Satz A]{harder1}, while if $H''$ is basic exotic, then it follows from the same result together with the fact that the canonical map $H'' \rightarrow \overline{H}''$ onto a simply connected group (see Remark \ref{canonicalsurjection}) induces a bijection ${\rm{H}}^1(k'', H'') \rightarrow {\rm{H}}^1(k'', \overline{H}'')$ \cite[Prop.\,7.3.3(1)]{cgp}.

The proof that the map $\Sha(F) \rightarrow \Sha(E)$ is surjective is trickier. This is where our careful choice of maximal torus $T' \subset G'$ via Proposition \ref{torustrivialSha^2} comes in. First note that it suffices to lift any element of $\Sha(E)$ to ${\rm{H}}^1(k, F)$, as this lift is then necessarily in $\Sha(F)$ due to the vanishing of ${\rm{H}}^1(k_v, \R_{k'/k}(G'))$ that we have already seen. Let $M$ be the commutative $k$-group scheme defined by the following pushout diagram
\[
\begin{tikzcd}
\mathscr{C} \arrow{r} \arrow[d, hookrightarrow] & C \times \mathscr{C} \arrow[d, hookrightarrow] \\
B \arrow{r} & M
\end{tikzcd}
\]
where the top horizontal map is the anti-diagonal map $(-\phi, {\rm{id}})$. 

Since the action of $C$ on $\R_{k'/k}(G')$ is trivial on $\mathscr{C}$ by definition, we get a commutative diagram with exact rows
\[
\begin{tikzcd}
1 \arrow{r} & \mathscr{C} \arrow{d}{j} \arrow{r} & M \arrow{r} \arrow{d} & E \arrow{r} \arrow[d, equals] & 1 \\
1 \arrow{r} & \R_{k'/k}(G') \arrow{r} & F \arrow{r} & E \arrow{r} & 1
\end{tikzcd}
\]
where the map $M \rightarrow F$ is induced by the two maps $i: B \rightarrow F$ and 
\[
C \times \mathscr{C} \xrightarrow{{\rm{id}}\times j} C \ltimes \R_{k'/k}(G') \xrightarrow{t} F
\]
(see (\ref{diagram1}); recall that $j$ is the inclusion map). In order to show that every element of $\Sha(E)$ lifts to ${\rm{H}}^1(k, F)$, therefore, it suffices to show they all lift to ${\rm{H}}^1(k, M)$. But for this it suffices to show that $\Sha^2(k, \mathscr{C}) = 0$. Since $\mathscr{C} = \R_{k'/k}(Z_{G'}(T'))$, where $T'$ was the torus chosen at the beginning of this proof, and since pushforward by a finite map is an exact functor between categories of abelian \'etale sheaves, we have $\Sha^2(k, \mathscr{C}) = \Sha^2(k', Z_{G'}(T'))$, and this vanishes by (\ref{Sha^2=0forourcartan}). The proof of Theorem \ref{tamagawaformula} is now complete for generalized standard pseudo-reductive groups.

Theorem \ref{predclassification} therefore implies that the proof of Theorem \ref{tamagawaformula} is complete for fields of characteristic greater than $2$. If ${\rm{char}}(k) = 2$, we may assume that $G = \R_{k'/k}(G')$ for a global field $k$ of characteristic $2$, a finite reduced $k$-algebra $k'/k$, and a basic non-reduced $k'$-group $G'$. We will show that all three of the quantities appearing in Theorem \ref{tamagawaformula} are $1$. That is: $\Sha(G) = 1$, $\Ext^1(G, \Gm) = 1$, and $\tau(G) = 1$. By Lemma \ref{Ext^1vanishessc}, $\Ext^1(G, \Gm)$ vanishes if $G$ is perfect and $G_{\overline{k}}^{{\rm{red}}}$ is semisimple and simply connected. The perfectness is part of \cite[Ex.\,10.1.3]{cgp}. To identify $G_{\overline{k}}^{{\rm{red}}}$ we can make a finite separable extension on $k$ and thereby arrange that each factor field of $k'$ is purely inseparable over $k$. Passing to $k'$ a field without loss of generality, by \cite[Ex.\,10.1.3]{cgp}, $G_{k_s}$ has root system BC$_n$ for some $n \geq 1$, so $G_{\overline{k}}^{{\rm{red}}} \simeq {\rm{Sp}}_{2n}$ by the end of \cite[Thm.\,2.3.10]{cgp}.

By \cite[Ch.\,IV, \S 2.3, Cor.\, and Ch.\,II, \S 1.3, Thm.\,(d)]{oesterle}, for $\Sha$ and $\tau$ we can replace $(G, k)$ with $(G', k')$ and thereby assume that $G$ is basic non-reduced. The vanishing of $\Sha(G)$ then follows from \cite[Prop.\,9.9.4(1)]{cgp}. So it only remains to show that $\tau(G) = 1$ for basic non-reduced $G$. This is part of Proposition \ref{basicexotictam=1}. The proof of Theorem \ref{tamagawaformula} is now complete!

\begin{remark}
\label{PicneqExtexamples}
Upon comparing Theorem \ref{tamagawaformula} with Sansuc's formula (\ref{sansucformula}), it is natural to ask whether $\Ext^1(G, \Gm) = \Pic(G)$ for all pseudo-reductive groups. The answer in general, even for commutative pseudo-reductive groups, is no. One may construct counterexamples as follows.

Let $U$ be a wound form of $\Ga$. If either $p \neq 2$ or $p = 2$ but $U$ is not $k$-isomorphic to a group of the form $\{Y^2 = X + aX^2\} \subset \Ga \times \Ga$ for some $a \in k - k^2$, then for some finite separable extension $k'/k$, $\Ext^1_{k'}(U, \Gm) \neq 0$, by \cite[Props.\,5.5, 5.7]{rospic}. Renaming $k'$ as $k$, our assumptions about $U$ also imply that $\Ext^1(U, \Gm) \neq \Pic(U)$ \cite[Prop.\,5.12]{rospic}.

At any rate, by passing to a finite separable extension, we obtain (in any characteristic) a form $U$ of $\Ga$ over a global function field $k$ such that $0 \subsetneq \Ext^1(U, \Gm) \subsetneq \Pic(U)$. Let $E$ be a group scheme representing a nontrivial extension of $U$ by $\Gm$. That is, the (nontrivial) extension class is given by an exact sequence
\begin{equation}
\label{commpredexeqn}
0 \longrightarrow \Gm \longrightarrow E \xlongrightarrow{\pi} U \longrightarrow 0.
\end{equation}
Note that $E$ is commutative by \cite[Lemma 4.2]{rospic}. We claim that $E$ is pseudo-reductive.

Indeed, if not, then $E$ contains a nontrivial smooth connected unipotent $k$-subgroup $W$. Since $W$ is unipotent, $W \cap \Gm = 0$, so $\pi$ induces an inclusion $W \rightarrow U$, which is necessarily an isomorphism because $U$ is $1$-dimensional, connected, and smooth. But this then yields a splitting of (\ref{commpredexeqn}), violating the nontriviality of that extension. So $E$ is commutative pseudo-reductive. But we claim that $\Ext^1(E, \Gm) \neq \Pic(E)$.

Suppose to the contrary that $\Ext^1(E, \Gm) = \Pic(E)$. Then we have a commutative diagram
\[
\begin{tikzcd}
\widehat{\Gm}(k) \arrow{r} \arrow[d, equals] & \Ext^1(U, \Gm) \arrow{r} \arrow[d, hookrightarrow] & \Ext^1(E, \Gm) \arrow[d, equals] \arrow{r} & 0 \\
\widehat{\Gm}(k) \arrow{r} & \Pic(U) \arrow{r} & \Pic(E) &
\end{tikzcd}
\]
in which the rows are exact by \cite[Lemma 3.3]{rospic} and \cite[Cor.\,6.11]{sansuc} (the corollary is stated for exact sequences of reductive groups, but the proof only requires that the group on the left -- in this case $\Gm$ -- be reductive), and because $\Ext^1(\Gm, \Gm) \subset \Pic(\Gm) = 0$. A simple diagram chase then shows that the middle vertical inclusion is actually an isomorphism, contrary to our choice of $U$.
\end{remark}

\subsection{Invariance under inner twist}

Now we turn to the proof of Theorem \ref{innerinvariance}, which, though simpler, proceeds along similar lines to that of Theorem \ref{tamagawaformula}. We first prove the inner-twisting invariance of $\Ext^1(G, \Gm)$ (Proposition \ref{extinvtinnertwist}), which holds for smooth connected group schemes over arbitrary fields. This will require some preparatory lemmas.

\begin{lemma}
\label{mapofcenterssurj}
If $G$ is a smooth connected group scheme over a field $k$, and 
\[
1 \longrightarrow \Gm \longrightarrow E \xlongrightarrow{\pi} G \longrightarrow 1
\]
is an extension of $G$ by $\Gm$, then the induced map $Z_E \rightarrow Z_G$ from the center of $E$ to the center of $G$ is fppf surjective.
\end{lemma}

\begin{proof}
First note that $\Gm \subset E$ is {\em central}. Indeed, conjugation induces a map $E \rightarrow {\rm{Aut}}_{\Gm/k}$, and the latter is an \'etale $k$-scheme (in fact, it is $\Z/2\Z$), hence, since $E$ is connected, the map must be trivial. That is, $\Gm \subset E$ is central. Consider the bilinear map $E \times \pi^{-1}(Z_G) \rightarrow \Gm$ obtained by taking commutators; that is, the map is given by $(e, z) \mapsto eze^{-1}z^{-1}$. Since $\Gm \subset E$ is central, this descends to a bilinear map $G \times Z_G \rightarrow \Gm$. The statement that $\pi^{-1}(Z_G) \subset E$ is central is equivalent to the triviality of this bilinear map.

The bilinear map above yields a homomorphism $\phi: G \rightarrow \calHom(Z_G, \Gm)$, which we want to show to be trivial. In order to do this, we may assume that $k = \overline{k}$. Then by Chevalley's Theorem, $G$ is an extension
\[
1 \longrightarrow L \longrightarrow G \longrightarrow A \longrightarrow 1
\]
with $L$ a connected linear algebraic $k$-group and $A$ an abelian variety over $k$. Let $U \subset L$ be the $k$-unipotent radical. Then $U$ is normal in $G$ because it is preserved by all $k$-automorphisms of the normal subgroup $L$, hence preserved by $G(k)$-conjugation, and $G(k)$ is Zariski dense in $G$ because $G$ is smooth. For any $g \in G(k)$, we have $\phi(g)|_{U \cap Z_G} = 1$, since subgroups of unipotent groups admit no nontrivial homomorphisms to $\Gm$ over a field. Therefore, $E(k)$ centralizes $\pi^{-1}(U\cap Z_G)$, hence so does $E$, since $E(k)$ is Zariski dense in $E$ because $E$ is smooth and $k = k_s$. It follows that $\phi$ factors as a map (which we still call $\phi$, by abuse of notation) $\phi: G \rightarrow \calHom(Z_G/(U \cap Z_G), \Gm) \subset \calHom(Z_G, \Gm)$, and we want to show that this map is trivial.

Let $\overline{G} := G/U$. Then $\overline{Z_{G}} := Z_G/(U \cap Z_G) \subset Z_{\overline{G}}$. The group $\overline{G}$ is an extension
\[
1 \longrightarrow H \longrightarrow \overline{G} \longrightarrow A \longrightarrow 1
\]
with $H$ reductive (and $A$ still an abelian variety). Therefore, we have an exact sequence
\[
0 \longrightarrow Z_1 \longrightarrow \overline{Z_G} \longrightarrow Z_2 \longrightarrow 0
\]
where $Z_1 \subset Z_{H}$ is a central $k$-subgroup of the reductive group $H$, and $Z_2 \subset A$ is a closed $k$-subgroup. 

We claim that the composition $G \xrightarrow{\phi} \calHom(\overline{Z_G}, \Gm) \rightarrow \calHom(Z_1, \Gm)$ is trivial. Indeed, since $Z_1$ is a central $k$-subgroup of the reductive $k$-group $H$, it fits into an exact sequence
\[
0 \longrightarrow T \longrightarrow Z_1 \longrightarrow F \longrightarrow 0
\]
with $T$ a $k$-torus and $F$ a finite commutative $k$-group scheme. The composition $G \rightarrow \calHom(Z_1, \Gm) \rightarrow \calHom(T, \Gm)$ is trivial, since $G$ is connected while $\calHom(T, \Gm)$ is represented by an \'etale $k$-group scheme. Therefore, the map $G \rightarrow \calHom(Z_1, \Gm)$ factors through a map $G \rightarrow \calHom(F, \Gm)$. But $\calHom(F, \Gm)$ is represented by a finite commutative $k$-group scheme (the Cartier dual of $F$), so since $G$ is smooth and connected, this map too is trivial, hence the map $G \rightarrow \calHom(Z_1, \Gm)$ is trivial, as claimed.

The map $\phi$ therefore factors as a map $\phi: G \rightarrow \calHom(Z_2, \Gm)$. Arguing as above, it suffices to filter $Z_2$ by $k$-group schemes $F_i$ such that the sheaves $\calHom(F_i, \Gm)$ are represented by finite $k$-group schemes. In fact, one can show by a simple trick with prime-to-characteristic torsion that $Z_2$ is an extension of a finite group scheme by an abelian variety, but we do not need this. By \cite[VII$_{\rm{A}}$, Prop.\,8.3]{sga3}, there is an infinitesimal subgroup scheme $I \subset Z_2$ such that $\overline{Z}_2 := Z_2/I$ is smooth. The sheaf $\calHom(I, \Gm)$ is the Cartier dual of $I$, hence represented by a finite commutative $k$-group scheme. Further, $A/I$ is still an abelian variety, so $\overline{Z}_2^0$ is an abelian variety, hence $\calHom(\overline{Z}_2^0, \Gm) = 0$. (That there is no nontrivial homomorphism over any base from an abelian scheme to $\Gm$ follows from the assertion over a field together with \cite[Prop.\,6.1]{git}.) Finally, $\overline{Z}_2/\overline{Z}_2^0$ is a finite (\'etale) commutative $k$-group scheme, hence its Cartier dual is a finite $k$-group scheme. This completes the proof of the lemma.
\end{proof}

\begin{lemma}
\label{conjinducesid}
For a smooth connected group scheme $G$ over a field $k$ and $g \in (G/Z_G)(k)$, ``conjugation'' by $g$ on $G$ induces the identity map on $\Ext^1(G, \Gm)$.
\end{lemma}

\begin{proof}
Consider an extension $E$ of $G$ by $\Gm$, and pull it back by ``conjugation'' by $g$:
\[
\begin{tikzcd}
1 \arrow{r} & \Gm \arrow{r} \arrow[d, equals] & E' \arrow[dr, phantom, "\square"] \arrow{r}{\pi'} \arrow{d}{F} & G \arrow{r} \arrow{d}{x \mapsto gxg^{-1}} & 1 \\
1 \arrow{r} & \Gm \arrow{r} & E \arrow{r}{\pi} & G \arrow{r} & 1
\end{tikzcd}
\]
Since $E'$ is a central extension of $G$, $G$ acts on $E'$ by ``conjugation''. This descends to an action by $G/Z_G$ by Lemma \ref{mapofcenterssurj}. By abuse of notation, we still denote this action using multiplication notation. Then one easily checks that the map $E' \rightarrow E$ defined by $e' \mapsto F(g^{-1}e'g)$ is an isomorphism of extensions.
\end{proof}

\begin{proposition}
\label{extinvtinnertwist}
Let $G$ be a smooth connected group scheme over a field $k$, and let $G'$ be an inner form of $G$. Then the groups $\Ext^1(G, \Gm)$ and $\Ext^1(G', \Gm)$ are canonically isomorphic, by an isomorphism depending only upon the cohomology class in ${\rm{H}}^1(k, G/Z_G)$ by which one twists to get $G'$.
\end{proposition}

\begin{proof}
Let $\alpha \in {\rm{H}}^1(k, G/Z_G)$, and fix a torsor/cocycle $x$ (whichever language you are more comfortable with) representing $\alpha$. Suppose we have an extension
\begin{equation}
\label{extensioneqn3}
1 \longrightarrow \Gm \longrightarrow E \longrightarrow G \longrightarrow 1
\end{equation}
of $G$ by $\Gm$. Since the extension is central, $G$ acts on $E$ by ``conjugation''. Since $Z_E \rightarrow Z_G$ is surjective by Lemma \ref{mapofcenterssurj}, $G/Z_G$ acts on $E$ by ``conjugation'', so we may twist (\ref{extensioneqn3}) by $x$ (see \cite[Ch.\,I, \S5.3]{serre}) to obtain a new sequence
\[
1 \longrightarrow \Gm \longrightarrow E_x \longrightarrow G_x \longrightarrow 1.
\]
This map $\Ext^1(G, \Gm) \rightarrow \Ext^1(G_x, \Gm)$ is easily checked to be a well-defined homomorphism. It is an isomorphism, as may be seen by twisting back from $G_x$ to $G$. 

This yields an isomorphism from $\Ext^1(G, \Gm)$ to $\Ext^1(G_x, \Gm)$, but why is this isomorphism canonical, i.e., independent of the class of $x$? Given another torsor $y$ representing $\alpha$, the groups $G_x$ and $G_y$ are $k$-isomorphic, but noncanonically so. This isomorphism is canonical, however, up to ``conjugation'' by an element of $(G_x/Z_{G_x})(k)$. (This is {\em not} the same thing as conjugation by an element of $G_x(k)$ in general!) Such conjugation induces the identity map on $\Ext^1(G_x, \Gm)$, by Lemma \ref{conjinducesid}. Hence the isomorphism on Ext groups is independent of the choice of torsor representing $\alpha$.
\end{proof}

We will now prove the inner invariance of $\tau(G)$, and that of $\# \Sha(G)$ then follows from the already-proved Theorem \ref{tamagawaformula} and Proposition \ref{extinvtinnertwist}. We begin with our pseudo-reductive $k$-group $G$, which we may assume to be either generalized standard or totally non-reduced (since inner twisting -- and in fact any $k_s$-automorphism of $G$ -- preserves the product decomposition in Theorem \ref{predclassification}, because the non-reduced component represents the maximal totally non-reduced subgroup, and the other component is the maximal reduced subgroup). First we treat the totally non-reduced case (which only occurs in characteristic $2$, so the reader willing to ignore this case may skip this paragraph). By Proposition \ref{measuresagreebasic} and the invariance of Tamagawa numbers under Weil restriction \cite[Ch.II, Thm.\,1.3(d)]{oesterle}, totally non-reduced pseudo-reductive groups have Tamagawa number $1$, so it suffices to show that an inner form of such a group is still totally non-reduced pseudo-reductive. In fact, this holds for any \'etale form, as being totally non-reduced is a condition on the root system of $G_{k_s}$, while pseudo-reductivity may be checked over $k_s$ \cite[Prop.\,1.1.9(1)]{cgp}.

Now assume that $G$ is a generalized standard pseudo-reductive group over a global function field $k$. Let us keep the notation from the proof of Theorem \ref{tamagawaformula} for generalized standard groups. Let $x \in Z^1(k, G/Z_G)$ be a $1$-cocycle. We need to show that $\tau(G_x) = \tau(G)$. Since the extension in the second row of (\ref{diagram1}) is central, $G$ acts on $F$ by conjugation. In order to show that $G/Z_G$ also acts on $F$, we need to show that the map $Z_F \rightarrow Z_G$ is fppf surjective. This is the content of the following lemma.

\begin{lemma}
\label{centerssurj}
In the notation above, the map $Z_F \rightarrow Z_G$ is fppf surjective.
\end{lemma}

\begin{proof}
It suffices to show that the map $Z_{C \ltimes \R_{k'/k}(G')} \rightarrow Z_G$ is fppf surjective. We have a bilinear map
\[
(C \ltimes \R_{k'/k}(G')) \times \psi^{-1}(Z_G) \rightarrow \mathscr{C}
\]
defined by taking commutators. That is, the map is defined by $(x, y) \mapsto xyx^{-1}y^{-1}$. We want to show that this map is trivial. Since $\R_{k'/k}(G')$ is equal to its own derived group, as we saw immediately after the proof of Lemma \ref{Ext^1vanishessc}, the induced map of fppf sheaves $1 \times \R_{k'/k}(G') \rightarrow \calHom(\psi^{-1}(Z_G), \mathscr{C})$ is trivial because the image is commutative. Therefore, $1 \times \R_{k'/k}(G')$ commutes with $\psi^{-1}(Z_G)$.

Now suppose that $x = (c, r) \in \psi^{-1}(Z_G) \subset C \ltimes \R_{k'/k}(G')$. (More precisely, $x$ is an $A$-valued point of this group for some $k$-algebra $A$.) Let $(1, r') \in 1 \times \mathscr{C} \subset C \ltimes \R_{k'/k}(G')$. As we have just seen, $x$ commutes with $(1, r')$. Computing the $\R_{k'/k}(G')$-component of both sides of
\[
(c, r)(1, r') = (1, r')(c, r),
\]
and using the fact that the $C$-action on $\R_{k'/k}(G')$ restricts to the trivial action on $\mathscr{C}$, we find that
\[
rr' = r'r.
\]
Since $r' \in 1 \times \mathscr{C}$ was arbitrary, we find that $r$ centralizes $\mathscr{C} \subset \R_{k'/k}(G')$. But we claim that 
\begin{equation}
\label{Z_R(C)=C}
Z_{\R_{k'/k}(G')}(\mathscr{C}) = \mathscr{C},
\end{equation}
which would imply that $r \in \mathscr{C}$. This would then imply that $x$ commutes with $C \times 1$. Since $x \in \psi^{-1}(Z_G)$ was arbitrary, it would then follow that $\psi^{-1}(Z_G)$ commutes with $C \times 1$. Since we have already seen that it commutes with $1 \times \R_{k'/k}(G')$, it would then follow that $\psi^{-1}(Z_G) \subset C \ltimes \R_{k'/k}(G')$ is central, which is what we want to show.

In order to prove (\ref{Z_R(C)=C}), we recall that $\mathscr{C} = \R_{k'/k}(D')$, where $D' = Z_{G'}(T')$ for some maximal $k'$-torus $T' \subset G'$. The Cartan subgroup $D'$ of $G'$ is its own centralizer in $G'$, since if $g' \in G'$ centralizes $D'$, then in particular it centralizes $T' \subset D'$, hence $g' \in D'$. Since $D'$ is smooth, \cite[Prop.\,A.5.15(1)]{cgp} implies that the centralizer in $\R_{k'/k}(G')$ of $\mathscr{C} = \R_{k'/k}(D')$ is $\R_{k'/k}(D') = \mathscr{C}$, as claimed.
\end{proof}

By Lemma \ref{centerssurj},
we may twist the second row of (\ref{diagram1}) by a cocycle $x \in Z^1(k, G/Z_G)$ to obtain a new central extension
\[
1 \longrightarrow B \longrightarrow F_x \longrightarrow G_x \longrightarrow 1.
\]
Note that $F_x$ is just an \'etale twist of $F$, not necessarily an inner form. As in the deduction of (\ref{tamproofeqn2}), Lemma \ref{tamagawasimplekernel} implies that
\[
\frac{\tau(G_x)}{\# \Ext^1(G_x, \Gm)} = \frac{\tau(F_x)}{\# \Ext^1(F_x, \Gm)}.
\]
Comparing this with (\ref{tamproofeqn2}), it suffices to show that $\# \Ext^1(G_x, \Gm) = \# \Ext^1(G, \Gm)$ and that 
\begin{equation}
\label{tau/exteqn}
\frac{\tau(F_x)}{\# \Ext^1(F_x, \Gm)} \stackrel{?}{=} \frac{\tau(F)}{\# \Ext^1(F, \Gm)}.
\end{equation}
(In fact, we have $\# \Ext^1(F_x, \Gm) = \# \Ext^1(F, \Gm)$ and $\tau(F_x) = \tau(F)$, but we will not need this.) The first equality follows from Proposition \ref{extinvtinnertwist}, so it only remains to prove (\ref{tau/exteqn}).

The ``conjugation'' action of $G$ on $F$ preserves the normal subgroup $\R_{k'/k}(G') \subset F$, hence also induces an action on the quotient $E$, and this action is trivial because $E$ is commutative. Thus, the exact sequence (\ref{exactseq2}) also gets twisted by $x$:
\[
1 \longrightarrow (\R_{k'/k}(G'))_x \longrightarrow F_x \longrightarrow E \longrightarrow 1.
\]
By Lemma \ref{standardetaletwist}, the group $(\R_{k'/k}(G'))_x$ is of the form $\R_{k''/k}(G'')$ for some finite reduced $k$-algebra $k''$ and some $k''$-group $G''$ whose fibers are all either simply connected or basic exotic pseudo-reductive. Therefore, by Lemmas \ref{tamagawasimplekernel} and \ref{R_{k'/k}(G')trivial}, we have
\[
\frac{\tau(F_x)}{\# \Ext^1(F_x, \Gm)} = \frac{\tau(E)}{\# \Ext^1(E, \Gm)}.
\]
Comparing this with (\ref{tau/exteqn3}) proves (\ref{tau/exteqn}), and therefore completes the proof of Theorem \ref{innerinvariance}. \hfill \qed

\section{An exact sequence for pseudo-reductive groups}
\label{sectionexactsequence}

In this section we prove Theorem \ref{pseudoredcomplexexact}. The commutative case is part of Tate duality \cite[Thm.\,1.2.8]{rostateduality} by Remark \ref{compatiblewithtateduality}, so assume that $G$ is a non-commutative pseudo-reductive $k$-group. Proposition \ref{adeliccohomologyprop} tells us that for a smooth connected group $H$ over a global field $k$, the natural map ${\rm{H}}^1(\A, H) \rightarrow \prod_v {\rm{H}}^1(k_v, H)$ induces a bijection ${\rm{H}}^1(\A, H) \rightarrow \coprod_v {\rm{H}}^1(k_v, H)$. The same holds for ${\rm{H}}^i$ for all $i > 0$ if $H$ is also commutative by \cite[Prop.\,5.2.2]{rostateduality}.

By Theorem \ref{predclassification}, we may assume that the pseudo-reductive $G$ is either generalized standard or $G = \R_{k'/k}(G')$ for a nonzero finite reduced $k$-algebra $k'$ and a $k'$-group $G'$ fiber $G'_i$ over each factor field $k'_i$ of $k'$ is basic non-reduced. In the latter case, we claim that ${\rm{H}}^1(\A, G) = 1$, so the desired exactness is clear. Indeed, we may work on each factor field of $k'$ and thereby assume that $k'$ is a field. By \cite[Ch.\,IV, \S 2.3, Cor.]{oesterle}, we have ${\rm{H}}^1(\A_k, \R_{k'/k}(G')) = {\rm{H}}^1(\A_{k'}, G')$, so we may assume that $G$ is basic non-reduced. Then \cite[Prop.\,9.9.4(1)]{cgp} implies that ${\rm{H}}^1(\A, G) = 1$.

So we may and do assume that $G$ is generalized standard and non-commutative. Then $G$ is built in the usual manner as follows. There is a nonzero finite reduced $k$-algebra $k'$ and a $k'$-group $G'$ with fibers that are each either connected semisimple and simply connected semisimple or basic exotic pseudo-reductive, a maximal $k'$-torus $T'$ which we are free to choose (see Remark \ref{choiceoftorus}), a commutative pseudo-reductive $k$-group $C$ acting on $\R_{k'/k}(G')$ whose action is trivial on $\mathscr{C} := \R_{k'/k}(Z_{G'}(T')) \subset \R_{k'/k}(G')$, and a map $\phi: \mathscr{C} \rightarrow C$ such that the action of $C$ on $\R_{k'/k}(G')$ is compatible with the conjugation action of $\mathscr{C}$ on $\R_{k'/k}(G')$ via $\phi$, and a central extension
\begin{equation}
\label{extension4}
1 \longrightarrow \mathscr{C} \longrightarrow C \ltimes \R_{k'/k}(G') \longrightarrow G \longrightarrow 1,
\end{equation}
where the first map is $(-\phi, j)$, with $j: \mathscr{C} \hookrightarrow \R_{k'/k}(G')$ the inclusion. 

Pick $\alpha \in {\rm{H}}^1(\A, G)$ such that $\alpha \mapsto 0 \in \Ext^1(G, \Gm)^*$. We need to show that $\alpha$ lifts to ${\rm{H}}^1(k, G)$. Having fixed $\alpha$, we now choose the maximal $k'$-torus $T'$ above as follows. Recall that we have a canonical map $f: G' \rightarrow \overline{G}'$ onto a $k'$-group with simply connected semisimple fibers. Indeed, we take $\overline{G}'$ to be the same as $G'$ on the simply connected fibers of $G'$ over Spec$(k')$, and we take it to be the canonical simply connected quotient on the basic exotic fibers; see Remark \ref{canonicalsurjection}. Then there is a bijection between the maximal $k'$-tori of $\overline{G}'$ and those of $G'$ defined by sending a maximal torus $\overline{T}' \subset \overline{G}'$ to the unique maximal $k'$-torus in $f^{-1}(\overline{T}')$ \cite[Cor.\,7.3.4]{cgp}. 

Let $S$ be the finite set of places $v$ of $k$ such that $1 \neq \alpha_v \in {\rm{H}}^1(k_v, G)$. By Proposition \ref{torustrivialSha^2}, there is a maximal $k'$-torus $\overline{T}' \subset \overline{G}'$ such that $\Sha^2(k', \overline{T}') = 0$ and ${\rm{H}}^2(k'_w, \overline{T}') = 0$ for all places $w$ of $v$ lying above a place of $S$. We take $T' \subset G'$ to be the maximal torus corresponding to $\overline{T}'$ under the above bijection. Then we claim that
\begin{equation}
\label{Sha^2=0eqn2}
\Sha^2(k, \mathscr{C}) = 0
\end{equation}
and
\begin{equation}
\label{H^2=0inS}
{\rm{H}}^2(k_v, \mathscr{C}) = 0 \mbox{ if } \alpha_v \neq 1.
\end{equation}

Indeed, we may proceed fiber by fiber over Spec$(k')$ and thereby temporarily assume that $k'$ is a field and $G'$ is either simply connected or basic exotic. We first note that ${\rm{H}}^2(k_v, \mathscr{C}) = {\rm{H}}^2(k_v, \R_{k'/k}(Z_{G'}(T'))) = \prod_{w \mid v} {\rm{H}}^2(k'_w, Z_{G'}(T'))$ since finite pushforward is an exact functor between categories of abelian \'etale sheaves; similarly, $\Sha^2(k, \mathscr{C}) = \Sha^2(k', Z_{G'}(T'))$. Now the simply connected case is immediate from our choice of $T'$, since $Z_{G'}(T') = T' = \overline{T}'$ on those fibers. In the basic exotic case, we invoke Lemma \ref{pointsofcartan}, whose hypotheses are preserved when passing from $k'$ to $k'_w$, since this is a separable extension. The claim then follows from the corresponding property for the torus $\overline{T}'$ arranged by design.

Because the extension (\ref{extension4}) is central, there is a connecting map $\gamma: {\rm{H}}^1(\A, G) \rightarrow {\rm{H}}^2(\A, \mathscr{C})$ having the usual properties. Then $\gamma(\alpha) = 0$. Indeed, for $v$ such that $\alpha_v = 1$ the image clearly vanishes, while for all other $v$ it vanishes by (\ref{H^2=0inS}). Therefore, $\alpha$ lifts to a class $\beta \in {\rm{H}}^1(\A, C \ltimes \R_{k'/k}(G'))$. We claim that we may choose $\beta$ so that $\beta \mapsto 0 \in \Ext^1(C \ltimes \R_{k'/k}(G'), \Gm)^*$. Assuming this, if Theorem \ref{pseudoredcomplexexact} holds for $C \ltimes \R_{k'/k}(G')$, then $\beta$ lifts to ${\rm{H}}^1(k, C \ltimes \R_{k'/k}(G'))$, hence $\alpha$ lifts to ${\rm{H}}^1(k, G)$. Modulo the claim, Theorem \ref{pseudoredcomplexexact} for $G$ therefore reduces to Theorem \ref{pseudoredcomplexexact} for $C \ltimes \R_{k'/k}(G')$.

In order to prove the claim, consider the following commutative diagram with exact rows
\begin{equation}
\label{diagram3}
\begin{tikzcd}
{\rm{H}}^1(\A, \mathscr{C}) \arrow{r}{j} \arrow[d, twoheadrightarrow, "g"] & {\rm{H}}^1(\A, C \ltimes \R_{k'/k}(G')) \arrow{d}{h} \arrow{r} & {\rm{H}}^1(\A, G) \arrow{d} \\
\Ext^1(\mathscr{C}, \Gm)^* \arrow{r}{j^*} & \Ext^1(C \ltimes \R_{k'/k}(G'), \Gm)^* \arrow{r} & \Ext^1(G, \Gm)^*
\end{tikzcd}
\end{equation}
where the exactness of the bottom row follows from \cite[Lemma 3.3]{rospic}. The first vertical arrow in this diagram is surjective. Indeed, by the compatibility of the map ${\rm{H}}^1(\A, \mathscr{C}) \rightarrow \Ext^1(\mathscr{C}, \Gm)^*$ with the map arising in Tate duality ${\rm{H}}^1(\A, \mathscr{C}) \rightarrow {\rm{H}}^1(k, \widehat{\mathscr{C}})^*$ (Remark \ref{compatiblewithtateduality}), it suffices to show that this latter map is surjective. By global Tate duality for $\mathscr{C}$ \cite[Thm.\,1.2.8]{rostateduality}, the cokernel of this map is isomorphic to $\Sha^2(\mathscr{C})$, which vanishes by (\ref{Sha^2=0eqn2}).

A simple diagram chase in diagram (\ref{diagram3}) using the fact that $\alpha \mapsto 0 \in \Ext^1(G, \Gm)^*$ now shows that there exists $\gamma \in {\rm{H}}^1(\A, \mathscr{C})$ such that $h(j(\gamma)) = h(\beta)$. Since the extension (\ref{extension4}) is central, the group ${\rm{H}}^1(\A, \mathscr{C})$ acts on the set ${\rm{H}}^1(\A, C \ltimes \R_{k'/k}(G'))$, and the fibers of the map $\sigma$ are the orbits of this action (\cite[Chap.\,I, \S 5.7, Prop.\,42]{serre} and Proposition \ref{adeliccohomologyprop}). Let us denote this action by $*$. Then we claim that $h(-\gamma * \beta) = 0$, which will prove our claim that we may choose $\beta$ lifting $\alpha$ so that $h(\beta) = 0$, and thereby reduce us to proving Theorem \ref{pseudoredcomplexexact} for the group $C \ltimes \R_{k'/k}(G')$. Since the map $h \circ j$ is a group homomorphism (because it equals $j^* \circ g$), this equality follows from the following lemma.

\begin{lemma}
\label{actionExt*}
Let $\epsilon \in {\rm{H}}^1(\A, \mathscr{C})$, $\zeta \in {\rm{H}}^1(\A, C \ltimes \R_{k'/k}(G'))$. Then $h(\epsilon * \zeta) = h(j(\epsilon)) + h(\zeta)$.
\end{lemma}

\begin{proof}
Let $E \in \Ext^1(C \ltimes \R_{k'/k}(G'), \Gm)$; this is a central extension
\[
1 \longrightarrow \Gm \longrightarrow E \longrightarrow C \ltimes \R_{k'/k}(G') \longrightarrow 1
\]
Let $\delta$ denote the connecting map ${\rm{H}}^1(C \ltimes \R_{k'/k}(G')) \rightarrow {\rm{H}}^2(\Gm)$. Then the pairing between $E$ and an element $\alpha \in {\rm{H}}^1(\A, C \ltimes \R_{k'/k}(G'))$ is given by adding up the elements $\delta(\alpha_v) \in {\rm{H}}^2(k_v, \Gm) \simeq \Q/\Z$ for all places $v$ of $k$. Thus, by definition of $h$, the assertion reduces to a local one. That is, we need to show that for $\epsilon \in {\rm{H}}^1(k_v, \mathscr{C})$, $\zeta \in {\rm{H}}^1(k_v, C \ltimes \R_{k'/k}(G'))$, $\delta(\epsilon * \zeta) = \delta(j(\epsilon)) + \delta(\zeta)$. 

Consider the following pullback diagram with exact rows
\[
\begin{tikzcd}
1 \arrow{r} & \Gm \arrow[d, equals] \arrow{r} & E' \arrow{r} \arrow[d, hookrightarrow] \arrow[dr, phantom, "\square"] & \mathscr{C} \arrow{r} \arrow[d, hookrightarrow, "j"] & 1 \\
1 \arrow{r} & \Gm \arrow{r} & E \arrow{r}{\pi} & C \ltimes \R_{k'/k}(G') \arrow{r} & 1
\end{tikzcd}
\]
where $E'$ is defined by the above pullback diagram. Let $\delta': {\rm{H}}^1(\mathscr{C}) \rightarrow {\rm{H}}^2(\Gm)$ denote the connecting map for the top sequence in the diagram. Then the equality we need to show is equivalent to $\delta(\epsilon * \zeta) = \delta'(\epsilon) + \delta(\zeta)$. If we choose $1$-cocycles $\zeta_{\sigma}, \epsilon_{\sigma}$ representing $\zeta$ and $\epsilon$ (where the notation means that $\epsilon_{\sigma} \in \mathscr{C}(k_s)$ is the image of $\sigma \in \Gal((k_v)_s/k_v)$ under the chosen cocycle representing $\epsilon$), then by definition $\epsilon * \zeta$ is represented by the $1$-cocycle $\epsilon_{\sigma}\zeta_{\sigma}$ (see \cite[Chap.\,I, \S 5.7]{serre}). Now $\delta(\zeta_{\sigma})$ is computed by lifting $\zeta_{\sigma}$ to a $1$-cochain valued in $E(k_s)$, and then taking its differential to get a $2$-cocycle valued in $\Gm$, and similarly for $\delta'$. The desired equality is therefore immediate  if we know that $E' = \pi^{-1}(\mathscr{C})$ is central in $E$ (by choosing a lift $\overline{\epsilon}_{\sigma} \in E'(k_s)$ of $\epsilon_{\sigma}$, a lift $\overline{\zeta}_{\sigma} \in E(k_s)$ of $\zeta_{\sigma}$, and then using the lift $\overline{\epsilon}_{\sigma} \overline{\zeta}_{\sigma} \in E(k_s)$ of $\epsilon_{\sigma} \zeta_{\sigma}$).

To prove this centrality, let $T \subset \mathscr{C}$ be the maximal torus. Then $\pi^{-1}(T) \subset E$ is central. Indeed, it is a normal torus (since $T \subset \mathscr{C}$ is normal (even central) in $C \ltimes \R_{k'/k}(G')$), and the conjugation action of $E$ on it induces a map $E \rightarrow {\rm{Aut}}_{\pi^{-1}(T)/k}$ which is constant because $E$ is connected and the automorphism scheme of a torus is \'etale. Thus, $\pi^{-1}(T)$ is central in $E$.

The group $E(\overline{k})$ is Zariski dense in $E_{\overline{k}}$ (because $E$ is geometrically reduced, even smooth), so it suffices to check that $E(\overline{k})$ centralizes $E'_{\overline{k}}$. So fix $e \in E(\overline{k})$ and consider the commutator map $\phi_e: E'_{\overline{k}} \rightarrow E'_{\overline{k}}$ given by $e' \mapsto ee'e^{-1}e'^{-1}$. We want to show that this map is trivial. Since $\mathscr{C} \subset C \ltimes \R_{k'/k}(G')$ is central, the image of $\phi_e$ actually lives in $\ker(E'_{\overline{k}} \rightarrow \mathscr{C}_{\overline{k}}) = (\Gm)_{\overline{k}}$. We claim that $\phi_e: E'_{\overline{k}} \rightarrow (\Gm)_{\overline{k}}$ is a homomorphism. Assuming this, if we let $U := \mathscr{C}/T$, a unipotent group, then since $\pi^{-1}(T) \subset E'$ is central in $E$, $\phi_e$ descends to a homomorphism $\phi_e: U_{\overline{k}} \rightarrow (\Gm)_{\overline{k}}$, which is trivial because unipotent groups admit no nontrivial characters over a field.

It remains to prove the claim. This is a formal calculation once one notes that for any $e' \in E'$ (more precisely, $e' \in E'(R)$ for some $\overline{k}$-algebra $R$), $\phi_e(e') = ee'^{-1}e^{-1}e'^{-1} \in \Gm$ is central in $E$. Then we compute that for $c, d \in E'$,
\[
\phi_e(cd) = e(cd)e^{-1}(cd)^{-1} = ece^{-1}(ede^{-1}d^{-1})c^{-1} = ece^{-1}c^{-1}(ede^{-1}d^{-1}) = \phi_e(c)\phi_e(d)
\]
where in the penultimate inequality, we have used the centrality of $ede^{-1}d^{-1}$ observed above.
\end{proof}

We are therefore reduced to proving Theorem \ref{pseudoredcomplexexact} for the group $C \ltimes \R_{k'/k}(G')$. This follows immediately from Theorem \ref{pseudoredcomplexexact} for the commutative group $C$ together with the following lemma.

\begin{lemma}
\label{H^1(CtimesR)=H^1(C)}
The maps ${\rm{H}}^1(k, C \ltimes \R_{k'/k}(G')) \rightarrow {\rm{H}}^1(k, C)$ and ${\rm{H}}^1(\A, C \ltimes \R_{k'/k}(G')) \rightarrow {\rm{H}}^1(\A, C)$ induced by the projection $C \ltimes \R_{k'/k}(G') \rightarrow C$ are bijections.
\end{lemma}

\begin{proof}
We first prove the lemma for ${\rm{H}}^1(k, \cdot)$. That the map is surjective follows from the existence of a section to the projection $C \ltimes \R_{k'/k}(G') \rightarrow C$, namely, the map $C \rightarrow C \ltimes \R_{k'/k}(G')$ given by $c \mapsto (c, 1)$. For the injectivity, consider the exact sequence
\[
1 \longrightarrow \R_{k'/k}(G') \longrightarrow C \ltimes \R_{k'/k}(G') \longrightarrow C \longrightarrow 1
\]
where the first map is the obvious one $r \mapsto (1, r)$ and the second map is projection. A standard twisting argument shows that in order to prove the desired injectivity, it suffices to show that for any $x \in {\rm{H}}^1(k, C \ltimes \R_{k'/k}(G'))$, after twisting the sequence by $x$, the resulting kernel $(\R_{k'/k}(G'))_x$ has vanishing ${\rm{H}}^1$. 

Any such kernel is itself of the form $\R_{k''/k}(G'')$ for some finite reduced $k$-algebra $k''$ and some $k''$-group $G''$ all of whose fibers are either simply connected semisimple or basic exotic pseudo-reductive, by Lemma \ref{standardetaletwist}. Thus, the desired vanishing boils down to the vanishing of the global ${\rm{H}}^1$ that we have already seen several times for such groups: we may work fiber by fiber over Spec$(k'')$, so we may assume that $k''$ is a field. Then \cite[Ch.\,IV, \S 2.3 Cor.]{oesterle} implies that ${\rm{H}}^1(k, \R_{k''/k}(G'')) = {\rm{H}}^1(k'', G'')$. If $G''$ is simply connected, then this ${\rm{H}}^1$ vanishes by \cite[Satz A]{harder1}. If $G''$ is basic exotic, then the canonical surjection $f: G'' \twoheadrightarrow \overline{G}''$ to a simply connected group (see Remark \ref{canonicalsurjection}) induces a bijection on ${\rm{H}}^1$'s, by \cite[Prop.\,7.3.3(1)]{cgp}, so the ${\rm{H}}^1$-vanishing follows from the vanishing in the simply connected case once again.

The proof for ${\rm{H}}^1(\A, \cdot)$ is exactly the same, but using Proposition \ref{adeliccohomologyprop} and the vanishing of the local cohomology (rather than the global cohomology) of simply connected groups \cite[Thm.\,4.7]{bruhattits}.
\end{proof}

\section{Strong approximation for quotients}
\label{approximationsection}

In this section we will prove Theorem \ref{weakapprox}. We maintain the notation from the statement of that result. The key point is to note that we have a natural homomorphism $$\phi: \Ext^1(G, \Gm) \rightarrow \Br(Y).$$ This map may be defined in multiple ways, but the following is probably the simplest. Suppose given $E \in \Ext^1(Y, \Gm)$. Then $E$ is represented by a central extension of algebraic $k$-groups
\[
1 \longrightarrow \Gm \longrightarrow E \longrightarrow G \longrightarrow 1.
\]
Because this extension is central, it yields a connecting map ${\rm{H}}^1(Y, G) \rightarrow {\rm{H}}^2(Y, \Gm) = \Br(Y)$. In particular, the $G$-torsor $H$ over $Y$ gets mapped to an element of $\Br(Y)$.

We note that $\phi$ is compatible with the pairing ${\rm{H}}^1(\A, G) \times \Ext^1(G, \Gm) \rightarrow \Q/\Z$ defining the second map in Theorem \ref{pseudoredcomplexexact}, in the sense that the following diagram commutes:
\begin{equation}
\label{commdiagweakapproxpf}
\begin{tikzcd}
Y(\A) \arrow{d}{\delta} \arrow[r, phantom, "\times"] & \Br(Y) \arrow{r} & \Q/\Z \arrow[d, equals] \\
{\rm{H}}^1(\A, G) \arrow[r, phantom, "\times"] & \Ext^1(G, \Gm) \arrow{u}{\phi} \arrow{r} & \Q/\Z
\end{tikzcd}
\end{equation}
Here $\delta$ denotes the connecting map in cohomology coming from the $G$-torsor $\pi: H \rightarrow Y$. (Concretely, given $y \in Y(\A)$, one defines $\delta(y)$ to be the (isomorphism class of the) $G$-torsor $\pi^{-1}(y)$ over $\A$.) That the diagram above commutes follows from the functoriality of the connecting map ${\rm{H}}^1(\cdot, G) \rightarrow {\rm{H}}^2(\cdot, \Gm)$ associated to an extension of $G$ by $\Gm$. Now we give the proof of Theorem \ref{weakapprox}.

\begin{proof}[Proof of Theorem $\ref{weakapprox}$]
Let $y \in Y(\A^S)^{\Br}$. We need to show that there are points in $Y(k)$ arbitrarily close to $y$. By definition of $Y(\A^S)^{\Br}$, there is an element $\widetilde{y} \in Y(\A)^{\Br}$ whose projection onto $Y(\A^S)$ is $y$. Consider the following commutative diagram:
\[
\begin{tikzcd}
& Y(k) \arrow{r}{\delta} \arrow{d} & {\rm{H}}^1(k, G) \arrow{r}{f} \arrow{d} & {\rm{H}}^1(k, H) \arrow{d} \\
H(\A) \arrow{r} & Y(\A) \arrow{r}{\delta} & {\rm{H}}^1(\A, G) \arrow{d} \arrow{r} & {\rm{H}}^1(\A, H) \\
&& \Ext^1(G, \Gm)^* &
\end{tikzcd}
\]
The rows are exact sequences of pointed sets. Over fields (such as $k$ or $k_v$), this is \cite[Ch.\,I, \S 5.4, Prop.\,36]{serre}. In fact, the same assertion holds over an arbitrary base scheme, with straightforward proofs. Alternatively, one may invoke Proposition \ref{adeliccohomologyprop} for $G$ in order to reduce the claim to the corresponding assertion over the fields $k_v$.

Since $\widetilde{y} \in Y(\A)^{\Br}$, the commutativity of (\ref{commdiagweakapproxpf}) implies that $\delta(\widetilde{y}) \in {\rm{H}}^1(\A, G)$ maps to $0 \in \Ext^1(G, \Gm)^*$. By Theorem \ref{pseudoredcomplexexact}, therefore, $\delta(\widetilde{y})$ lifts to an element $\alpha \in {\rm{H}}^1(k, G)$. Then $f(x) \in \Sha^1(k, H) = 1$, since the composition $Y(\A) \rightarrow {\rm{H}}^1(\A, G) \rightarrow {\rm{H}}^1(\A, H)$ is trivial. It follows that $x = \delta(y')$ for some $y' \in Y(k)$. Then $\delta(\widetilde{y}) = \delta(y'_{\A}) \in {\rm{H}}^1(\A, G)$. We claim that it follows that $\widetilde{y} = h\cdot y'$ for some $h \in H(\A)$.

Indeed, by definition $\delta(y'_{\A}) = \pi^{-1}(y'_{\A})$ as right $G$-torsors, and similarly for $\delta(\widetilde{y})$. Let $\phi: \pi^{-1}(y'_{\A}) \rightarrow \pi^{-1}(\widetilde{y})$ be a $G$-equivariant isomorphism. Consider the map $\psi: \pi^{-1}(y'_{\A}) \rightarrow H_{\A}$ defined by $\psi(x) := \phi(x)x^{-1}$. The $G$-equivariance of $\phi$ ensures that $\psi$ is $G$-invariant, i.e., $\psi(gx) = \psi(x)$ for all $x \in \pi^{-1}(y'_{\A})$ and $g \in G$. This implies that $\psi$ is constant, since $G$ acts transitively on $\pi^{-1}(y'_{\A})$. Therefore, there is some $h \in H(\A)$ such that $\psi(x) = h$ for all $x \in \pi^{-1}(y'_{\A})$. That is, $\phi(x) = h\cdot x$. Composing with $\pi$ then yields $\widetilde{y} = h\cdot y'$, as claimed.

In particular, projection to $\A^S$ yields $y = h\cdot y'$ for some $h \in H(\A^S)$. Since strong approximation holds for $H$ with respect to $S$ by assumption, there is some $h' \in H(k)$ such that $h'$ is as close as we desire to $h$. We may therefore make $h' \cdot y' \in Y(k)$ as close as we desire to $y \in Y(\A^S)^{\Br}$. The proof of Theorem \ref{weakapprox} is therefore complete.
\end{proof}

\end{document}